\documentclass[10pt,reqno]{amsart}
\makeatletter
\def\l@subsection{\@tocline{2}{0pt}{2.5pc}{5pc}{}}
\def\l@subsubsection{\@tocline{2}{0pt}{5pc}{7.5pc}{}}
\makeatother
\usepackage{lmodern}
\usepackage{amsbsy}
\usepackage{amsmath}
\usepackage{amsthm}
\usepackage{wasysym}
\usepackage{amssymb}
\usepackage{mathrsfs}
\usepackage{amsfonts}
\usepackage{newlfont}
\usepackage[sans]{dsfont}
\usepackage{dcolumn}
\usepackage{bm}
\usepackage{stmaryrd}
\usepackage{bbm}
\usepackage{relsize}
\usepackage{euscript}
\usepackage{eufrak}
\usepackage[margin=1.1in]{geometry}
\numberwithin{equation}{section}
\newtheorem{thm}{Theorem}[section]

\newtheorem{cor}[thm]{Corollary}
\newtheorem{lem}[thm]{Lemma}
\newtheorem{prop}[thm]{Proposition}
\newtheorem{defn}[thm]{Definition}
\newtheorem{rem}[thm]{Remark}
\newtheorem{exam}[thm]{Example}
\begin{document}
\allowdisplaybreaks{
\title[]{STOCHASTIC CAUCHY INITIAL-VALUE FORMULATION OF THE HEAT EQUATION FOR RANDOM FIELD INITIAL DATA:~SMOOTHING, HARNACK-TYPE BOUNDS AND p-MOMENTS}
\author{Steven D. Miller}
\email{stevendm@ed-alumni.net}
\address{Rytalix Analytics, Strathclyde, Scotland}
\maketitle
\begin{abstract}
The following stochastic Cauchy initial-value problem is studied for the parabolic heat equation on a domain $\mathbf{Q}\subset{\mathbf{R}}^{n}$ with random field initial data.
\begin{align}
{\square}\widehat{u(x,t)}&\equiv \left(\tfrac{\partial}{\partial t}-\bm{\Delta}_{x}\right)\widehat{u(x,t)}=0,~x\in\mathbf{Q},
t> 0\nonumber\\&
\widehat{u(x,0)}=\phi(x)+\bm{\mathscr{J}}(x),~x\in\mathbf{Q},t=0\nonumber
\end{align}
where $\phi(x)\in C^{\infty}({\mathbf{Q}})$ and $\mathscr{J}(x)$ is a classical Gaussian random scalar field with expectation  $\mathbb{E}\llbracket
\pmb{\mathscr{J}}(x)\rrbracket=0$ and with a regulated covariance $\mathbb{E}\llbracket \mathscr{J}(x)\otimes\mathscr{J}(y)\rrbracket=\zeta J(x,y;\ell)$. The correlation length is $\ell$ and $\mathbb{E}\llbracket \mathscr{J}(x)\otimes\mathscr{J}(x)\rrbracket=\zeta<\infty$. The randomly perturbed solution $\widehat{u(x,t)}$ can be represented as stochastic convolution integral. This leads to stochastic extensions and versions of some classical results for the heat equation; in particular, a Li-Yau differential Harnack inequality
\begin{align}
\mathbb{E}\left\llbracket\frac{|\nabla\widehat{u(x,t)}|^{2}|}{|\widehat{u(x,t)}|^{2}}
\right\rrbracket-\mathbb{E}\left\llbracket\frac{\tfrac{\partial}{\partial t}\widehat{u(x,t)}}{\widehat{u(x,t)}}
\right\rrbracket\le \frac{1}{2}n\frac{1}{t}\nonumber
\end{align}
and a parabolic Harnack inequality $
\mathbb{E}\left\llbracket|\widehat{u(y,t_{2})}|^{2}\right\rrbracket
\ge\mathbb{E}\left\llbracket|\widehat{u(x,t_{1})}|^{2}
\right\rrbracket\left|\frac{t_{1}}{t_{2}}\right|^{n}\exp({-\tfrac{\|x-y\|^{2}}{4|
t_{2}-t_{1}|}})$. Decay estimates and bounds for the volatility $\mathbb{E}\llbracket|\widehat{u(x,t)}|^{2}\rrbracket $ and p-moments $\mathbb{E}\llbracket|\widehat{u(x,t)}|^{p}\rrbracket $ are derived. Since
$\lim_{t\uparrow \infty}\mathbb{E}\llbracket|\widehat{u(x,t)}|^{p}\rrbracket=0 $, the Cauchy evolution of the randomly perturbed solution is stable since the heat equation smooths out or dissipates volatility induced by initial data randomness as $t\rightarrow\infty$.
\end{abstract}
\raggedbottom
\maketitle
\tableofcontents
\section{Introduction}
There have been long-standing and efficacious interactions between probability theory, stochastic analysis, geometry and PDE, perhaps best exemplified by random motion or Brownian motion and stochastic analysis on manifolds $\mathbf{[1-7]}$. The well-known Feynman-Kac formula $\mathbf{[1-3,7]}$ connects the heat equation with Brownian motion, and the heat equation and heat kernel in turn play crucial roles in geometric analysis and the Ricci flow $\mathbf{[8-11]}$. Other examples are the probabilistic proofs of the Gauss-Bonnet-Chern Theorem and the Atiyah-Singer Index Theorem $\mathbf{[1,12]}$ and Driver-type integration by parts formulas $\mathbf{[13]}$. Within the theory of stochastic PDE, the parabolic heat equation has been studied quite extensively from the perspective of a stochastic or noisy heat source term $\mathbf{[14-18]}$. However, this paper is concerned with random Gaussian perturbations of the initial Cauchy data for the classical parabolic heat equation defined on an Euclidean domain. The initial data is then essentially a classical Gaussian random (scalar) field and there is a stochastic convolution integral solution. This leads to stochastic extensions and versions of some well known and classical results for the heat equation, including Harnack-type estimates and bounds. Moments and volatility can also be estimated and the heat equation is expected to smooth out the initial data randomness at $t=0$, as t increases.

Classical random fields correspond naturally to structures, and properties of systems, that are varying randomly in time and/or space. They have found many useful applications in mathematics and applied science: in the statistical theory or turbulence, in geoscience, medical science, engineering, imaging, computer graphics, statistical mechanics and statistics, biology and cosmology $\mathbf{[19-36]}$. Coupling random fields or noise to ODEs or PDEs is also a useful methodology in studying turbulence, chaos, random systems, pattern formation etc. $\mathbf{[19-35,37,38,39]}$ The study of stochastic partial differential equations (SPDEs),arising from the coupling of random fields/noises to PDEs is also a rapidly growing area $\mathbf{[15,37,38,39]}$. Such SPDES can model the propagation of heat, diffusions or waves in random medias or randomly fluctuating medias. Many dynamical systems are affected or influenced by colored noise which is regulated.$\mathbf{[39-43]}$.

An important class of random fields are the Gaussian random scalar fields (GRSF) which are characterized only by the first and second moments. The GRSF can also be isotropic, homogenous and stationary. The details are made more precise in Appendix A, but the advantages of GRSF are:
\begin{enumerate}
\item GRSF have convenient mathematical properties which generally simplify calculations;indeed, many results can only be evaluated using Gaussian fields.
\item A GRSF can be classified purely by its first and second moments, and all high-order moments and cumulants can be ignored. If $\mathscr{J}(x)$ is a time-independent GRSF existing for all
    $x\in{\mathbf{R}}^{n}$ or $x\in\mathbf{Q}\subset {\mathbf{R}}^{n}$ then
    $\mathbb{E}\llbracket\mathscr{J}(x)\rrbracket=0$ and
    \begin{align}
    \mathbb{E}\llbracket\mathscr{J}(x)\otimes\mathscr{J}(y)\rrbracket=\zeta J(x,y;\ell)
    \end{align}
    with a correlation length $\ell$. It is regulated if $\mathbb{E}\llbracket\mathscr{J}(x)\otimes\mathscr{J}(x)\rrbracket=\zeta<\infty$
    For a white-in-space Gaussian noise or random field $\mathbb{E}\llbracket\mathscr{J}(x)\otimes\mathscr{J}(y)\rrbracket
    =\zeta\delta^{n}(x-y)$ and is unregulated. This paper utilises only regulated GRSFs.
\item Gaussian fields accurately describe many natural stochastic processes including Brownian motion.
\item A large superposition of non-Gaussian fields can approach a Gaussian field.
\end{enumerate}
The following generic stochastic scenarios can arise for linear, quasi-linear and nonlinear PDEs subject to randomness or noise.
\begin{prop}
Let $\mathbf{Q}\subset {\mathbf{R}}^{n}$ be a domain or (compact) set with boundary $\partial\mathbf{Q}$ and let $\pmb{\mathscr{L}}$ be a linear differential operator (usually elliptic)acting on a function or field $u(x,t)$ defined on $\mathbf{Q}\times [0,T]$ or $\mathbf{Q}\times\mathbf{R}^{+}$. The generic (deterministic)Cauchy initial value problem (CIVP) and boundary value problem (BVP) is then
\begin{align}
&\pmb{\mathscr{L}}u(x,t)=f(x,t),~(x,t)\in\mathbf{Q}\times[0,T]\\&
u(x,0)=\phi(x),~x\in\mathbf{Q},t=0\\&
\pmb{\mathscr{B}}u(x,t)=\Xi(x),~x\in\partial\mathbf{Q}
\end{align}
where $\pmb{\mathscr{B}}$ is a 'boundary operator', $f(x,t)$ is a source term, $\Xi(x)$ is some function and $\phi(x)$ is the initial Cauchy data with $\phi\in C^{\infty}(\mathbf{Q})$ and at least $C^{2}(\mathbf{Q})$. Similarly, one can have a quasi-linear operator $\pmb{\mathscr{Q}}$ and a nonlinear operator $\pmb{\mathscr{N}}$. Introducing randomness, fluctuations or noise leads to the following possibilities for generic random problems:
\begin{enumerate}
\item The random or stochastic CIVP with initial data $\phi(x)$ having additive randomness, random perturbations or fluctuations such that
    \begin{align}
    \widehat{\phi(x)}=\phi(x)+\mathscr{J}(x)\nonumber
    \end{align}
    where $\mathscr{J}(x)$ is a (Gaussian) random scalar field (GRSF).(Appendix A.) For multiplicative random perturbations
    \begin{align}
    \widehat{\phi(x)}=\phi(x)\otimes\mathscr{J}(x)\nonumber
    \end{align}
\item The random boundary value problem with $\Xi(x)$ having randomness, random perturbations or fluctuations such that $\widehat{\Xi(x)}=\Xi(x)+\mathscr{J}(x)$.
\item The random source problem such that the source term $f(x,t)$ is subject to randomness, random perturbations or noise such that $\widehat{f(x,t)}=f(x,t)+\mathscr{J}(x,t)$, where $\mathscr{J}(x,t)$ is a random field or noise in space and time; for example, a white noise or Wiener process. This is the usual approach to studying stochastic PDEs.
\item The random operator problem:$\pmb{\mathscr{L}}$ or $\pmb{\mathscr{B}}$ is random.
\item The random or stochastic geometry problem:the domain $\mathbf{Q}$ is random.
\item Some combination of these conditions.
\end{enumerate}
\end{prop}
The stochastic heat equation is usually studied as a stochastic problem of the type (3), with either additive or multiplicative Gaussian noise $\mathbf{[15]}$, usually with white noise as a random source term. This is the prototype for a basic SPDE $\mathbf{[14-18]}$.
\begin{align}
&\left(\frac{\partial}{\partial t}-{\Delta}\right)\widehat{u(x,t)}=\mathscr{W}(x,t),~~
(x\in {\mathbf{R}}^{n}, t>0)\\&
\widehat{u(x,0)}=\phi(x),~~(x\in{\mathbf{R}}^{n},t=0)
\end{align}
where $\phi\in C^{2}(\mathbf{R}^{n}$ and $\mathscr{W}(x)$ is a space-time white noise on ${\mathbf{R}}^{n}\times [0,\infty)$. This is characterised as a (distribution-valued) centred Gaussian process with $\mathbb{E}\llbracket\mathscr{W}(x,t)\rrbracket$ and covariance
\begin{equation}
\mathbb{E}\llbracket\mathscr{W}(x,t)\otimes\mathscr{W}(y,s)\rrbracket
=\delta^{n}(x-y)\delta(t-s)
\end{equation}
The SPDE has the 'mild solution', which is a stochastic convolution integral of the form
\begin{align}
\widehat{u(x,t)}&=\int_{0}^{t}\int_{\Omega}{(4 \pi |t-s|)^{-n/2}}\exp\left(-\frac{\|x-y\|^{2}}{4|t-s|}\right)\otimes\mathscr{W}(y,s)dsd\mu_{n}(y)
\nonumber\\&+\int_{\Omega}{(4 \pi t)^{-n/2}}\exp\left(-\frac{\|x-y\|^{2}}{4|t-s|}\right)\phi(y)d\mu_{n}(y)\nonumber\\&
=\int_{0}^{t}\int_{\Omega}h(x-y,t-s)\otimes\mathscr{W}(y,s)ds d\mu_{n}(y) +
\int_{\Omega}h(x-y,t)\phi(y)d\mu_{n}(y)
\end{align}
where $h(x-y,t)={(4\pi t)^{-n/2}}\exp\left(-\tfrac{|x-y|^{2}}{4t}\right)$ is the heat kernel. If $\phi(x)=0$ then
\begin{align}
\widehat{u(x,t)}&=\int_{0}^{t}\int_{\Omega}{(4 \pi |t-s|)^{-n/2}}\exp\left(-\frac{|x-y|^{2}}{4|t-s|}\right)\otimes\mathscr{W}(y,s)dsd\mu_{n}(y)\nonumber\\&
=\int_{0}^{t}\int_{\Omega}h(x-y,t-s)\otimes\mathscr{W}(y,s)dsd\mu_{n}(y)
\end{align}
This is also a centred Gaussian process, but its covariance function is somewhat more complicated. The space-time regularity properties of (1.8) are discussed in $\mathbf{[15]}$. Solutions to ordinary stochastic differential equations are in general $\gamma$-H¨older continuous (in time) for every $\gamma < 1/2$ but not for $\gamma$ =1/2, in dimension n = 1, u as given by (1.8) is only ‘almost’ 1/4-Holder continuous in time and ‘almost’ 1/2-Holder continuous in space. The covariance can be estimated as
\begin{equation}
\mathbb{E}\big\llbracket \widehat{u(x,t)}\otimes \widehat{u(x^{\prime},t^{\prime})}\big\rrbracket=\int_{0}^{t}\int_{{\mathbf{R}}^{n}}
h(x-y,t-s)h(x^{\prime}-y,t^{\prime}-s)dsd\mu_{n}(y)
\end{equation}
so that $\mathbb{E}\big\llbracket|\widehat{u(x,t)}|^{2}
\big\rrbracket=\lim_{x\rightarrow x^{\prime},t\rightarrow t^{\prime}}\mathbb{E}\big\llbracket \widehat{u}(x,t)\otimes \widehat{u}(x^{\prime},t^{\prime})\big\rrbracket$. The variance or volatility is then
\begin{align}
\mathbb{E}\big\llbracket |\widehat{u(x,t)}|^{2}
\big\rrbracket&=\int_{0}^{t}\int_{\mathbf{R}^{n}}|h(x-y,t)|^{2}dsd\mu_{n}(y)\equiv
\int_{0}^{t}\|(h(x-\bullet,t)\|_{L_{2}({\mathbf{R}}^{n}})^{2}ds\nonumber\\&
=\int_{0}^{t}{8\pi|t-s|^{-n/2}}\int_{{\mathbf{R}}^{n}}
\exp\left(-\frac{|x-y|^{2}}{2|t-s|}\right){2\pi|t-s|^{-n/2}}d\mu_{n}(y)ds\nonumber\\&
=\int_{0}^{t}8\pi |t-s|^{-n/2}ds
\end{align}
since $\int_{{\mathbf{R}}^{n}}h(x-y,t)d\mu_{n}(y)=1$. The integral (1.11) converges only if $n=1$ and diverges logarithmically for $n=2$. For $n=1$ we have the estimate
\begin{equation}
\mathbb{E}\big\llbracket |\widehat{u(s,0)}-\widehat{u(t,0)}|^{2}\big\rrbracket \sim \int_{0}^{t} |t-s^{-1/2}\sim |t-s|^{1/2}
\end{equation}
which can be compared to Standard Brownian motion for which $\mathbb{E}\big\llbracket|\mathscr{B}(s)-\mathscr{B}(t)|^{2}\big\rrbracket=|t-s|$. So the process $\widehat{u(x,t)} $, for any fixed $x$, is almost certainly $\gamma$-Holder continuous for $\gamma<1/4$ but not for $\gamma=1/4$.
\begin{rem}
If $u(x,t)=\int_{\mathbf{R}^{n}}h(x-y,t)d\mu_{n}(y)$ is a solution of the homogenous deterministic heat equation $\frac{\partial}{\partial t}\psi={\Delta}\psi$ then
$\lim_{t\uparrow\infty}u(x,t)=0$ and $\lim_{t\uparrow\infty}|u(x,t)|^{2}=0$. However, the solution for the stochastic heat equation grows without bound for all $n\ge 1$ since
\begin{align}
\lim_{t\uparrow\infty}\mathbb{E}\big\llbracket |\widehat{u(x,t)}|^{2}
\big\rrbracket=\lim_{t\uparrow\infty}\int_{0}^{t}8\pi |t-s|^{-1/2}ds=\infty
\end{align}
for any fixed $x$. This is because the stochastic forcing term pumps energy into the system faster than the deterministic evolution can dissipate it $\mathbf{[15]}$.
\end{rem}
Instead of a random source or forcing term, one can consider a deterministic linear PDE like the heat equation and consider initial Cauchy data which is random or randomly perturbed; that is, the stochastic problem (1) of Proposition 1.1. Kamp de Feriet $\mathbf{[35]}$ first studied the random heat equation for random Cauchy data. He considered the heat equation on an infinite rod $(-\infty,\infty)$ with a random field or function at $x$ at $t=0$. The initial-value problem is
\begin{align}
&\left(\frac{\partial}{\partial t}-\Delta\right)\widehat{u(x,t)},~~(x\in(-\infty,\infty),t>0)\\&
\widehat{u(x,0)}=\mathscr{F}(x),~~(x\in(-\infty,\infty),t=0)
\end{align}
where $\mathscr{F}(x)$ is a generic 'random function'. The solution is the stochastic convolution integral
\begin{equation}
\widehat{u(x,t)}=\int_{0}^{\infty}\frac{1}{\sqrt{(4\pi t)}}\exp\left(-\frac{|x-y|^{2}}{4t}\right)\otimes
\mathscr{F}(y)dy
\end{equation}
In this paper, the following similar stochastic Cauchy initial-value problem (SCIVP) is considered in much greater detail.
\begin{prop}
Let $\mathbf{Q}\subset{\mathbf{R}}^{n}$ be a bounded domain with boundary $\partial\Omega$. Let $\mathscr{J}(x)$ be a Gaussian random scalar field (GRSF) defined with respect to a probability space $(\Omega,\bm{\mathfrak{F}},\bm{\mathrm{I\!P}})$ and which exists for all $x\in\mathbf{R}^{n}$.(See Appendix A). The expectation is
$\mathbb{E}\llbracket\bullet\rrbracket=\int_{\Omega}\bullet d\bm{\mathrm{I\!P}}$ such that $\mathbb{E}\llbracket\mathscr{J}(x)\rrbracket=0$. The covariance is regulated in that $\mathbb{E}\llbracket\mathscr{J}(x)\otimes\mathscr{J}(y)\rrbracket=\zeta J(x,y;\ell)$, where $\ell $ is a correlation length and
$\mathbb{E}\llbracket\mathscr{J}(x)\otimes\mathscr{J}(y)\rrbracket=\mu<\infty$. The SCIVP is an initial-value problem for the linear parabolic heat equation with random initial Cauchy data-that is, the initial data itself is a Gaussian random scalar field such that
\begin{align}
&\square u(x,t)=\left(\frac{\partial}{\partial t}-{\Delta}\right)u(x,t),~~(x\in\mathbf{Q},t>0)\\&
\widehat{u(x,0)}=\phi(x)+\mathscr{J}(x),~~(x\in\mathbf{Q},t=0)
\end{align}
and possibly $\widehat{u(x,t)}=0$ on $\partial{\mathbf{Q}}$. Here $\phi\in C^{2}(\mathbf{Q})$. Then the following can be studied:
\begin{enumerate}
\item The general solution $\widehat{u(x,t)}$, which can be expressed as a stochastic convolution integral over the heat kernel. And its average or expectation $\mathbb{E}\llbracket \widehat{u(x,t)}\rrbracket$.
\item Stochastic extensions of basic and well-known estimates and results for the heat equation.
\item Estimates for the decay or growth of the p-moments $\mathbb{E}\llbracket|\widehat{u(x,t)}|^{p}\rrbracket$
and volatility $\mathbb{E}\llbracket|\widehat{u(x,t)}|^{2}\rrbracket$ for  $t\rightarrow\infty$.
\item The solution is stable and decays if the initial volatility is smoothed out so that
\begin{align}
\lim_{t\uparrow \infty}\mathbb{E}\llbracket|\widehat{u(x,t)}|^{p}\rrbracket=0,~~~
\lim_{t\uparrow \infty}\mathbb{E}\llbracket|\widehat{u(x,t)}|^{2}\rrbracket=0
\end{align}
or
\begin{align}
\lim_{t\uparrow \infty}\mathbb{E}\llbracket|\widehat{u(x,t)}|^{p}\le \lambda,~~~
\lim_{t\uparrow \infty}\mathbb{E}\llbracket|\widehat{u(x,t)}|^{2}\le \lambda
\end{align}
for some $\lambda>0$. This is equivalent to
\begin{align}
\lim_{t\uparrow \infty}{\bm{\mathrm{I\!P}}}[|\widehat{u(x,t)}|^{p}=0]=1,~~~
\lim_{t\uparrow \infty}{\bm{\mathrm{I\!P}}}[|\widehat{u(x,t)}|^{p}>0]=0
\end{align}
and
\begin{align}
\lim_{t\uparrow \infty}{\bm{\mathrm{I\!P}}}[|\widehat{u(x,t)}|^{p}\le \lambda]=1,~~~
\lim_{t\uparrow \infty}{\bm{\mathrm{I\!P}}}[|\widehat{u(x,t)}|^{p}> \lambda]=0
\end{align}
\item The solution is unstable and grows without bound if
\begin{align}
\lim_{t\uparrow \infty}\mathbb{E}\llbracket|\widehat{u(x,t)}|^{p}=\infty,~~~
\lim_{t\uparrow \infty}\mathbb{E}\llbracket|\widehat{u(x,t)}|^{2}=\infty
\end{align}
or
\begin{align}
\lim_{t\uparrow \infty}{\bm{\mathrm{I\!P}}}[|\widehat{u(x,t)}|^{p}=\infty]=1,~~~
\lim_{t\uparrow \infty}{\bm{\mathrm{I\!P}}}[|\widehat{u(x,t)}|^{p}< \infty]=0
\end{align}
\item The solution blows up if for some finite $t_{B}\in[0,\infty)$
\begin{align}
\lim_{t\uparrow t_{B}}\mathbb{E}\llbracket|\widehat{u(x,t)}|^{p}=\infty,~~
\lim_{t\uparrow t_{B}}\mathbb{E}\llbracket|\widehat{u(x,t)}|^{2}=\infty
\end{align}
or
\begin{align}
\lim_{t\uparrow t_{B}}{\bm{\mathrm{I\!P}}}[|\widehat{u(x,t)}|^{p}=\infty]=1,~~~
\lim_{t\uparrow t_{B}}{\bm{\mathrm{I\!P}}}[|\widehat{u(x,t)}|^{p}< \infty]=0
\end{align}
\item The inhomogenous problem with a source function can also be studied such that
\begin{align}
&\square\widehat{u(x,t)}=\left(\frac{\partial}{\partial t}-\Delta\right)\widehat{u(x,t)}=f(x,t),~~(x\in\mathbf{Q},t>0)\\&
\widehat{u(x,0)}=\phi(x)+\mathscr{J}(x),~~(x\in\mathbf{Q},t=0)
\end{align}
where $f\in C^{2}(\mathbf{Q}\times [0,\infty))$, and only the initial Cauchy data is randomly perturbed or is a GRSF.
\item Finally, the problem for multiplicative random perturbations of initial Cauchy data can be studied such that
\begin{align}
&\square\widehat{u(x,t)}=\left(\frac{\partial}{\partial t}-{\Delta}\right)\widehat{u(x,t)},~~(x\in\mathbf{Q},t>0)\\&
\widehat{u(x,0)}=\phi(x)+\mathscr{J}(x),~~(x\in\mathbf{Q},t=0)
\end{align}
\end{enumerate}
These various estimates are also evaluated explicitly when $\mathbf{Q}={\mathbf{L}}\subset\mathbf{Q}^{+}$, a segment on the real line; $\mathbf{Q}=\mathbf{B}_{R}(0)\subset\mathbf{R}^{3}$, and Euclidean ball of radius R and centre at the origin; and $\mathbf{Q}={\mathbf{S}}^{1}$ a circle or ring.
\begin{exam}(\underline{Initial~temperature~profile~induced by a~'noisy'~ultrafast laser~pulse})
For a more practical potential application one could consider the heat equation in a finite one-dimensional medium $\mathbf{Q}=[0,L]$ or semi-infinite slab geometry $\mathbf{Q}=[0,\infty)$ along $z$, when the initial data is due to an ultrashort laser 'impulse from $t=0 $ to $t=t_{*}$ with $|t_{*}|\ll 0$. For an ultrashort flat or step function pulse $t_{*}>0$ but $t_{*}\sim 0$. If the medium has an absorption coefficient $\alpha$ then a beam with radiant flux or intensity I decays exponentially within the medium as Beer's Law so that $I(z)=I(0,t)\exp(-\alpha z)\equiv I\exp(-\alpha z)$. Absorbed laser energy is then (instantaneously) converted to heat from $t=0$ to $t=t_{*}$ during this impulse, but the heat equation then has no source term for $t>t_{*}$. The initial temperature profile through the slab then has the same Beer Law exponential decay so that $u(z,0)\sim I(z)$ or $u(z,0)=\beta\exp(-\alpha z)$. The Cauchy initial value problem (CIVP) is then the homogeneous problem
\begin{align}
&\left(\frac{\partial}{\partial t}-\frac{\partial^{2}}{\partial z^{2}}\right)u(z,t)=0, z\in \mathbf{Q},t>t_{*}\\&
u(z,0)=\phi(0)=\beta\exp(-\alpha z),~~z\in\mathbf{Q},t\in[0,t_{*}]
\end{align}
This is an ultrashort but flat or 'step function' profile so that $I=I(0,t)=const.$
for $t\le t_{*}$ and $I=I(0,t)=0 $ for $t>t_{*}$.

If the beam flux or intensity is subject to some intrinsic classical noise or (short-scale)
spatial random fluctuations along the beam axis then in vacuum
\begin{equation}
\widehat{I(z)}=I_{o}+b\mathscr{J}(z)\nonumber
\end{equation}
where $\mathscr{J}(z)$ is a GRSF and $b$ is a constant. The average intensity is then
$\mathbb{E}\llbracket \widehat{I(x)}\rrbracket=I_{o}$ and the correlations in intensity between any pair of points $(z,\overline{z})$ on the beam axis is
\begin{equation}
\mathbb{E}\llbracket \mathscr{J}(z)\otimes\mathscr{J}(\overline{z})\rrbracket=
|I_{o}|^{2}+b^{2}\zeta J(z,\overline{z};\ell)
\end{equation}
and $\ell$ is a very short correlation length. Within the matter the intensity fluctuations are described by
\begin{equation}
 \widehat{I(z)}=(I_{o}+b\mathscr{J}(z))\exp(-\alpha z)\nonumber
 \end{equation}
Fluctautions in beam intensity will induced fluctuations in the initial temperature profile.
This scenario could then be described by a stochastic CIVP for a perturbed temperature $\widehat{u(x,t)}$
\begin{align}
&\left(\frac{\partial}{\partial t}-\frac{\partial^{2}}{\partial z^{2}}\right)\widehat{u(z,t)}=0, z\in\mathbf{Q},t>t_{*}\\&~or~
\widehat{u(z,0})=\widehat{\phi(0)}=(\beta+b\mathscr{J}(x))\exp(-\alpha z),~~z\in\mathbf{Q},t\in[0,t_{*}]
\end{align}
where $\mathbb{E}\llbracket\mathscr{J}(z)\rrbracket=0$ and
$\mathbb{E}\llbracket\mathscr{J}(z)\otimes\mathscr{J}(z^{\prime})\rrbracket=\zeta J(z,z^{\prime};\ell)$. The random field $\mathscr{J}(z)$ could be a colored noise, which
occurs in laser physics and dynamical systems $\mathbf{[39-43]}$.
\end{exam}
\end{prop}
\section{Heat kernal and the Cauchy initial value problem--classical results and estimates}
Some basic definitions and theorems are first given for the heat kernel $\mathbf{[9,44]}$..
\begin{defn}
The heat kernel can exist on any manifold ${\mathbf{M}}$ and plays a central role in geometric analysis $\mathbf{[9]}$. On Euclidean space ${\mathbf{M}}={\mathbf{R}}^{n}$ and the metric is
\begin{align}
ds^{2}=\sum_{i=1}^{n}(dx^{j})^{2}=|d\bm{x}|^{2}
\end{align}
The fundamental solution or heat kernel is then $\mathbf{[9,44]}$.
\begin{equation}
h(x,t)={(4\pi t)^{-n/2}}\exp\left(-\frac{\|x-y\|^{2}}{4t}\right),~(x\in {\mathbf{R}}^{n},t>0)
\end{equation}
with $h(x,t)=0$ for $t<0$. Then $\square u(x,t)=\left(\tfrac{\partial}{\partial t}-{\Delta}\right)h(x,t)=0$. At the origin $h(0,0)$ is singular and normalized as $\int_{\mathbf{R}^{3}}h(x,t)d^{3}x=1$. As $t\rightarrow 0$, $h(x,t)$ approximates the delta function on ${\mathbf{R}}^{n}$. Physically, the heat kernel describes temperature distribution in space and time arising from a point source at the origin. $h(x,t)$ is also infinitely differentiable or smooth and has uniformly bounded derivatives of all orders. Specifically, the heat kernel has the following well-known properties:
\begin{enumerate}
\item $h(x,t)$ is singular at $(0,0)$.
\item For all $\lambda>0$, $\alpha^{n}h(x,t)(\alpha x,\alpha^{2}t)=h(x,t)$.
\item $h(\bullet,t)\in L^{p}({\mathbf{R}}^{n})$ for all $p\in[1,\infty)$, and for all $t>0$ and $x\in {\mathbf{R}}^{n}$.
    \begin{align}
    &\|h(x-\bullet,t)\|_{L_{1}}=\int_{\mathbf{R}^{n}}h(x-y,t)d\mu_{n}(y)=1\\&
    \|h(x-\bullet,t)\|_{L_{p}}=\left(\int_{\mathbf{R}^{n}}|
    h(x-y,t)|^{p}d\mu_{n}(y)\right)^{1/p}=
    \frac{1}{p^{n/2p} (4\pi t)^{{\frac{n}{2}(1-\frac{1}{p})}}}
    \end{align}
\item The norm decays as $\lim_{t\uparrow\infty}\|h(\bullet,t)\|_{L_{p}}=0$.
\item $\square u(x,t)=\tfrac{\partial}{\partial t}h(x,t)-{\Delta}h(x,t)=0$.
\item Smoothness. $h(x,t)\in C^{\infty}(\bm{\mathbf{R}}^{n+1}\symbol{92}(0,0))$ and all derivatives to all orders are uniformly bounded.
\item Satisfies the Cauchy IVP
\begin{align}
\square h(x,t)=\left(\frac{\partial}{\partial t}-{\Delta}\right)h(x,t)=0
\end{align}
with $h(x,0)=\delta^{n}(x)$.
\item Obeys Varadhan's large deviation formula such that as $t\rightarrow 0$
\begin{align}
\lim_{t\uparrow 0}-4t\log h(x-y,t)=|d(x,y)|^{2}
\end{align}
on a Riemannian manifold with distance function $d(x,y)$ which is
$d(x,y)=\|x-y\|$ on ${\mathbf{R}}^{n}$.
\end{enumerate}
\end{defn}
In Euclidean space ${\mathbf{R}}^{n}$, the heat kernel satisfies a fundamental double-sided Gaussian bound $\mathbf{[45]}$.
\begin{thm}
If $h(x-y;t)$ is the heat kernel for all $(x,y)\in {\mathbf{R}}^{n},t>0$ then there exists constants $\Lambda_{1},\Lambda_{2}>0$ such that $h(x-y;t)$ satisfies the double-sided Gaussian bound.
\begin{align}
\Lambda_{1}t^{-n/2}\exp\left(-\frac{|x-y|^{2}}{\Lambda_{1}t}\right)~\le~
h(x-y;t)~\le~\Lambda_{2}t^{-n/2}\exp\left(-\frac{|x-y|^{2}}{\Lambda_{2}t}\right)
\end{align}
It follows that for $p=2$ and $p\ge 2$
\begin{align}
&\Lambda_{1}^{2}t^{-n}\exp\left(-\frac{2|x-y|^{2}}{\Lambda_{1}t}\right)~\le~|h(x-y;t)|^{2}~
\le~\Lambda_{2}^{2}t^{-n}\exp\left(-2\frac{|x-y|^{2}}{\Lambda_{2}t}\right)\\&
\Lambda_{1}^{p}t^{-pn/2}\exp\left(-\frac{2p|x-y|^{2}}{\Lambda_{1}t}\right)~\le~|h(x-y;t)|^{p}~
\le~\Lambda_{2}^{p}t^{-pn/2}\exp\left(-2p\frac{|x-y|^{2}}{\Lambda_{2}t}\right)
\end{align}
The double-sided Gaussian bound will also hold for all $(x,y)$ within any bounded domain $\mathbf{Q}\subset{\mathbf{R}}^{n}$.
\end{thm}
\begin{cor}
The gradient of the heat kernel satisfies the 2-sided (Faber-Krahn) bound $\mathbf{[46,47]}$.
\begin{align}
&\left|\Lambda_{1}t^{-n/2}\left(-\frac{2|x-y|}{\Lambda_{1}t}\right)\exp\left(-\frac{|x-y|^{2}}{\Lambda_{1}t}\right)
\right|\le|{\nabla}h(x-y,t)|\nonumber\\&\le \left|\Lambda_{1}t^{-n/2}\left(-\frac{2|x-y|}{\Lambda_{2}t}\right)\exp\left(-\frac{|x-y|^{2}}{\Lambda_{2}t}\right)
\right|
\end{align}
\end{cor}
On general manifolds, the heat kernel is relevant to many scenarios within differential geometry and geometric analysis and has been studied extensively $\mathbf{[8,9}$ and references therein$\mathbf{]}$. The key question on general complete non-compact manifolds ${\mathbf{M}}$ is: what geometric properties are sufficient to describe the long-time behaviour or decay of the heat kernel? Much is known about upper bounds. The works of Nash, Aronson, Varopoulos , Carlen, Kusuoka, Stroock  and Davies, have elucidated an understanding that uniform upper bounds of the heat kernel are closely related to isoperimetric-type inequalities, including the Sobolev, Nash and the logarithmic Sobolev inequalities $\mathbf{[45-51]}$.

The heat kernel will satisfy the following estimate
\begin{align}
h(x-t,t)\le \psi(t)
\end{align}
where $\psi:{\mathbf{R}}^{+}\rightarrow {\mathbf{R}}^{+}$ is a smooth monotonically decreasing function. If $\psi(t)=Ct^{-n/2}$ then $h(x-y;t)\le Ct^{-n/2}$. Carlen and Strook prove $\mathbf{[50]}$ prove that the bound (2.11) is equivalent to the Nash inequality
\begin{equation}
\left(\int_{\mathbf{R}^{n}} |\psi|^{2}\right)^{1+\frac{2}{n}}
\le C\int_{\mathbf{R}^{n}}|{\nabla} \psi|^{2}\left(\int_{\mathbf{R}^{n}}|\psi|
\right)^{4/n}
\end{equation}
Varapoulos $\mathbf{[49]}$ established that the inequality is also equivalent to a Sobolov inequality but only for $n>2$ so that
\begin{equation}
\int_{\mathbf{R}^{n}}|{\nabla} \psi|^{2}\ge \left(\int_{\mathbf{R}^{n}}|\psi|^{2n/(n-2)}\right)^{(n-2)/2}
\end{equation}
This inequality is somewhat close to the classical isoperimetric inequality for the area and volume of a bounded domain $\mathbf{Q}\subset{\mathbf{R}}^{n}$ such that
\begin{equation}
area(\mathbf{Q})\ge C|v(\mathbf{Q})|^{(n-1)/n}
\end{equation}
\begin{lem}
The heat kernel has the following derivatives, which will prove useful
\begin{align}
&\frac{\partial}{\partial t}h(x-y,t)=-\frac{2^{-n-2}(2nt-|y-x|^{2})\exp\left(-\frac{|x-y|^{2}}{4t}\right)
}{\pi^{n/2}t^{\frac{n+4}{2}}}\\&
\nabla_{x}h(x-y,t)\equiv\nabla_{x}h(x-y,t)=\frac{1}{(4 \pi t)^{n/2}}\left(
-\frac{|x-y|}{2t}\right)e^{-|x-y|^{2}/4t}\\&
\nabla^{2}\equiv\Delta_{x}h(x-y,t)=-\frac{2^{-n-2}(2nt-|y-x|^{2})\exp\left(-\frac{|x-y|^{2}}{4t}\right)
}{\pi^{n/2}t^{\frac{n+4}{2}}}
\end{align}
Then it can be seen that ${\square}h(x-y,t)
=\tfrac{\partial}{\partial t}h(x-y,t)-\nabla_{x}h(x-y,t)=0$ as required.
\end{lem}
\begin{prop}
Given the double-sided Gaussian bounds on the heat kernal and its derivatives, the
double-sided integral bounds over all space follow as
\begin{align}
&\Lambda_{1}^{2}t^{-n}\int_{{\mathbf{R^{n}}}}\exp\left(-\frac{2|x-y|^{2}}{\Lambda_{1}t}\right)
d\mu_{n}(y)~
\le~\int_{\mathbf{R}^{n}}|h(x-y,t)|^{2}d\mu_{n}(y)\nonumber\\&~
\le~\Lambda_{2}^{2}t^{-n}\int_{{\mathbf{R}^{n}}}\exp\left(-2\frac{|x-y|^{2}}{\Lambda_{2}t}\right)
d\mu_{n}(y)\\&
\Lambda_{1}^{2}t^{-n}\int_{{\mathbf{Q}}}\left(-\frac{2|x-y|}{\Lambda_{1}t}\right)^{2}e^{-\frac{2|x-y|^{2}}{
\Lambda_{1}t}}d\mu_{n}(y)~
\le~\int_{\mathbf{Q}^{n}}|h(x-y,t)|^{2}d\mu_{n}(y)~\nonumber\\&
\le~\Lambda_{2}^{2}t^{-n}\int_{\mathbf{Q}}\left(-\frac{2|x-y|}{\Lambda_{2}t}\right)^{2}
\exp\left(-2\frac{|x-y|^{2}}{\Lambda_{2}t}\right)d\mu_{n}(y)
\end{align}
\end{prop}
\begin{proof}
These can be shown to hold for $n=1$ at $x=0$ and integrating over $\bm{\mathbf{R}}^{+}=[0,\infty)$. The double-sided bound (2.18) becomes
\begin{align}
\Lambda_{1}^{2}t^{-1}\int_{0}^{\infty}\exp\left(-\frac{2|y|^{2}}{\Lambda_{1}t}\right)dy\le
\frac{1}{4\pi t}\int_{0}^{\infty}\exp\left(-\frac{|y|^{2}}{2t}\right)dy
\le \Lambda_{2}^{2}t^{-1}\int_{0}^{\infty}
\exp\left(-\frac{2|y|^{2}}{\Lambda_{2}t}\right)dy
\end{align}
Performing the integrals
\begin{align}
\frac{\pi^{1/2}}{2^{3/2}}\Lambda_{1}^{5/2}t^{-1/2}\le \frac{1}{2^{5/2}\pi^{1/2}}t^{-1/2}\le \frac{\pi^{1/2}}{2^{3/2}}\Lambda_{2}^{5/2}t^{-1/2}
\end{align}
Redefining the constants
\begin{align}
\overline{\Lambda}_{1}t^{-1/2}\le \frac{1}{2^{5/2}\pi^{1/2}}t^{-1/2}\le
\overline{\Lambda}_{2}t^{-1/2}
\end{align}
so the double-sided integral bound is satisfied. Similarly, for (2.18) we find that
\begin{align}
\overline{\Lambda}_{1}t^{-1/2}\le \frac{3}{2^{9/2}\pi^{1/2}}t^{-1/2}\le
\overline{\Lambda}_{2}t^{-1/2}
\end{align}
\end{proof}
\begin{defn}$(\underline{Global~Initial~Value~Cauchy~Problem (GIVCP)})$\newline
The GIVCP of the homogeneous heat equation $\mathbf{ [44]}$ consists in finding a solution $ u\in\mathrm{C}^{2}(\mathbf{R}^{n}\times(0,\infty))\bigcap\mathrm{C}({\mathbf{R}}^{3}\times(0,\infty))$ such that
\begin{align}
{\square}u(x,t)&=\left(\frac{\partial}{\partial t}-{\Delta}\right)u(x,t)=0,~(x\in {\mathbf{R}}^{n},t\in(0,\infty))\\&
u(x,0)=\phi(x),~(x\in {\mathbf{Q}}^{n},t=0)
\end{align}
\end{defn}
where $\phi\in\mathrm{C}^{2}({\mathbf{R}}^{n})$ is the initial data at $t=0$.
\begin{defn}$(\underline{Initial~Value~Dirichlet-Cauchy~Problem (IVDCP)})$\newline
Let ${\mathbf{Q}}\subset{\mathbf{R}}^{n}$ be an open bounded domain. The IVDCP of the homogeneous heat equation consists in finding a solution $ \psi\in\mathrm{C}^{2}({\mathbf{Q}}\times(0,\infty))
\bigcap\mathrm{C}({\mathbf{Q}}\times(0,\infty))$ such that
\begin{align}
{\square}u(x,t)&\equiv\left(\frac{\partial}{\partial t}-{\Delta}\right)u(x,t)=0,~(x\in {\mathbf{Q}},t\in(0,\infty))\\& u(x,t)=b(x,t),~(x\in\partial {\mathbf{Q}},t>0)\\&
u(x,0)=\phi(x),~(x\in {\mathbf{Q}},t=0)
\end{align}
\end{defn}
where $\phi\in\mathrm{C}({\mathbf{R}}^{n})$ is the initial data at $t=0$ and $u=h\in\mathrm{C}(\partial\mathbf{Q}\times (0,\infty))$ is the boundary condition.
The GCIVP has a unique solution.
\begin{thm}
Given the initial value Cauchy problem for the homogenous heat equation
\begin{align}
{\square}u(x,t)&=\left(\frac{\partial}{\partial t}-{\Delta} \right)u(x,t)=0,~(x\in\mathbf{R}^{n},t\in(0,\infty))\\&
u(x,0)=\phi(x),~(x\in {\mathbf{R}}^{n},t=0)
\end{align}
and the fundamental solution $h(x,t)$ which solves the heat equation away from the singularity $(0,0)$, then $h(x-y,t)$ also solves the heat equation. Here, $\phi\in C^{2}({\mathbf{R}}^{n})$ is a suitable smooth continuous function. The general (and unique) solution is the representation formula or convolution integral
\begin{align}
u(x,t)&=\int_{\mathbf{R}^{n}}h(x-y,t)\phi(y)d\mu_{n}(y)\nonumber\\&={(4\pi t)^{-n/2}}\int_{\mathbf{R}^{n}}\exp\left(-\frac{\|x-y\|^{2}}{4t}\right)
\phi(y)d\mu_{n}(y),~(x\in\mathbf{R}^{n},t>0)
\end{align}
where $\int d\mu_{n}(x)\equiv\int d\mu_{n}(x)$ The solution has the following properties [39]:
\begin{enumerate}
\item $\lim_{t\uparrow\infty}\|u(\bullet,t)-u(\bullet,0)\|_{L_{1}}=0$ or
\begin{equation}
\lim_{t\uparrow\infty}\int_{\mathbf{R}^{n}}(u(x,t)-u(x,0))d\mu_{n}(x)=0
\end{equation}
\item 'Mass conservation'
\begin{equation}
\int_{{\mathbf{R}}^{n}}|u(x,t)|d\mu_{n}(x)=\int_{{\mathbf{R}}^{n}}\phi(x)|d\mu_{n}(x)
\end{equation}
\item Dissipation property
\begin{equation}
\lim_{t\uparrow\infty}\|u(\bullet,t)\|_{L_{p}}= 0
\end{equation}
for all $p\in(1,\infty)$.
\end{enumerate}
\end{thm}
\begin{proof}
Since $h(x-y,t)$ solves the heat equation then
\begin{equation}
{\square}h(x-y,t)
=\left(\frac{\partial}{\partial t}-{\Delta}_{x}\right)h(x-y,t)=0
\end{equation}
Equation (2.31) then solves (2.29) if
${\square}u(x,t)=\frac{\partial}{\partial t}u(x,t)-\Delta_{x}u(x,t)=0$ so that
\begin{equation}
{\square}u(x,t)=\left(\frac{\partial}{\partial t}-{\Delta}\right) u(x,t)=\int_{\mathbf{R}^{n}}
\left(\frac{\partial}{\partial t}-{\Delta}_{x}\right)h(x-y,t)\phi(y)d\mu_{n}(y)=0
\end{equation}
The derivative can be taken under the integral since $h(x-y,t)$ is smooth. The proof of (1) is standard and can be found in Evans $\mathbf{[44]}$. for example. The proof of (2) follows immediately from the convolution property and the Fubini-Tonelli Theorem so that
\begin{equation}
\int_{\mathbf{R}^{n}}u(x,t)d\mu_{n}(x)=\int_{\mathbf{R}^{n}}u(x,t)
\int_{\mathbf{R}^{n}}h(x-y,t)d\mu_{n}(y)=\int_{\mathbf{R}^{n}}u(x,t)d\mu_{n}(x)
\end{equation}
Finally, (3) follows from Young's convolution inequality. For functions $(f,g)\in L_{r}(\mathbf{R}^{n})$ with convolution $(f*g)$ and $\tfrac{1}{q}+\tfrac{1}{r}=1+\tfrac{1}{p}$ Young's inequality is $\|f*g\|_{L_{p}}\le \|f\|_{L_{r}}\|\phi\|_{L_{q}}$. Hence
\begin{align}
&\lim_{t\uparrow\infty}\left\|\int_{{\mathbf{R}}^{n}}h(x-\bullet,t)\phi(y)d\mu_{n}(y)\right\|_{L_{p}}
\le\lim_{t\uparrow\infty}\|\phi(\bullet)\|_{L_{r}}\|\|h(x-\bullet,t)\|_{L_{q}}\nonumber\\&=
\lim_{t\uparrow\infty}\|\phi(\bullet)\|_{L_{r}}\frac{1}{q^{n/2q} (4\pi t)^{{\frac{n}{2}(1-\frac{1}{q})}}}=0
\end{align}
\end{proof}
\begin{lem}(\underline{Holder decomposition of $u(x,t)$)})
Given the convolution integral solution $u(x,t)$ for the CIVP for the heat
equation on $\mathbf{Q}\subset {\mathbf{R}}^{n}$ then for $\tfrac{1}{p}+\tfrac{1}{q}=1$
there is a Holder decomposition
\begin{align}
u(x,t)&=\int_{\mathbf{Q}}h(x-y,t)\phi(y)d\mu_{n}(y)\le
\left(\int_{\mathbf{Q}}|h(x-y,t)|^{q}d\mu_{n}(y)\right)^{1/q}
\left(\int_{\mathbf{Q}}|\phi(y)|^{p}d\mu_{n}(y)\right)^{1/p}\nonumber
\\&\equiv\left(\int_{\mathbf{Q}}h(x-y,t)|^{\frac{p}{p-1}}d\mu_{n}(y)\right)^{\frac{p-1}{p}}
\left(\int_{\mathbf{Q}}|\phi(y)|^{p}d\mu_{n}(y)\right)^{1/p}\nonumber\\&\le
\left(\int_{\mathbf{Q}}|h(x-y,t)|^{p}d\mu_{n}(y)\right)^{p}
\left(\int_{\mathbf{Q}}|\phi(y)|^{p}d\mu_{n}(y)\right)^{1/p}\nonumber\\&
\equiv\left(\|h(x-\bullet,t)\|_{L_{p}(\mathbf{Q})}\right)^{p^{2}}
\|\phi(\bullet)\|_{L_{p}(\mathbf{Q})}
\end{align}
since $q=p/(p-1)\le p$ and $h(x-y,t)\ge 0 $ and
$\int_{{\mathbf{Q}}}h(x-y,t)\ge 0$ for all $t>0$.
\end{lem}
\begin{proof}
Define the $L_{p}(\mathbf{Q})$ norms
\begin{align}
&\|h(x-\bullet,t)\|_{L_{p}(\mathbf{Q})}
=\left(\int_{\mathbf{Q}}|h(x-y,t)|^{p}d\mu_{n}(y)
\right)^{1/p}\nonumber\\&
\|\phi(\bullet)\|_{L_{p}(\mathbf{Q})}
=\left(\int_{\mathbf{Q}}|{\phi}(y)|^{p}d\mu_{n}(y)
\right)^{1/p}
\end{align}
Let $\bm{u}(x-y,t)$ and $v(y)$ be two generic functions defined for all
$(x,y)\in\mathbf{Q}$ and $t>0$. From Young's inequality
\begin{align}
|\mathbf{u}(x-y,t)v(y)|\le
\frac{|u(x,y,t)|^{q}}{q}+
\frac{|\phi(y)|^{p}}{p}
\end{align}
Define
\begin{align}
&\mathbf{u}(x-y,t)=
\frac{h(x-y,t)}{\|h(x-\bullet,t)\|_{L_{p}({\mathbf{Q}})}},~~\bm{v}(y)=
\frac{\phi(y)}{\|\phi(\bullet)\|_{L_{p}({\mathbf{Q}})}}
\end{align}
then
\begin{align}
&\int_{\mathbf{Q}}\bm{u}(x-y,t)\bm{v}(y)d\mu_{n}(y)=
\int_{\mathbf{Q}}\frac{h(x-y,t)}{\|h(x-\bullet,t)\|_{L_{p}({\mathbf{Q}})}}
\frac{\phi(y)}{\|\phi(\bullet)\|_{L_{p}({\mathbf{Q}})}}d\mu_{n}(y)\nonumber\\&
\le \int_{\mathbf{Q}}\frac{|h(x-y,t)|^{q}}{q\|h(x-\bullet,t)\|^{q}_{L_{p}({\mathbf{Q}})}}
d\mu_{n}(y)+\int_{\mathbf{Q}}\frac{\phi(y)|^{p}}{p\|\phi(\bullet)\|^{p}_{L_{p}({\mathbf{Q}})}}
d\mu_{n}(y)\nonumber\\&=\frac{\|h(x-\bullet,t)\|^{q}_{L_{p}({\mathbf{Q}})}}
{q\|h(x-\bullet,t)\|^{q}_{L_{p}({\mathbf{Q}})}}+
\frac{\|\phi(\bullet)\|^{p}_{L_{p}({\mathbf{Q}})}}{p\|\phi(\bullet)\|^{p}_{L_{p}({\mathbf{Q}})}}=
\frac{1}{q}+\frac{1}{p}=1
\end{align}
and so (2.39) follows.
\end{proof}
\begin{cor}
The moments can be Holder decomposed as
\begin{align}
u(x,t)&=\int_{\mathbf{Q}}|h(x-y,t)\phi(y)|^{p}d\mu_{n}(y)\le
\left(\int_{\mathbf{Q}}|h(x-y,t)|^{q}d\mu_{n}(y)y\right)^{p/q}
\left(\int_{\mathbf{Q}}|\phi(y)|^{p}d\mu_{n}(y)y\right)\nonumber
\\&\equiv\left(\int_{\mathbf{Q}}|h(x-y,t)|^{\frac{p}{p-1}}d^{n}y
\right)^{(p-1)}
\left(\int_{\mathbf{Q}}|\phi(y)|^{p}d\mu_{n}(y)\right)^{1/p}\nonumber\\&\le
\left(\int_{\mathbf{Q}}|h(x-y,t)|^{p}d\mu_{n}(y)\right)^{p}
\left(\int_{\mathbf{Q}}|\phi(y)|^{p}d\mu_{n}(y)\right)\nonumber\\&
\equiv\left(\|h(x-\bullet,t)\|_{L_{p}({\mathbf{Q})}}\right)^{p^{2}}
\|\phi(\bullet)\|_{L_{p}(\pmb{\mathbf{Q})}}
\end{align}
\end{cor}
\begin{lem}
Given the Cauchy IVP (2.29) and the solution (2.31) with $u\in C^{\infty}(\bm{\mathbf{R}}^{n}\times [0,\infty)$, then $\exists$ a dimensional constant $c=c(n)$ such that for all $x\in\bm{\mathbf{R}}^{n}$ and $t>0$.
\begin{equation}
|{\nabla} u(x,t)|\le \frac{c(n)}{\sqrt{t}}\|\phi\|_{\infty}
\end{equation}
\end{lem}
\begin{proof}
\begin{align}
|{\nabla} u(x,t)|&={(4\pi t)^{-3/2}}\int_{\mathbf{R}^{3}}\frac{\|x-y\|}{-2t}\exp\left(-\frac{\|x-y\|^{2}}{4t}\right)\phi(y)d^{3}y\nonumber\\&
\le \frac{\|f\|_{\infty}}{(4\pi t)^{3/2}}\int_{\mathbf{R}^{n}}\int_{\mathbf{R}^{n}}\frac{\|x-y\|}{-2t}
\exp\left(-\frac{\|x-y\|^{2}}{4t}\right)d\mu_{n}(y)\nonumber\\&=
\frac{\|f\|_{\infty}}{(4\pi)^{n/2}\sqrt{t}}\int_{\mathbf{R}^3}\|y\|\exp\left(-\|y\|^{2}
\right)d\mu_{n}(y)
\end{align}
\end{proof}
\begin{thm}
Given the initial value Cauchy problem for the inhomogenous heat equation
\begin{align}
&\left(\frac{\partial}{\partial t}-\Delta\right) u(x,t)=f(x,t),~(x\in {\mathbf{R}}^{n},t\in(0,\infty))\\&
u(x,0)=\phi(x),~(x\in \mathbf{R}^{n},t=0)
\end{align}
Here, $f\in C^{2}({\mathbf{R}}^{n})$ is a suitable smooth continuous function. The general(and unique) solution $ u\in\mathrm{C}^{2}({\mathbf{R}}^{n}\times(0,\infty))
\bigcap\mathrm{C}({\mathbf{R}}^{n}\times(0,\infty))$ is $\mathbf{[39]}$.
\begin{align}
u(x,t)&=\frac{1}{(4\pi)^{n/2}}\int_{0}^{R}\int_{{\mathbf{R}}^{n}}
\frac{\exp\left(\frac{-\|x-y\|^{2}}{|t-s|^{2}}\right)}{|t-s|^{2}} f(y,s)d^{3}yds\nonumber\\&+\frac{1}{(4\pi t)^{3/2}}\int_{\mathbf{R}^{3}}\exp\left(-\frac{\|x-y\|^{2}}{4t}\right)\phi(y)d\mu_{n}(y),~
(x\in {\mathbf{R}}^{n},t>0)
\end{align}
\end{thm}
\begin{proof}
The proof is standard and readily follows from Duhamel's Principle $\mathbf{[44]}$.
\end{proof}
\begin{thm}
$(\underline{Maximum~principle~for~the~CIVP}
)$\newline
If $u\in \mathrm{C}^{2}({\mathbf{R}}^{n}\times (0,\infty)
\bigcap\mathrm{C}({\mathbf{R}}^{n}\times (0,\infty)$ solves the homogenous CIVP with
$f=0$ and satisfies the growth condition
\begin{equation}
u(x,t)\le A e^{a\|x\|^{2}},~(x\in{\mathbf{R}}^{n},t\in (0,\infty))
\end{equation}
then
\begin{equation}
\sup |u|_{{\mathbf{R}}^{n}\times (0,\infty)}
=\sup|\phi|_{{\mathbf{R}}^{n}}
\end{equation}
\end{thm}
The solution (2.49) is unique.
\begin{lem}
Let $u\in\mathrm{C}^{2}(\bm{\mathbf{Q}}\times (0,T])\bigcap\mathrm{C}^{2}(\bm{\mathbf{Q}}\times (0,T])$ solve the inhomogeneous Cauchy problem. Then there is at most one unique solution $u(x,t)$ satisfying the growth condition $|u(x,t)|\le A e^{\|x\|^{2}}$ for $t>0$.
\end{lem}
\begin{proof}
If u and v are both solutions then set $w=\pm (u-v)$ and apply the maximum principle so that one must have $u=v$.
\end{proof}
\begin{defn}
Given a general solution for the heat equation CIVP, the following Dirichlet energy integrals can be defined $\mathbf{[44]}$:
\begin{align}
&\bm{\mathrm{I\!E}}_{I}(t)=\frac{1}{2}\int_{\mathbf{R}^{n}}|u(x,t)|^{2}d\mu_{n}(x)
\equiv\frac{1}{2}\|u(\bullet,t)\|_{L_{2}}^{2}\\&
\bm{\mathrm{I\!E}}_{II}(t)=\frac{1}{2}\int_{\mathbf{R}^{n}}|\nabla u(x,t)|^{2}d\mu_{n}(x)
\equiv\frac{1}{2}\|u(\bullet,t)\|_{L_{2}}^{2}
\end{align}
On a bounded domain $\bm{\mathbf{Q}}\subset {\mathbf{R}^{n}}$
\begin{align}
&\bm{\mathrm{I\!E}}_{I,\bm{\mathbf{Q}}}(t)=\frac{1}{2}\int_{\bm{\mathbf{Q}}}|u(x,t)|^{2}
d\mu_{n}(x)
\equiv\frac{1}{2}\|u(\bullet,t)\|_{L_{2},\bm{\mathbf{Q}}}^{2}\\&
\bm{\mathrm{I\!E}}_{II,\bm{\mathbf{Q}}}(t)=\frac{1}{2}\int_{\bm{\mathbf{Q}}}|\nabla u(x,t)|^{2}d\mu_{n}(x)
\equiv\frac{1}{2}\|u(\bullet,t)\|_{L_{2},\bm{\mathbf{Q}}}^{2}
\end{align}
The time derivative is then non decreasing in that
\begin{align}
\frac{d}{dt}\bm{\mathrm{I\!E}}_{I}(t)
=\frac{1}{2}\frac{d}{dt}\int_{{\mathbf{Q}}^{n}}|u(x,t)|^{2}d\mu_{n}(x)=\frac{1}{2}
\frac{d}{dt}\|u(\bullet,t)\|_{L_{2}}^{2}<0
\end{align}
\end{defn}
The energy integral for the CIVP with a heat source term, satisfies the following estimate.
\begin{thm}
Let ${\mathbf{Q}}\subset{\mathbf{R}}^{n}$ be a bounded domain with the CIVP
\begin{align}
&\square u(x,t)=\left(\frac{\partial}{\partial t}-\Delta\right) u(x,t)=f(x,t),~~x\in{\mathbf{Q}},t\in[0,T]\\&
u(x,0)=\phi(x),~~x\in{\mathbf{Q}},t=0\\&
u(x,t)=0,~~x\in\partial{\mathbf{Q}}
\end{align}
where $f(x,t)$ is a heat source term and the temperature is zero on the boundary. Then the following estimate holds
\begin{align}
&\frac{1}{4}\sup_{[0,T]}\int_{\mathbf{Q}}|u(x,t)|^{2}d\mu_{n}(x)
+\int_{0}^{T}\int_{\mathbf{Q}}|\nabla u(x,t)|^{2}d\mu_{n}(x)dt\nonumber\\&\le
T\int_{0}^{T}\int_{\mathbf{Q}}|f(x,t)|^{2}d\mu_{n}(x)
+\frac{1}{2}\int_{\mathbf{Q}}|\phi(x)|^{2}d\mu_{n}(x)
\end{align}
or
\begin{align}
&\frac{1}{2}\sup_{[0,T]}\bm{\mathrm{I\!E}}_{I}(t)
+\int_{0}^{T}\bm{\mathrm{I\!E}}_{I}(t)dt\nonumber\\&\le T\int_{0}^{T}\int_{\mathbf{Q}}|f(x,t)|^{2}d\mu_{n}(x)
+\frac{1}{2}\int_{\mathbf{Q}}|\phi(x)|^{2}d\mu_{n}(x)\equiv\frac{1}{2}\int_{0}^{T}
\|f(\bullet,t)\|_{L_{2}}^{2}dt+\| \phi(\bullet)\|_{L_{2}}^{2}
\end{align}
For the inhomogeneous problem with $f(x,t)=0$
\begin{align}
&\frac{1}{2}\sup_{[0,T]}\bm{\mathrm{I\!E}}_{I}(t)
+\int_{0}^{T}{\bm{\mathrm{I\!E}}}_{I}(t)dt\le
+\frac{1}{2}\int_{\mathbf{Q}}|\phi(x)|^{2}d\mu_{n}(x)\equiv+\| \phi(\bullet)\|_{L_{2}}^{2}
\end{align}
\end{thm}
\subsection{Eigenfunction expansion representation of the heat kernel}
The heat kernel has a formal eigenfunction-eigenvalue representation via the spectral properties of the Laplacian $\mathbf{[9]}$.
\begin{thm}
Let ${\mathbf{Q}}\subset{\mathbf{R}}^{n}$ be a compact set or domain with boundary $\partial{\mathbf{Q}}$ and let $\chi_{k}(x)$ be a set of functions existing for all $x\in{\mathbf{Q}}$ such that $\chi_{k}:{\mathbf{Q}}\rightarrow {\mathbf{R}}$. Elliptic spectral theory for the Laplacian $\Delta_{x}$ establishes that ${\Delta}_{x}\chi_{k}(x)=-\vartheta_{k}\chi_{k}(x)$, where $\lbrace\vartheta_{k}\rbrace$ are a set of eigenvalues. The eigenfunctions form an orthonormal set so that
$\int_{{\mathbf{R}}^{n}}\chi_{i}(x)\chi_{k}(x)d\mu_{n}(x)=\delta_{ij}$. The set $\lbrace\chi_{k}(x)\rbrace$ form an orthonormal basis with respect to the $L_{2}$ norms. The heat kernel then has the eigenvalue-eigenvector representation $\mathbf{[9]}$
\begin{align}
h(x-y,t)=\sum_{k=0}^{\infty}\exp(-\vartheta_{k}t){\chi}_{k}(x)\chi_{k}(y)
\end{align}
which is unique and well defined on ${\mathbf{Q}}\times\bm{\mathbf{Q}}\times (0,\infty)$. Given any function $\phi_{0}(x)=u(x,0)\in L_{2}({\mathbf{Q}})$ the function
\begin{align}
u(x,t)=\int_{{\mathbf{Q}}}\sum_{k=0}^{\infty}\exp(-\vartheta_{k}t)
\chi_{k}(x)\chi_{k}(y)
\phi(y)d\mu_{n}(y)=\int_{{\mathbf{Q}}}h(x-y,t)\phi(y)d\mu_{n}(y)
\end{align}
solves the heat equation ${\square}u(x,t)=(\frac{\partial}{\partial t}-{\Delta}_{x})u((x,t)=0$ for $x\in{\mathbf{Q}},t>0$ and $u(x,0)=u_{0}(x)$ for $x\in{\mathbf{Q}},t=0$ and the BC $u(x,t)=0$ on $\partial{\mathbf{Q}}\times (0,\infty)$. Also $\lim_{t\uparrow 0}u(x,t)=u_{0}(x)$. Moreover $h(x-y,t)$ is positive on
$\bm{\mathbf{Q}}{\symbol{92}}\partial{\mathbf{Q}}\times{\mathbf{Q}}{\symbol{92}}
\partial{\mathbf{Q}}\times (0,\infty)$ and $\int_{{\mathbf{Q}}}h(x-y,t)\le 1$
\end{thm}
\begin{proof}
A detailed proof and discussion is given in Li $\mathbf{[9]}$. Basically, any function $\psi_{0}(x)=\phi(x)$ can be expanded as a weighted sum
\begin{align}
u(x,0)=u_{0}(x)=\phi(x)=\sum_{k=1}^{\infty}\mathcal{A}_{k}\chi_{k}(x)
\end{align}
with $ \mathcal{A}_{k}=\int_{\mathbf{Q}} u_{0}(x)\chi_{k}(x)d\mu_{n}(x) $. Any $u(x,t)$ with the expansion
\begin{align}
u(x,t)=\sum_{k=1}^{\infty}\exp(-\vartheta_{k}t)\mathcal{A}_{k}\chi_{k}(x)
\end{align}
then satisfies the heat equation ${\square}u(x,t)=(\tfrac{\partial}{\partial t}-{\Delta}_{x})u(x,t)=0$ with prescribed boundary and initial conditions. To show that the sum (2.66) is a solution of the heat equation, take the derivatives so that
\begin{align}
&\frac{\partial}{\partial t}u(x,t)=-\sum_{k=1}^{\infty}\vartheta_{k}\exp(-\vartheta_{k}t)A_{k}\chi_{k}(x)\\&
\nabla_{x}u(x,t)=\sum_{k=1}^{\infty}\exp(-\vartheta_{k}t)A_{k}\nabla_{x}\chi_{k}(x)\\&
\Delta_{x}u(x,t)=\sum_{k=1}^{\infty}\exp(-\vartheta_{k}t)A_{k}\Delta_{x}\chi_{k}(x)=
\sum_{k=1}^{\infty}\exp(-\vartheta_{i}t)A_{k}(-\vartheta_{k}\chi_{k}(x))
\end{align}
since the Laplacian has the spectral property $\Delta \chi_{k}(x)=-\vartheta_{k}\chi_{k}(x)$. Then clearly ${\square}u(x,t)=\left(\tfrac{\partial}{\partial t}-\nabla\right) u(x,t)=0$ as required.
\end{proof}
\begin{exam}
The solution of the heat equation on the ring or circle ${\mathbf{S}}^{1}$
has a Fourier series representation. The CIVP for the heat equation on the circle is
\begin{align}
&{\square}_{\theta}u(\theta,t)=\left(\frac{\partial}{\partial t}-
\frac{\partial^{2}}{\partial\theta^{2}}\right)u(\theta,t)=0,
~x\in {\mathbf{S}},t>0\\&
u(\theta,0)=f(\theta),x\in {\mathbf{S}^{1}},t=0
\end{align}
Since $u(\theta,t)=u(\theta+2\pi,t)$, the solution is periodic on the ring so there is a Fourier series solution
\begin{align}
u(\theta,t)=\mathcal{A}_{o}+\sum_{m=1}^{\infty}\exp(-m^{2}t)
\left(\mathcal{A}_{m}\cos(m\theta)+\mathcal{B}_{m}\sin(m\theta)\right)
\end{align}
with Fourier coefficients $\mathcal{A}_{i}=\tfrac{1}{\pi}\int_{-\pi}^{\pi}
f(\theta)\cos(m\theta)d\theta$ and $\mathcal{B}_{i}=\tfrac{1}{\pi}
\int_{-\pi}^{\pi}f(\theta)\sin(m\theta)d\theta$
\end{exam}
\begin{rem}
Any heat kernel $h_{(\mathbf{M})}(x,y,t)$ on a Riemannian manifold $({\mathbf{M}},\bm{g})$ approaches the heat kernel in Euclidean space as $t\rightarrow 0$ and $d(x,y)\rightarrow 0$ since all Riemannian manifolds are locally Euclidean $\mathbf{[9]}$. Then for all $(x,y)\in{\mathbf{M}}\symbol{92}\partial {\mathbf{M}}$
\begin{align}
h_{(\mathbf{M})}(x-y,t)\sim|4\pi t|^{-n/2}\exp\left(-\frac{d^{2}(x,y)}{4t}\right)
\end{align}
as $t\rightarrow 0$ and $d(x,y)\rightarrow 0$. Also, for all fixed
$(x,y)\in {\mathbf{M}}\symbol{92}\partial{\mathbf{M}}$ one has $h_{\mathbf{M}}(x-y,t)\rightarrow 0$ as $t\rightarrow\infty$. Also, for $t\gg 0$, the asymptotic behavior is dominated by the first eigenvalue and eigenvectors so that
\begin{align}
h_{(\mathbf{M})}(x-y,t\sim \exp(-\vartheta_{1}t)\chi_{1}(x)\chi_{1}(y)
\end{align}
As an example, the heat kernel is known for hyperbolic 3-space ${\mathbf{H}}_{3}$ which has negative curvature $\mathbf{[52]}$ so that
\begin{align}
h_{(\mathbf{H}_{3})}
(x-y,t) = |4\pi t|^{-n/2}\frac{d(x,y)}{\sinh d(x,y)}\exp\left(-\left((t+\frac{d(x,y)}{4t}\right)\right)
\end{align}
Since $\sinh d(x,y)=d(x,y)+\tfrac{1}{6}|d(x,y)|^{2}+\tfrac{1}{120}|d(x,y)|^{5}+...$ then
$\sinh d(x,y)\sim d(x,y)$. For small $t$ and $d(x,y)\sim |x-y|$ the kernel
$h_{(\mathbf{H}_{3})}$ then approaches the Euclidean heat kernel so that
\begin{align}
h_{(\mathbf{H}_{3})} = |4\pi t|^{-n/2}\frac{d(x,y)}{d(x,y)}\exp\left(-\left(t+\frac{d(x,y)}{4t}\right)\right)
\sim |4\pi t|^{-n/2}\exp\left(-\frac{|x-y|^{2}}{4t}\right)
\end{align}
\end{rem}
The eigenvalue-eigenvector representation can also be used to prove that heat
kernel obeys the semi-group property $\mathbf{[9]}$.
\begin{lem}
The heat kernel obeys the semi-group property $\mathbf{[9]}$
\begin{align}
h(x-y,t+s)= \int_{{\mathbf{R}}^{n}}h(x-z,t)h(y-z,s)d\mu_{n}(z)
\end{align}
and in particular
\begin{align}
h(x-x,2t)= \int_{{\mathbf{R}}^{n}}h(x-y,t)h(x-y,t)d\mu_{n}(y)
\end{align}
\end{lem}
\begin{proof}
\begin{align}
&\int_{{\mathbf{R}}^{n}}h(x-z,t)h(y-z,t)
=\int_{{\mathbf{R}}^{n}}\sum_{i,j}^{n}
\exp(-\vartheta_{i}t)e^{-\vartheta_{j}t}\chi_{i}(x)\chi_{i}(z)\chi_{j}(x)
\chi_{j}(z)d\mu_{n}(z)\nonumber\\&
=\sum_{i}^{n}e^{-\vartheta_{i}(t+s)}\chi_{i}(x)\chi_{i}(y)
\int_{{\mathbf{R}}^{n}}\chi_{i}(z)\chi_{j}(z)d\mu_{n}(z)
=\sum_{i}^{n}e^{-\vartheta_{i}(t+s)}\chi_{i}(x)\chi_{i}(y)\equiv h(x-y,t+s)
\end{align}
and (2.77) follows. This property holds over a closed domain ${\mathbf{Q}}$.
\end{proof}
The following lemma (Jensen's inequality) is required for Theorem (2.22)
\begin{lem}
Let $\Phi$ be a concave function and $f:{\mathbf{Q}}\rightarrow\mathbf{R}$. Then
\begin{align}
\int_{{\mathbf{Q}}}\Phi(f(x))d\mu_{n}(x) \le \Phi\left(\int_{{\mathbf{Q}}}f(x)d\mu_{n}(x)\right)
\end{align}
If $\varphi(f)=f^{p}$ for $p\ge 2$ then
\begin{align}
\int_{{\mathbf{Q}}}|f(x)|^{p}d\mu_{n}(x) \le\left( \int f(x)d\mu_{n}(x)\right)^{p}
\end{align}
Then if $f(x)=\varphi(x)\varphi(x)$
\begin{align}
\int_{{\mathbf{Q}}}|\varphi(x)\varphi(x)|^{p}d\mu_{n}(x) \le\left( \int \varphi(x)\varphi(x)d\mu_{n}(x)
\right)^{p}
\end{align}
\end{lem}
A bound on the $L_{p}$ moments of the heat kernel were discussed in Theorem 2.8. A comparable estimate can be derived from the eigenfunction expansion.
\begin{thm}
The $L_{p}$ moments of the heat kernel have the estimate
\begin{align}
(\|h(x,\bullet,t)\|_{L_{p}})^{p}
=\int_{{\mathbf{Q}}}h(x-y,t)|^{p}d\mu_{n}(y)
\le |h(x-x,t)|^{p/2}\lambda(t)=\frac{\lambda(t)}{(4\pi t)^{pn/4}}
\end{align}
where $\lambda(t)=\exp(-\xi t)/(1-\exp(-\xi t))$ and $\xi $ is a constant.
\end{thm}
\begin{proof}
The proof utilises the Cauchy Schwartz and Jensen inequalities
\begin{align}
&(\|h(x-\bullet,t)\|_{L_{p}})^{p}
=\int_{{\mathbf{Q}}}h(x-y,t)|^{p}d\mu_{n}(y)
=\int_{{\mathbf{Q}}}\left|\sum_{k=1}^{n}\exp(-\vartheta_{k}t)\chi_{k}(x)
\chi_{k}(y)\right|^{p}d\mu_{n}(y)\nonumber\\&
\equiv =\int_{{\mathbf{Q}}}\left|\sum_{k=1}^{n}\exp(-\tfrac{1}{2}\vartheta_{k}t)\chi_{k}(x)
\exp(-\tfrac{1}{2}\vartheta_{k}t)\chi_{k}(y)\right|^{p}d\mu_{n}(y)\nonumber\\&
\le =\int_{{\mathbf{Q}}}\left(\sum_{k=1}^{n}\left|\exp(-\tfrac{1}{2}\vartheta_{k}t)\chi_{k}(x)
\right|^{2}\right)^{p/2}\left(\sum_{k=1}^{n}\left|\exp(-\tfrac{1}{2}\vartheta_{k}t)\chi_{k}(y)\right|^{2}\right)^{p/2}
d\mu_{n}(y)\nonumber\\&
=\int_{{\mathbf{Q}}}\left(\sum_{i=1}^{n}\exp(-\vartheta_{k}t)\chi_{k}(x)\chi_{k}(x)
\right)^{p/2}\left(\sum_{k=1}^{n}\exp(-\vartheta_{k}t)\chi_{k}(y)\chi_{k}(y)\right)^{p/2}
d\mu_{n}(y)\nonumber
\\&=|h(x-x,t)|^{p/2}\int_{{\mathbf{Q}}}\left(\sum_{k=1}^{\infty}\exp(-\vartheta_{k}t)
\chi_{k}(y)\chi_{k}(y)
\right)^{p/2}d\mu_{n}(y)
\nonumber\\&
\le|h(x-x,t)|^{p/2}\int_{{\mathbf{q}}}
\left(\sum_{k=1}^{n}\exp(-\vartheta_{k}t)\int_{{\mathbf{Q}}}\chi_{k}(y)
\chi_{k}(y)d\mu_{n}(y)\right)^{p/2}
\nonumber\\&
\le|h(x-x,t)|^{p/2}\left(\sum_{k=1}^{\infty}\exp(-\vartheta_{k}t)\right)^{p/2}
\end{align}
The sum $\left(\sum_{k=1}^{\infty}\exp(-\vartheta_{k}t)\right)^{p/2}$ is essentially
a geometric series and converges. Since $\vartheta_{1}<\vartheta_{2}<\vartheta_{3}...$, it can be bounded by a convergent series of the form
\begin{align}
\left(\sum_{k=1}^{\infty}\exp(-\vartheta_{k}t)\right)^{p/2}\le\left(\sum_{k=1}^{\infty}
\exp(-k\alpha t)\right)^{p/2}\equiv
\left(\sum_{k=1}^{\infty}|\exp(-\alpha t)|^{k}\right)^{p/2}
\end{align}
where $\alpha $ is a constant. Since $\exp(-\alpha t)\le 1 $ for all $t>0$, the series converges so that
\begin{align}
\lambda(t)=\left(\sum_{k=1}^{\infty}e^{-k\alpha t}\right)^{p/2}\equiv
\left(\sum_{k=1}^{\infty}|e^{-\alpha t}|^{k}\right)^{p/2}=\frac{1}{1-e^{-\alpha t}}
\end{align}
and $S\rightarrow 1$ as $t\rightarrow \infty$. Hence, we have the estimate
\begin{align}
&(\|h(x-y,t)\|_{L_{p}})^{p}
=\int h(x-y,t)|^{p}d\mu_{n}(y) \le|h(x-x,t)|^{p/2}\frac{1}{1-e^{-\alpha t}}
\nonumber\\&=|h(x-x,t)|^{p/2}\lambda(t)\sim |
h(x,x,t)|
\end{align}
\end{proof}
\begin{cor}
The estimate can be compared to the standard estimate given in (2.38). For $p=2$
\begin{align}
(\|h(x-y,t)\|_{L_{2}})^{2}\le
 \left|\Lambda_{p}t^{-\frac{n}{2}(1-\frac{1}{2})}\right|^{2} = \Lambda_{p}t^{-n/4}
\end{align}
whereas (2.87) gives
\begin{align}
(\|h(x-y,t)\|_{L_{2}})^{2}=\frac{1}{4\pi t^{2n/4}}\sim t^{-n/4}
\end{align}
and so the estimates are comparable.
\end{cor}
\subsection{Mean value properties on a heat ball}
\begin{defn}
Let ${\mathbf{U}}\subset{\mathbf{R}}^{n}$ be an open and bounded domain and fix a time interval $(0,T]$. Then:
\begin{enumerate}
\item The parabolic cylinder is defined as ${\mathrm{U}}_{T}={\mathbf{U}}\times (0,T]$.
\item The parabolic boundary is $\Gamma_{T}=\widehat{{\mathbf{U}}}_{T}-{\mathbf{U}}_{T}$
\end{enumerate}
${\mathbf{U}}_{T}$ is defined as the parabolic interior of $\widetilde{{\mathbf{U}}}_{T}\times (0,T]$ and includes the top $\mathbf{U}_{T}\times\lbrace t=T\rbrace$. The boundary includes the vertical sides and bottom but not the top $\mathbf{[44]}$.
\end{defn}
The mean-value property and maximum principle for the elliptic Laplace equation can be defined with respect to a ball ${\mathbf{B}}_{R}(0)\subset{\mathbf{R}}^{n}$. An analogous mean-value property and maximum principle for the heat equation can be defined with respect to a "heat ball" $\bm{\mathrm{I\!B}}(x,t;R)\subset{\mathbf{U}}_{T}$.
\begin{defn}
For fixed $x\in{\mathbf{R}}^{n},t\in{\mathbf{R}}^{+}$ and $R>0$ define a heat ball as the region
\begin{equation}
\bm{\mathrm{I\!B}}(x,t;R)=\bigg{\lbrace}(y,s)\in\bm{\mathrm{R}}^{n+1}|s\le t,h(x-y,t-s)\ge \frac{1}{R^{n}}\bigg{\rbrace}\subset {\mathbf{U}}_{T}
\end{equation}
where $h(x-y,t-s)=\tfrac{1}{4\pi|t-s|^{3/2}}\exp\left(-\|x-y\|^{2}/4|t-s|\right)$ is the heat kernel or fundamental solution. The point $(x,t)$ is the centre of the top of the heat ball $\mathbf{[44]}$.
\end{defn}
\begin{thm}($\underline{Mean-value~property~of~heat~equation}$)
Let $u\in C^{2}({\mathbf{U}}_{T})$ solve the heat equation then for all $\bm{\mathrm{I\!B}}(x,t;R)\subset {\mathbf{U}}_{T}$
\begin{equation}
u(x,t)=\frac{1}{4R^{n}}\iint_{\bm{\mathrm{I\!B}}(x,t;R)}
u(y,s)\frac{\|x-y\|^{2}}{|t-s|^{2}}d^{3}yds
\end{equation}
Values of $u(x,t)$ depend on past times $s\le t$.
\end{thm}
The proof is given in Evans [39].
\section{Gradient estimates: Li-Yau and Harnack inequalities}
The solution of the heat equation and its gradient on a generic Riemannian manifold satisfies some Harnack-Like inequalities which are important within differential geometry and
geometric analysis $\mathbf{[53]}$. Firstly, there are also some basic formulas arising from the positivity of solutions.
\begin{lem}
Let ${\mathbf{M}}=\bm{\mathbf{R}}^{n}$ and $u=u(x,t)$ is a solution of the heat equation
on ${\mathbf{R}}^{n}$ or ${\mathbf{Q}}\subset\bm{\mathrm{R}}^{n}$. Then the following
equalities hold
\begin{align}
&{\square}(\log u)=\left(\frac{\partial}{\partial t}-\Delta_{x}\right)(\log u)
=\frac{|\nabla u|^{2}}{|u|^{2}}\\&
{\square}(u\log u)=\left(\frac{\partial}{\partial t}-\Delta_{x}\right)(u\log u)
=-\frac{|\nabla u|^{2}}{|u|}
\end{align}
or equivalently
\begin{align}
&|u|^{2}{\square}(\log u)=|u|^{2}\left(\frac{\partial}{\partial t}-\Delta_{x})\right)(\log u)
=|\nabla\psi|^{2}\\&
-|u|{\square}(u\log u)=-|u|\left(\frac{\partial}{\partial t}-\Delta_{x}\right)(u \log u)
={|\nabla u|^{2}}
\end{align}
In terms of the convolution integral solution
\begin{align}
&{\square}(\log u)=\left(\frac{\partial}{\partial t}-\Delta_{x}\right)(\log u)
=\frac{\left|\int_{\mathbf{Q}}\nabla h(x-y,t)\phi(y)d\mu_{n}(y)\right|^{2}}
{\left|\int_{\mathbf{Q}}h(x-y,t)\phi(y)d\mu_{n}(y)\right|^{2}}
\end{align}
\begin{align}
&{\square}(u\log u)=\left(\frac{\partial}{\partial t}-\Delta_{x}\right)(u\log u)
=-\frac{\left|\int_{\mathbf{Q}}\nabla h(x-y,t)\phi(y)d\mu_{n}(y)\right|^{2}}
{\left|\int_{\mathbf{Q}}h(x-y,t)\phi(y)d\mu_{n}(y)\right|}
\end{align}
or equivalently
\begin{align}
|u|^{2}{\square}(\log u)=|u|^{2}\left(\frac{\partial}{\partial t}-\Delta_{x}\right)(\log u)
={\left|\int_{\mathbf{Q}}\nabla h(x-y,t)\phi(y)d\mu_{n}(y)\right|^{2}}
\end{align}
\begin{align}
-|u|{\square}(u\log u)=-|u|\left(\frac{\partial}{\partial t}-\Delta_{x}\right)(u\log u)
={\left|\int_{\mathbf{Q}}\nabla h(x-y,t)\phi(y)d\mu_{n}(y)\right|^{2}}
\end{align}
\end{lem}
\begin{proof}
The proof follows easily by direct calculation so that
\begin{align}
&{\square}(\log u)=\left(\frac{\partial}{\partial t}-\Delta u\right))(\log u)\equiv
\frac{\frac{\partial}{\partial t}u}{u}-\nabla(\nabla\log u)\nonumber\\&
=\frac{\Delta u}{u}-\nabla\left(\frac{\nabla u}{u}\right)=\frac{\Delta u}{u}
-\frac{(\nabla\nabla u)u-\nabla u\nabla u}{|u|^{2}}\nonumber\\&=
\frac{\Delta u}{u}-\frac{\Delta u}{u}+\frac{|\nabla u|^{2}}{|u|^{2}}=\frac{|\nabla u|^{2}}{|u|^{2}}
\end{align}
and similary for (3.2).
\end{proof}
\begin{thm}($\underline{Li-Yau~inequality}$)
Let ${\mathbf{M}}$ be an n-dimensional manifold with $\bm{\mathrm{Ric}}({\mathbf{M}})\ge k^{2}$. The solution $u(x,t)$ of the heat equation
\begin{align}
&{\square}u(x,t)\equiv \left(\frac{\partial}{\partial t}-\Delta \right)u(x,t)=0,~x\in{\mathbf{M}},t>0\nonumber\\&
u(x,0)=u(x),~x\in{\mathbf{M}},t=0
\end{align}
then satisfies the estimate
\begin{align}
\frac{\|\nabla u(x,t)\|^{2}}{\|u(x,t)\|^{2}}-\alpha\frac{\frac{\partial}{\partial t} u(x,t)}{u(x,t)}\le
\frac{1}{2}n\alpha^{2}\left\lbrace\frac{1}{t}+\frac{k^{2}}{2(\alpha-1)}\right\rbrace
\end{align}
On Euclidean space with ${\mathbf{M}}={\mathrm{R}}^{n}$ then $\mathbf{g}_{ij}=\delta_{ij}$ and $k=0$. Letting $\alpha\rightarrow 1$
\begin{align}
\frac{\|\nabla u(x,t)\|^{2}}{\|u(x,t)\|^{2}}-\frac{\frac{\partial}{\partial t} u(x,t)}{u(x,t)}\le\frac{1}{2}n\frac{1}{t}
\end{align}
or equivalently
\begin{align}
|\nabla\log u(x,t)|^{2}-|\frac{\partial}{\partial t}log u(x,t)|\le \frac{n}{2t}
\end{align}
\end{thm}
\begin{proof}
Define $\mathcal{Y}=\log u$ and $\mathcal{W}=t(|\nabla u|^{2}-\frac{\partial}{\partial t}u)$. For any function $\psi\in C^{\infty}({\mathbf{M}})$ the Bochner formula is
\begin{align}
&\frac{1}{2}{\Delta}|{\nabla}\psi|^{2}=|\nabla^{2}\psi|+\big\langle \nabla(\Delta\psi),\nabla\psi\big\rangle-\bm{\mathrm{Ric}}(\nabla\psi,\nabla\psi)\nonumber\\&
|\nabla^{2}\psi|+\big\langle\nabla(\Delta\psi),\nabla\psi\big\rangle
\equiv|\nabla_{i}\nabla_{j}\psi|+\big\langle\nabla(\Delta\psi),\nabla\psi\big\rangle
\end{align}
where $\langle A,B\rangle=\sum_{i=1}^{n}A_{i}B^{i}$ is the inner product.
Application to $\mathcal{Y}$ gives
\begin{align}
{\Delta}|\nabla\mathcal{Y}|^{2}\ge 2|\nabla_{i}\nabla_{j}\mathcal{Y}|+2\bigg\langle{\nabla}({\Delta}\mathcal{Y}),
\nabla\mathcal{Y}\bigg\rangle
\end{align}
Since $\mathcal{Y}=\log u$ and since $u$ is a solution of the heat equation ${\square} u(x,t)=\left(\frac{\partial}{\partial t}-{\Delta}\right)u=0$, then taking the Laplacian ${\Delta}\mathcal{Y}=\Delta \log u=\tfrac{1}{u}\nabla_{i}\nabla^{i}u-\tfrac{1}{u^{2}}
\nabla_{i}u \nabla_{i}u$ gives
\begin{align}
{\Delta}\mathcal{Y}=\frac{\partial}{\partial t}\mathcal{Y}-|{\nabla}\mathcal{Y}|^{2}
=-\frac{1}{t}\mathcal{W}
\end{align}
The B-W formula the gives
\begin{align}
{\Delta}|\nabla\mathcal{Y}|^{2}\ge 2|{\nabla}_{i}{\nabla}\mathcal{Y}_{j}\mathcal{Y}|+2\left\langle
{\nabla}\left(-\frac{1}{t}\mathcal{W}\right),
{\nabla}\mathcal{Y}\right\rangle
\end{align}
or
\begin{align}
{\Delta}|t\nabla\mathcal{Y}|^{2}\ge 2t|\nabla_{i}\nabla_{j}\mathcal{Y}-2\bigg\langle\nabla\mathcal{W},
\nabla\mathcal{Y}\bigg\rangle
\end{align}
Now since $|\nabla_{i}\nabla_{j}\mathcal{Y}|^{2}\ge
\frac{1}{n}|{\Delta}\mathcal{Y}|^{2}$ and using (3.17)
\begin{align}
|\nabla_{i}\nabla_{j}\mathcal{Y}|^{2}\ge \frac{1}{n}|{\Delta}\mathcal{Y}|^{2}\equiv \frac{1}{n}
\frac{\mathcal{W}^{2}}{t^{2}}
\end{align}
so that
\begin{align}
{\Delta}(t|\nabla\mathcal{Y}|^{2})\ge \frac{2}{nt}|\mathcal{W}\|^{2}
-2\big\langle\nabla\mathcal{W},{\nabla}\mathcal{Y}\big\rangle
\end{align}
Since
\begin{align}
\frac{\partial}{\partial t}(t|{\nabla}\mathcal{Y}|^{2})
=|{\nabla}\mathcal{Y}|^{2}+t\frac{\partial}{\partial t}(|\nabla\mathcal{Y}|^{2})
\end{align}
it follows that
\begin{align}
-{\square}\mathcal{Y}=\left({\Delta}-\frac{\partial}{\partial t}\right)(t|\nabla\mathcal{Y}|^{2})\ge
\frac{2}{nt}|\mathcal{W}|^{2}-
2\bigg\langle {\nabla}\mathcal{W},{\nabla}\mathcal{Y}\bigg\rangle
-|\nabla\mathcal{Y}|^{2}-t\frac{\partial}{\partial t}(|\nabla\mathcal{Y}|^{2})
\end{align}
However, using (3.18)
\begin{align}
&-{\square}\mathcal{Y}=\left(\Delta-\frac{\partial}{\partial t}\right))(t|\nabla\mathcal{Y}|^{2})
=t\Delta(\frac{\partial}{\partial t}\mathcal{Y})-t\frac{\partial^{2}}{\partial t^{2}}\mathcal{Y}-\frac{\partial}{\partial t}\mathcal{Y}\nonumber\\&
=t\frac{\partial}{\partial t}(\Delta\mathcal{Y})-t\frac{\partial^{2}}{\partial t^{2}}\mathcal{Y}-\frac{\partial}{\partial t}\mathcal{Y}
=t\frac{\partial}{\partial t}(|\nabla\mathcal{Y}|^{2})
-\frac{\partial}{\partial t}\mathcal{Y}
\end{align}
From this and $\mathcal{W}=t(|\nabla\mathcal{Y}|^{2}-\partial\mathcal{Y})$ we have
\begin{align}
&-{\square}\mathcal{W}=\left(\Delta-\frac{\partial}{\partial t}\right)\mathcal{W}\ge\frac{2}{nt}|\mathcal{W}|^{2}-2
\bigg\langle\nabla\mathcal{W},\nabla u\bigg\rangle-|\nabla\mathcal{Y}|^{2}+
\frac{\partial}{\partial t}\mathcal{Y}\\&
=\frac{2}{nt}|\mathcal{W}|^{2}-2\bigg\langle\nabla\mathcal{W},\nabla\mathcal{Y}\bigg\rangle-
\frac{\mathcal{W}}{t}
\end{align}
The Maximum Principle implies that $\mathcal{Z}<n/2$ so that $t(|\nabla\mathcal{Y}|^{2}-\frac{\partial}{\partial t}\mathcal{Y})<n/2$ and the inequality follows.
\begin{thm}(\underline{Li-Yau~inequality~for~convolution~integral})
Let $u(x,t)=\int_{{\mathbf{Q}}}h(x-y,t)\phi(y)d\mu_{n}(y)$
be a solution of the heat equation ${\square}u(x,t)=0$ for
$x\in{\mathbf{Q}},t>0$ for initial data $u(x,0)=\phi(x)$. Then the Li-Yau inequality (3.12) implies the following bounds
\begin{align}
&\left|\int_{{\mathbf{Q}}}|\nabla_{x}h(x-y,t)|^{2}d\mu_{n}(y)\right|
\le\left(\int_{{\mathbf{Q}}}|\frac{\partial}{\partial t}h(x-y,t)|^{2}d\mu_{n}(y)\right)^{1/2}
\left(\int_{{\mathbf{Q}}}|h(x-y,t)|^{2}d\mu_{n}(y)\right)^{1/2}\nonumber\\&
\le \frac{1}{2}nt^{-1}\int_{{\mathbf{Q}}}|h(x-y,t)|^{2}d\mu_{n}(y)
\end{align}
and also
\begin{align}
&\Lambda_{1}^{2}t^{-n}\int_{\bm{\mathbf{Q}}}\left(-\frac{2|x-y|}{\Lambda_{1}t}\right)^{2}
\exp(-\frac{2|x-y|^{2}}{\Lambda_{1}t})d\mu_{n}(y)\nonumber\\&\le
\left(\int_{{\mathbf{Q}}}|\frac{\partial}{\partial t}h(x-y,t)|^{2}d\mu_{n}(y)\right)^{1/2}
\left(\int_{{\mathbf{Q}}}|h(x-y,t)|^{2}d\mu_{n}(y)\right)^{1/2}\nonumber\\&
\le \Lambda_{2}^{2}t^{-n}\int_{\bm{\mathbf{Q}}}\exp(-\frac{2|x-y|^{2}}{\Lambda_{2}t})
d\mu_{n}(y)
\end{align}
utilising the double-sided integral bound (2.7).
\end{thm}
\begin{proof}
Using the Cauchy-Schwarz inequality
\begin{align}
|\nabla_{x}u(x,t)|^{2}
=\left|\int_{{\mathbf{Q}}}\nabla_{x}
h(x-y,t)\phi(y)d\mu_{n}(y)\right|^{2}\le\left(\int_{{\mathbf{Q}}}
|\nabla_{x}h(x-y,t)|^{2}d\mu_{n}(y)\right)
\left(\int_{{\mathbf{Q}}}|\phi(y)|^{2}d\mu_{n}(y)\right)
\end{align}
\begin{align}
|\frac{\partial}{\partial t} u(x,t)|
=\left|\int_{{\mathbf{Q}}}\frac{\partial}{\partial t}
h(x-y,t)\phi(y)d\mu_{n}(y)\right|^{2}
\le\left(\int_{\bm{\mathbf{Q}}}|\frac{\partial}{\partial t}h(x-y,t)|^{2}d\mu_{n}(y)\right)^{1/2}
\left(\int_{\bm{\mathbf{Q}}}|\phi(y)|^{2}d\mu_{n}(y)\right)^{1/2}
\end{align}
\begin{align}
|u(x,t)|
=\left|\int_{{\mathbf{Q}}}\frac{\partial}{\partial t}h(x-y,t)\phi(y)d\mu_{n}(y)\right|
\le\left(\int_{{\mathbf{Q}}}|h(x-y,t)|^{2}d\mu_{n}(y)\right)^{1/2}
\left(\int_{{\mathbf{Q}}}|\phi(y)|^{2}d\mu_{n}(y)\right)^{1/2}
\end{align}
The Li-Yau inequality then becomes
\begin{align}
&\left(\int_{{\mathbf{Q}}}|\nabla_{x}h(x-y,t)|^{2}d\mu_{n}(y)\right)
\left(\int_{{\mathbf{Q}}}|\phi(y)|^{2}d\mu_{n}(y)\right)\nonumber\\&\le
\left(\int_{{\mathbf{Q}}}|\frac{\partial}{\partial t}h(x-y,t)|^{2}d\mu_{n}(y)\right)^{1/2}
\left(\int_{{\mathbf{Q}}}|h(x-y,t)|^{2}d\mu_{n}(y)\right)^{1/2}
\left(\int_{{\mathbf{Q}}}|\phi(y)|^{2}d\mu_{n}(y)\right)^{1/2}\nonumber\\&
\le\frac{1}{2}nt^{-1} \left(\int_{{\mathbf{Q}}}|h(x-y,t)|^{2}d\mu_{n}(y)\right)
\left(\int_{{\mathbf{Q}}}|\phi(y)|^{2}d\mu_{n}(y)\right)
\end{align}
Cancelling the $\int_{\mathbf{Q}}|\phi(y)|^{2}d^{n}y$ terms then gives (3.26). Applying
the double-sided bound then gives (3.27).
\end{proof}
\begin{lem}
Consider the CIVP for the heat equation on ${\mathbf{R}}^{n}$ with constant initial data
\begin{align}
&{\square}u(x,t)=\left(\frac{\partial}{\partial t}-\Delta\right)u(x,t)=0,x\in{\mathbf{R}}^{n},t>0\\&
u(x,0)=\lambda,~x\in{\mathbf{R}}^{n},t=0
\end{align}
with solution
\begin{align}
u(x,t)=\frac{\lambda}{(4\pi t)^{n/2}}\int_{{\mathbf{R}}^{n}}e^{-\frac{(x-y)^{2}}{4t}}d\mu_{n}(y)\equiv
\lambda\int_{{\mathbf{R}}^{n}}h(x-y,t)d\mu_{n}(y)
\end{align}
Then the Li-Yau inequality (3.12)becomes
\begin{align}
\frac{\left|\int_{{\mathbf{R}}^{n}}{\nabla}_{x}h(x-y,t)d\mu_{n}(y)\right|^{2}}{
\left|\int_{{\mathbf{R}}^{n}}h(x-y,t)d\mu_{n}(y)\right|^{2}}-
\frac{\left|\int_{{\mathbf{R}}^{n}}\frac{\partial}{\partial t}h(x-y,t)d\mu_{n}(y)\right|}{
\left|\int_{{\mathbf{R}}^{n}}h(x-y,t)d\mu_{n}(y)\right|}\le \frac{n}{2t}
\end{align}
or equivalently
\begin{align}
\left|\int_{{\mathbf{R}}^{n}}{\nabla}_{x}\log(u(x,t))d\mu_{n}(x)\right|^{2}-
\int_{{\mathbf{R}}^{n}}\frac{\partial}{\partial t}\log(u(x,t))d\mu_{n}(x)
\end{align}
For $n=1$ this gives the bound
\begin{align}
\frac{\exp(^-\frac{x^{2}}{2t})}{\pi}\frac{1}{|erf\left(\frac{x}{2t^{1/2}}\right)+1|^{2}}
+\frac{x \exp(-\frac{x^{2}}{4t})}{2\pi^{1/2}t^{1/2}}\frac{1}{|erf\left(\frac{x}{2t^{1/2}}\right)+1|}\le \frac{1}{2}
\end{align}
Then:
\begin{enumerate}
\item At $x=0$ for any finite t (3.37) becomes $\frac{1}{\pi}\le \frac{1}{2}$.
\item At any finite $x$ in the limit as $t\rightarrow\infty$ (3.37) becomes $\frac{1}{\pi}\le \frac{1}{2}$.
\end{enumerate}
which is clearly satisfied.
\end{lem}
\end{proof}
A consequence or corollary of the Li-Yau estimate is the parabolic Harnack inequality (PHI) for the heat equation.
\begin{thm}(\underline{parabolic Harnack inequality})
Let $(\mathbf{M},\bm{g})$ be a Riemannian manifold with nonvanishing Ricci curvarture.
Let $u(x,t)$ be a solution of the heat equation ${\square}u(x,t)
=(\frac{\partial}{\partial t}-{\Delta})u(x,t)=0$, for all $x\in{\mathbf{M}}$ and $t>0$. The PHI is then either of the following
\begin{align}
&u(y,t_{2})\ge u(x,t_{1})\left|\frac{t_{1}}{t_{2}}\right|^{n/2}
\exp\left(-{|d(x,y)|^{2}}/{4|t_{2}-t_{1}|}\right)\nonumber\\&
u(x,t_{1})\ge u(y,t_{2}) \left|\frac{t_{2}}{t_{1}}\right|^{n/2}
\exp\left({|d(x,y)|^{2}}/{4|t_{2}-t_{1}|}\right)
\end{align}
for all $(x,y)\in{\mathbf{M}}$ and $0<t_{1}<t_{2}$, and where $d(x,y)$ is
the distance between $x$ and $y$. On Euclidean space ${\mathbf{M}}=\bm{\mathrm{R}}^{m}$ with zero curvature $d(x,y)=|x-y|$ and
\begin{align}
&u(y,t_{2})\ge u(x,t_{1}) \left|\frac{t_{1}}{t_{2}}\right|^{n/2}
\exp(-{\|x-y\|^{2}}/{4|t_{2}-t_{1}|})\nonumber\\&
u(x,t_{1})\ge u(y,t_{2}) \left|\frac{t_{2}}{t_{1}}\right|^{n/2}
\exp({\|x-y|^{2}}/{4|t_{2}-t_{1}\|})
\end{align}
\end{thm}
\begin{proof}
The estimate can be established using a curve or geodesic $\gamma(t)$ joining
$\gamma(t_{1})=x$ to $\gamma(t_{2})=y$ with $|d\gamma(t)/dt|=d(x,y)$ or
$|d\gamma(t)/dt|=|x-y|$. Therefore $\gamma:(t_{1},t_{2})\rightarrow\bm{\mathrm{R}}^{n}$.Now
the Li-Yau inequality in logarithmic form is
\begin{align}
|\nabla\log u(x,t)|^{2}-|\frac{\partial}{\partial t}u(x,t)|+\frac{n}{2t} \ge 0
\end{align}
which also implies that $ |\frac{\partial}{\partial t}u(x,t)|\ge -\frac{n}{2t}$ so that $\psi$ does not
decrease too quickly. Recall Young's inequality such that $|ab|\le
\tfrac{1}{p}|a|^{p}+\tfrac{1}{q}|b|^{q}$ with $\tfrac{1}{p}+\tfrac{1}{q}$. For $p=q=2$
$|ab|\le \tfrac{1}{2}|a|^{2}+\tfrac{1}{2}|b|^{2}$ or
$|ab|\ge -\tfrac{1}{2}|a|^{2}+\tfrac{1}{2}|b|^{2}$. Then
\begin{align}
&\frac{d}{dt}\log u=\frac{\partial}{\partial t}\log u
+\left\langle{\nabla} \log u,\frac{d\gamma(t)}{dt}\right\rangle\equiv\frac{\partial}{\partial t}\log u
+{\nabla} \log u\frac{d\gamma(t)}{dt}\nonumber\\&
=\frac{\partial}{\partial t}\log u+\left(\sqrt{2}\nabla \log u\right)
\left(\tfrac{1}{\sqrt{2}}\frac{d\gamma(t)}{dt}\right)\nonumber\\&
\frac{\partial}{\partial t}\log u-\left(\frac{1}{2}\left|\sqrt{2}\nabla\log u\right|^{2}
+\frac{1}{2}\left|\frac{1}{\sqrt{2}}\frac{d\gamma(t)}{dt}\right|^{2}\right)\nonumber\\&
=\frac{\partial}{\partial t}\log u-\left(|\nabla\log u|^{2}+\frac{1}{4})\left|\frac{d\gamma(t)}{dt}\right|^{2}\right)
\nonumber\\&
=\left(|\nabla \log u|^{2}-\frac{n}{2t}\right)-\left(|\nabla\log u|^{2}
+\frac{1}{4}\left|\frac{d\gamma(t)}{dt}\right|^{2}\right)=-\frac{n}{2t}-\frac{1}{4}
\left|\frac{d\gamma(t)}{dt}\right|^{2}
\end{align}
Integrating from $\gamma(t_{1})=x_{1}$ to $\gamma(t_{2})=x_{2}$
\begin{align}
\int_{x}^{y}d\log u(\gamma(t),t)\ge -\frac{n}{2}\int_{t_{1}}^{t_{2}}\frac{dt}{t}-\frac{1}{4}
\int_{t_{1}}^{t_{2}}\left|\frac{d\gamma(t)}{dt}\right|^{2}dt
\end{align}
so that
\begin{align}
\log\left|\frac{u(y,t_{2})}{u(x,t_{1})}\right|\ge
-\frac{n}{2}\log\left|\frac{t_{2}}{t_{1}}\right|-\frac{1}{4}
\int_{t_{1}}^{t_{2}}\left|\frac{d\gamma(t)}{dt}\right|^{2}dt
\end{align}
The inequality can be optimized by choosing the straight line from $x$ to $y$ and the result
then follows.
\end{proof}
\begin{cor}
Let $u(x,t)$ and $u(y,t)$ satisfy the heat equations
\begin{align}
&{\square}u(x,t)=\left(\frac{\partial}{\partial t}-{\Delta}_{x}\right)u(x,t)=0,x\in\bm{\mathrm{R}}^{n},t>0\\&
{\square}u(y,t)=\left(\frac{\partial}{\partial t}-{\Delta}_{y}\right)u(y,t)=0,x\in\bm{\mathrm{R}}^{n},t>0
\end{align}
with initial Cauchy data $u(x,0)=\lambda$ and $u(y,))=\lambda$ and  Laplacians
${\Delta}_{x}=\tfrac{\partial^{2}}{\partial x_{i}\partial x^{i}}$ and
${\Delta}_{y}=\tfrac{\partial^{2}}{\partial y_{i}\partial y^{i}}$
The (unique) solutions are
\begin{align}
&u(x,t)=\frac{\lambda}{(4\pi t)^{n/2}}\int_{\bm{\mathbf{Q}}}
\exp\left(-\frac{|x-z|^{2}}{4t}\right)d^{n}z\equiv \lambda\int_{{\mathbf{Q}}}h(x,z,t)d^{n}z
\\&
u(y,t)=\frac{\lambda}{(4\pi t)^{n/2}}\int_{\bm{\mathbf{Q}}}
\exp\left(-\frac{|y-z|^{2}}{4t}\right)d^{n}z\equiv\lambda \int_{{\mathbf{Q}}}h(y,z,t)d^{n}z
\end{align}
Then
\begin{align}
\frac{\int_{\bm{\mathrm{R}}^{n}}h(x,z,t_{1})d^{n}z}{
\int_{\bm{\mathrm{R}}^{n}}h(y,z,t_{2})d^{n}z} \ge\left|\frac{t_{1}}{t_{2}}\right|^{n/2}
\exp\left(-\frac{|x-y|^{2}}{4|t_{2}-t_{1}|}\right)
\end{align}
On $\bm{\mathrm{R}}^{+}$ with $n=1$ this becomes
\begin{align}
\frac{\int_{0}^{\infty}h(x,z,t_{1})d^{n}z}{
\int_{0}^{\infty}h(y,z,t_{2})d^{n}z} \ge\left|\frac{t_{1}}{t_{2}}\right|^{1/2}
\exp\left(-\frac{|x-y|^{2}}{4|t_{2}-t_{1}|}\right)
\end{align}
which is
\begin{align}
\frac{(erf(y/2\sqrt{t_{2}})+1)}{(erf(x/2\sqrt{t_{2}})+1)}
\ge\left|\frac{t_{1}}{t_{2}}\right|^{1/2}
\exp\left(-\frac{|x-y|^{2}}{4|t_{2}-t_{1}|}\right)
\end{align}
or equivalently
\begin{align}
(erf(y/2\sqrt{t_{2}})+1)\ge\left|\frac{t_{1}}{t_{2}}\right|^{1/2}
\exp\left(-\frac{|x-y|^{2}}{4|t_{2}-t_{1}|}\right){(erf(x/2\sqrt{t_{2}})+1)}
\end{align}
Letting $y\rightarrow\infty$ gives $2>0$.
\end{cor}
\raggedbottom
\section{The Cole-Hopf transform to the heat equation}
Solutions of the CIVP for the linear heat equation also enable solutions to be generated for certain classes of quasilinear parabolic PDEs including the Burger's equation. This is done via the Cole-Hopf transform $\mathbf{[44]}$.
\begin{thm}
Let $A>0$ and $B\in\bm{\mathrm{R}}$ and consider the Cauchy problem for the following quasi-linear parabolic PDE
\begin{align}
&\left(\frac{\partial}{\partial t}-a{\Delta}\right) \psi(x,t)+b|{\nabla}\psi(x,t)|^{2}=0,~x\in {\mathbf{R}}^{n},t>0\\&
\psi(x,0)=I(x),x\in{\mathbf{R}}^{n},t=0
\end{align}
Using the Cole-Hopf transform
\begin{equation}
u(x,t)=\exp\left(-\left(\frac{b}{a}\right)\psi(x,t)\right),~~
u(x,0)=\exp\left(-\left(\frac{b}{a}\right)\psi(x,0)\right)\equiv \exp\left(-\left(\frac{b}{a}\right)I(x)\right)
\end{equation}
the function $u(x,t)$ immediately satisfies the global Cauchy initial value problem for the linear heat equation (with 'conductance' $A$) such that
\begin{align}
&\left(\frac{\partial}{\partial t}-a{\Delta}\right)u(x,t),~x\in {\mathbf{R}}^{n},t>0\\&
u(x,0)=\exp\left(-\frac{b}{a}I(x)\right)\equiv \phi(x),~~x\in\bm{\mathrm{R}}^{n},t=0
\end{align}
This has the solution
\begin{align}
u(x,t)&=\frac{1}{(4\pi t)^{n/2}}\int_{{\mathbf{R}}^{n}}\exp\left(-\frac{\|x-y\|^{2}}{4t}\right)\exp\left(-\left(\frac{B}{A}
\right)I(x)\right)\nonumber\\&\equiv,\frac{1}{(4\pi t)^{n/2}}\int_{\mathbf{R}^{n}}\exp\left(-\frac{\|x-y\|^{2}}{4t}\right)
\phi(x),~x\in\mathbf{R}^{n},t>0
\end{align}
It follows from (4.3) that the solution to the CIVP for the quasilinear parabolic system (4.1) is then
\begin{align}
\psi(x,t)=-\left(\frac{a}{b}\right)\log\left(\frac{1}{4\pi a t)^{n/2}}\int_{\mathrm{R}^{n}}\exp\left(-\frac{|x-y|^{2}}{4 a t}\right)\exp\left(-\left(\frac{b}{a}\right)I(y)\right)d\mu_{n}(y)\right)
\end{align}
\end{thm}
\begin{proof}
Details are given in Evans $\mathbf{[44]}$
\end{proof}
As a corollary, the solution to the nonlinear Burgers equation can be found which has been studied extensively in relation to turbulence and other problems $\mathbf{[54]}$. The PPDE contains a nonlinear advection term associated with conservation laws and also the diffusion term associated with the heat equation. These two effects tend to be in conflict: the nonlinearity steepens fronts into discontinuous shocks while diffusion tends to smooth them
out. Again, a solution can be found via the Cole-Hopf transform the CIVP for the heat equation.
\begin{cor}
Let $n=1$ and $A>0$. The CIVP for the Burgers equation is
\begin{align}
&\left(\frac{\partial}{\partial t}-a{\partial^{2}}{\partial x}^{2}\right)\psi(x,t)+|\psi(x,t)\frac{\partial}{\partial x}\psi(x,t)|=0,~x\in\mathbf{R},t>0\\&
\psi(x,0)=I(x),~x\in\mathbf{R},t=0
\end{align}
Setting
\begin{align}
&\Xi(x,t)=\int_{\infty}^{x}\psi(y,t)dy,~~\psi(x,t)=\frac{\partial}{\partial x}\Xi(,t)\\&
\Xi(x,0)=J(x)=\int_{\infty}^{x}\psi(x,0)dy\equiv\int_{\infty}^{x}I(y)dy
\end{align}
then gives the quasilinear parabolic PDE
\begin{align}
\left(\frac{\partial}{\partial t}-a\frac{\partial^{2}}{\partial x^{2}}\right)\Xi(x,t)+\frac{1}{2}\left|\frac{\partial}{\partial x}\Xi(x,t)\right|^{2}=0
\end{align}
which is (4.1) with $n=1$ and $b=\tfrac{1}{2}$. The solution is given by (4.7) as
\begin{align}
&\Xi(x,t)=-2 a\log(u(x,t))\nonumber\\&=-2a\log\left((4\pi a t)^{1/2}\right)\int_{\bm{\mathrm{R}}}
\exp\left(-\frac{|x-y|^{2}}{4t}\right)\exp\left(-\frac{J(x)}{2a})dy\right)
\end{align}
where $u(x,t)$ solves the CIVP for the heat equation on $\bm{\mathrm{R}}$ such that
\begin{align}
&\left(\frac{\partial}{\partial_{x}}-a \frac{\partial}{\partial x^{2}}\right)u(x,t)=0,~x\in\bm{\mathrm{R}},t>0\\&
u(x,0)=\exp\left(-\frac{J(x)}{2a}\right)\equiv\phi(x),~x\in\bm{\mathrm{R}},t=0
\end{align}
Now since $u(x,t)=\frac{\partial}{\partial x}\Xi(x,t)$
\begin{align}
u(x,t)&=\frac{\partial}{\partial x}\Xi(x,t)=-2A \frac{\partial}{\partial x}\log(u(x,t))\nonumber\\&=-2 a\frac{\partial}{\partial_{x}}\log\left({(4\pi A t)^{-1/2}}\int_{\bm{\mathrm{R}}}\exp\left(-\frac{|x-y|^{2}}{4t}\right)
\exp\left(-\frac{J(x)}{2a}\right)dy\right)\nonumber\\&
=\frac{\int_{\bm{\mathrm{R}}}\frac{|x-y|}{t}\exp\left(-\frac{|x-y|^{2}}{4t}\right)
\exp\left(-\frac{J(x)}{2a}\right)dy}{\int_{\bm{\mathrm{R}}}
\exp\left(-\frac{|x-y|^{2}}{4t}\right))\exp\left(-\frac{J(x)}{2a}\right)dy}\nonumber\\&
=\frac{\int_{-\infty}^{\infty}\frac{|x-y|}{t}\exp\left(-\frac{|x-y|^{2}}{4t}\right)
\exp\left(-\frac{J(x)}{2a}\right)dy}{\int_{-\infty}^{\infty}\exp\left(-\frac{|x-y|^{2}}{4t}\right)
\exp\left(-\frac{J(x)}{2a}\right)dy}
\end{align}
\end{cor}
It may then be possible to extend this formalism to incorporate random or stochastic initial data.
\subsection{Greens functions and equilibrium}
Given a heat flow problem on a bounded domain $\bm{\mathbf{Q}}\subset\bm{\mathrm{R}}^{n}$, one would expect a convergence to equilibrium as $t\uparrow\infty$ so that
$\tfrac{\partial}{\partial t}u(x,t)=0$, $u(x,t)\rightarrow u(x)$ and ${\Delta}_{x} u(x)=0$. For a heat source £$f(x)$ one would expect convergence to a Poisson equation $\Delta u(x)=f(x)$. This turns out to be true only when $n\ge 3$. This can be established via the Greens function for the heat kernel.
\begin{defn}
Given a heat kernel $h(x-y,t)$ on ${\mathbf{M}}
=\bm{\mathrm{R}}^{n}$, for all $(x,y)\in {\mathbf{R}}^{n}, t\in(0,\infty)$ with $x\ne y$ then the Greens function $g(x-y)$ is defined as
\begin{equation}
g(x-y)=\int_{0}^{\infty}h(x-y,t)dt<\infty
\end{equation}
and the integral must converge.
\end{defn}
\begin{lem}
When $n\ge 3$ then $g(x-y)=\int_{0}^{\infty}h(x-y,t)dt$ is equivalent to the fundamental solution of the Laplace equation as $t\rightarrow\infty$ so that $\Delta g(x-y)=-\delta^{3}(x-y)$. Then
\begin{align}
g(x-y)=\int_{0}^{\infty}h(x-y,t)dt=
\int_{0}^{\infty}\frac{\exp\left(-\|x-y\|^{2}/4t\right)}{(4\pi t)^{n/2}}dt=
\frac{\Gamma \left(\frac{n}{2}-1\right)}{4\pi^{n/2}\|x-y\|^{n-2}}
\end{align}
If $n=3$ then $\Gamma(\frac{1}{2})=\pi^{1/2}$ and
\begin{align}
g(x-y)=\frac{1}{4\pi \|x-y\|}
\end{align}
When $n<2$, the integral $\int_{0}^{\infty}h(x-y,t)dt$ does not converge.
\end{lem}
\begin{proof}
First, let $n\ge 3$ and let $\vartheta=t^{-1}$ then
\begin{align}
g(x-y)&=\int_{0}^{\infty}\frac{\exp\left(-\frac{\|x-y|^{2}}{4t}\right)}{(4 \pi t)^{n/2}} dt\equiv \int_{0}^{\infty}\frac{\exp\left(-\frac{\vartheta\|x-y\|^{2}}{4}\right)}{ (4\pi)^{n/2}}
\vartheta^{n/2}\frac{d\vartheta}{\vartheta^{2}}\\&\equiv \int_{0}^{\infty}\frac{\exp\left(-\frac{\vartheta\|x-y\|^{2}}{4}\right)}{ (4\pi)^{n/2}}
\vartheta^{\frac{n}{2}-2}d\vartheta
\end{align}
Setting $\kappa=(\vartheta\|x-y\|^{2})/4$ then gives
\begin{align}
&g(x-y)=\frac{1}{(4\pi)^{n/2}}\int_{0}^{\infty}\exp(-\kappa)
\left(\frac{4\kappa}{\|x-y\|^{2}}\right)^{\frac{n}{2}-2} \frac{4}{\|x-y\|^{2}}d\kappa\nonumber\\&
\equiv \frac{1}{(4\pi)^{n/2}}\left(\frac{4\xi}{\|x-y\|^{2}}\right)^{\frac{n}{2}-1}
\int_{0}^{\infty}\exp(-\kappa)\kappa^{\frac{n}{2}-2}d\xi =\frac{1}{(4\pi)^{n/2}}\Gamma\left(\frac{n}{2}-1\right)
\left(\frac{4\kappa}{\|x-y\|^{2}}\right)^{\frac{n}{2}-1}\nonumber\\&
\equiv\frac{\Gamma \left(\frac{n}{2}-1\right)}{4\pi^{n/2}\|x-y\|^{n-2}}
\end{align}
which gives (1.71) when $n=3$. However, for $n\le2$ the integral diverges logarithmically.
\end{proof}
We have now established the following theorem.
\begin{thm}
Let $\mathbf{Q}\subset\mathbf{{R}}^{n}$ be a compact bounded domain and let
$f\in C^{2}(\mathbf{Q})$ and $u\in C^{2}(\mathbf{Q})$ be differentiable
functions. Let $u(x,t)$ solve the Cauchy IVP for the heat equation on ${\mathbf{Q}}$ such that ${\square}u(x,t)=(\tfrac{\partial}{\partial t}-{\Delta}_{x})u(x,t)=f(x),~x\in{\mathbf{Q}},t>0$ and $u(x,0)=u(x),~x\in
\mathbf{Q}, t=0$ with the boundary condition $u(x,t)=u(x), x\in\partial{\mathbf{Q}}$. If $n\ge 3$ then in the limit $t\rightarrow\infty$, the system converges to (thermal) equilibrium so that $\tfrac{\partial}{\partial t}u(x,t)=0$. The equilibrium scenario is then described by the Poisson equation with Dirichlet boundary conditions so that
\begin{align}
&{\Delta} u(x)=f(x),~x \in {\mathbf{Q}},t>0\\&
u(x)=\phi(x),~ x\in\partial {\mathbf{Q}},t=0
\end{align}
If $\mathbf{{Q}}=\mathbf{{B}}_{R}(0)$, an Euclidean ball of radius $R$, center zero, then the solution of (4.23) is the Poisson integral formula so that
\begin{align}
&u(x)=\int_{\mathbf{{B}}_{R}(0)}g(x-y)f(y)d^{n}y+
\frac{R^{2}-\|x\|^{2}}{area(\partial\mathbf{{B}}_{1}(0))R}
\int_{\partial\mathbf{{B}}_{R}(0)}\frac{\phi(y)d^{n-1}y}{\|x-y\|^{n}}\\&
\equiv\int_{\mathbf{{B}}_{R}(0)}\int_{0}^{\infty}h(x-y,t)f(y)dtd^{n}y\nonumber\\&+
\frac{R^{2}-\|x\|^{2}}{area(\partial\mathbf{{B}}_{1}(0))R}
\int_{\partial\mathbf{{B}}_{R}(0)}\frac{\phi(y)d^{n-1}y}{\|x-y\|^{n}},~~x\in
\bm{{B}}_{R}(0),y\in\partial \mathbf{{B}}_{R}(0)
\end{align}
If $f(x)=0$ then $u(x)$ satisfies the Laplace equation ${\Delta}u(x)=0$ and
\begin{align}
u(x)=\frac{R^{2}-\|x\|^{2}}{area(\partial{\mathbf{B}}_{1}(0))R}
\int_{\partial{\mathbf{B}}_{R}(0)}\frac{\phi(y)d^{n-1}y}{\|x-y\|^{n}}
\end{align}
\end{thm}
In Section (7), this equilibrium problem is considered in relation to random boundary conditions.
\section{Randomly perturbed solutions-the Cauchy problem with random field initial data}
In this section, the following stochastic Cauchy initial-value problem is studied for the parabolic heat equation on a domain $\bm{\mathbf{Q}}\subset{\mathbf{R}}^{n}$ with randomised initial data.
\begin{align}
{\square}\widehat{u(x,t)}&\equiv \left(\frac{\partial}{\partial t}-\bm{\Delta}_{x}\right)\widehat{u(x,t)}=0,~x\in {\mathbf{Q}},
t> 0\nonumber\\&
\widehat{u(x,0)}=\phi(x)+\bm{\mathscr{J}}(x),~x\in {\mathbf{Q}},t=0\nonumber
\end{align}
where $\phi(x)\in C^{\infty}({\mathbf{Q}})$ is the initial data/function and $\mathscr{J}(x)$ is a classical Gaussian random scaler field such that $\mathbb{E}\llbracket{\mathscr{J}}(x)\rrbracket=0$ with a regulated covariance $\mathbb{E}\llbracket \mathscr{J}(x)\otimes\mathscr{J}(y)\rrbracket=\zeta J(x,y;\ell)$ for all $(x,y)\in{\mathbf{Q}}$. The correlation length is $\ell$ and $\mathbb{E}\llbracket \mathscr{J}(x)\otimes\mathscr{J}(x)\rrbracket=\zeta<\infty$.
The general perturbed solution $\widehat{u(x,t)}$ can be represented as stochastic convolution integral. This leads to stochastic extensions and versions of classical results for the heat equation, so far discussed; in particular, a stochastically averaged Li-Yau differential Harnack inequality as well as moments an estimates.
\subsection{Preliminary theorems and lemmas}
First, the following preliminary definitions, theorems and lemmas are given for random Gaussian fields which will be utilised throughout. Gaussian random fields (GRFS) and random perturbations are defined. (More details are given in Appendix A.)
\begin{defn}
Let $(\Omega,\mathbb{F},\bm{\mathcal{P}})$ be a probability triplet with $\omega\in\Omega$ and $x\in\bm{\mathcal{D}}\subset\mathrm{R}^{n}$. A GRSF can be defined by a map $\mathcal{M}$
such that $\mathcal{M}:(\omega,x)\rightarrow\mathscr{J}(x,\omega)$, with expected values
$\mathbb{E}\llbracket\bullet\rrbracket=\int_{\Omega}\llbracket\bullet
\rrbracket d\bm{\mathcal{P}}$. Then:
\begin{enumerate}
\item For a Gaussian random field, with Gaussian-distributed values, the expectation vanishes and only the 2-point correlation is relevant so that for all $(x,y)\in{\mathbf{Q}}$
\begin{align}
&\mathbb{E}\llbracket\mathscr{J}(x)\rrbracket=0\\&
\mathbb{E}\llbracket\mathscr{J}(x)\otimes \mathscr{J}(y)\rrbracket=\zeta J(x,y,\ell)
\end{align}
where $\zeta>0$ and $\ell$ is a correlation length with $J(x,y,\ell)\rightarrow 0$ for
$|x-y|>\ell$ and $J(x,y,\ell)\rightarrow 1$ as $x\rightarrow y$. The GRSF is regulated so that
$\mathbb{E}\llbracket\mathscr{J}(x)\otimes \mathscr{J}(x)\rrbracket=\zeta<\infty$.
\item The first and second derivatives exist such that the random vector field $\mathscr{V}_{i}
=\nabla_{i}\mathscr{J}(x)$ and the random tensor field $\mathscr{U}_{ij}(x)
=\nabla_{i}\nabla_{j}\mathscr{J}(x)$ are also Gaussian and have the statistical properties
\begin{align}
&\mathbb{E}\llbracket\mathscr{V}_{i}(x)\rrbracket=0\\&
\mathbb{E}\llbracket\mathscr{V}_{i}(x)\otimes \mathscr{V}_{j}(y)\rrbracket=
\sigma\delta_{ij} S(x,y,\ell)
\end{align}
and
\begin{align}
&\mathbb{E}\llbracket\mathscr{U}_{ij}(x)\rrbracket=0\\&
\mathbb{E}\llbracket\mathscr{U}_{ij}(x)\otimes \mathscr{U}_{kl}(y)\rrbracket
=\chi(\delta_{ij}\delta_{kl}+\delta_{ik}\delta_{jl}+\delta_{il}\delta_{jk})
J(x,y,\ell)
\end{align}
\item The cross-correlations vanish so that $
\mathbb{E}\llbracket \mathscr{J}(x)\otimes\nabla_{i} \mathscr{J}(x)
\rrbracket=0$ and $\mathbb{E}\llbracket\nabla_{i}\mathscr{J}(x)\otimes
\nabla_{j}\nabla_{k}\mathscr{J}(x)\rrbracket=0$, and so on.
\item The integrals $\int_{\mathbf{Q}}|\mathscr{J}(x)|^{p}d\mu_{n}(x)$ and
$\int_{\partial\mathbf{Q}}|\mathscr{J}(x)|^{p}d^{n-1}x $ exist for all $p\ge 1$.(See below).
\item If $\psi(x)$ is a smooth and continuous deterministic scalar field or function then additive and multiplicative random perturbations of $\psi(x)$ are defined as
\begin{align}
&\widehat{\psi(x)}=\psi(x)+\alpha\mathscr{J}(x)\\&
\widehat{\psi(x)}_{*}=\psi(x)\otimes\mathscr{J}(x)
\end{align}
for some $\alpha\ge 1$, so that $\widehat{\psi(x)}$ and $\widehat{\psi(x)}_{*}$ are also GRSFs.
\end{enumerate}
\end{defn}
\begin{lem}($\underline{stochastic~Jensen~inequality}$)
Let $\mathscr{J}(x)$ be a GRSF or random function existing for all $x\in\bm{\mathrm{R}}^{n}$ and let $\Phi(\mathscr{J}(x))$ be a function of $\mathscr{J}$. Then
\begin{align}
&\mathbb{E}\llbracket\Phi(\mathscr{J}(x))\rrbracket\ge
\Phi(\mathbb{E}\llbracket\mathscr{J}(x)\rrbracket),~~\Phi~convex\\&
\mathbb{E}\llbracket\Phi(\mathscr{J}(x))\rrbracket\le
\Phi(\mathbb{E}\llbracket \mathscr{J}(x)\rrbracket),~~\Phi~concave
\end{align}
\end{lem}
\begin{thm}($\underline{Fubini~Thm}$)
If $\mathscr{J}(x)$ is a GRSF then the expectation of an integral of a function
$\Psi(\mathscr{J}(x))$ can be taken under the integral so that
\begin{align}
&\mathbb{E}\left\llbracket\int_{\bm{R}^{n}}\Phi(\mathscr{J}(x))d\mu_{n}(x)\right\rrbracket
=\int_{\mathbf{R}^{n}}\mathbb{E}\llbracket\Phi(\mathscr{J}(x))d\mu_{n}(x)\rrbracket\ge
\int_{\mathbf{R}^{n}}\Phi(\mathbb{E}\left\llbracket|\mathscr{J}(x)\right\rrbracket d\mu_{n}(y),~\Phi~convex\\&
\mathbb{E}\left\llbracket\int_{\bm{R}^{n}}\Phi(\mathscr{J}(x))d\mu_{n}(x)\right\rrbracket
=\int_{\mathbf{R}^{n}}\mathbb{E}\llbracket\Phi(\mathscr{J}(x))d\mu_{n}(x)\rrbracket\le
\int_{\mathbf{R}^{n}}\Phi(\mathbb{E}\left\llbracket|\mathscr{J}(x)\right\rrbracket d\mu_{n}(x),~\Phi~concave
\end{align}
If $x_{\xi}\in {\mathbf{Q}}_{\xi}$ and ${\mathbf{Q}}=\bigcup_{\xi=1}^{M}{\mathbf{Q}}_{\xi}$ is a partition of ${\mathbf{Q}}$ then the Fubini Theorem for sums is
\begin{align}
\mathbb{E}\left\llbracket\sum_{i=1}^{M}\Phi(\mathscr{J}(x_{\xi}))d\mu_{n}(x)\right\rrbracket
=\sum_{i=1}^{M}\mathbb{E}\llbracket\Phi(\mathscr{J}(x_{\xi}))\rrbracket
\end{align}
\end{thm}
\begin{lem}(\underline{Holder inequality for sums})
Let $(x_{\xi})_{\xi=1}^{M}$ be a set of discrete points in ${\mathbf{Q}}\subset\mathrm{R}^{n}$ such that $x_{\xi}\in{\mathbf{Q}}$ for all $\xi=1...M$. Let $\psi:\mathbf{Q}\rightarrow\mathrm{R}$ be a deterministic field or function which is smooth and continuous let $\mathscr{J}(x)$ be a GRSF $\mathscr{J}\times\Omega:
\rightarrow \bm{\mathrm{R}}^{n}$. The fields have values $\psi(x_{\xi})$ and $\mathscr{J}(x_{\xi})$ at $x_{\xi}$. If $\tfrac{1}{p}+\tfrac{1}{q}=1$ then the stochastic Holder inequality for sums is then
\begin{align}
\mathbb{E}\left\llbracket\sum_{\xi=1}^{M}\psi(x_{\xi})\otimes \mathscr{J}(x_{\xi})\right\rrbracket\le
\left(\sum_{\xi=1}^{M}|\psi(x_{\xi}|^{q}\right)^{1/q}\left(\sum_{\xi=1}^{M}
\mathbb{E}\llbracket |\mathscr{J}(x_{\xi})|^{p}\rrbracket\right)^{1/p}
\end{align}
\end{lem}
\begin{proof}
Define the norms
\begin{align}
&\|\psi(\bullet_{\xi})\|_{L_{p}^{\Sigma}(\mathbf{Q})}=\left(\sum_{\xi=1}^{M}|\psi(x_{\xi})|^{p}\right)^{1/p}
\nonumber\\&
\|\mathscr{J}(\bullet_{\xi})\|_{L_{p}^{\Sigma}(\mathbf{Q})}
=\left(\sum_{\xi=1}^{M}\mathbb{E}\llbracket|\mathscr{J}(x_{\xi})|^{p}\rrbracket \right)^{1/p}
\end{align}
If $Y(x_{\xi})$ and $\mathscr{Z}(x_{\xi})$ are generic deterministic and random scalar fields or functions respectively then
\begin{align}
\mathbb{E}\llbracket |Y(x_{\xi})\mathscr{Z}(x_{\xi})|\rrbracket\
\le \frac{|Y(x)|^{q}}{q}+\frac{\mathbb{E}\llbracket|\mathscr{Z}(x)|^{p}\rrbracket}{p}
\end{align}
Set
\begin{align}
& Y(x_{\xi})=\frac{\psi(x_{\xi})}{\|\psi(\bullet_{\xi})\|_{L_{p}^{\Sigma}(\mathbf{Q})}},~~~~
\mathscr{Z}(x_{\xi})=\frac{\mathscr{J}(x_{\xi})}{\|\mathscr{J}(\bullet_{\xi})\|_{L_{p}^{\Sigma}(\mathbf{Q})}}
\end{align}
Then
\begin{align}
&\mathbb{E}\left\llbracket \sum_{\xi=1}^{M}Y(x_{\xi}\mathscr{Z}(x_{xi})\right\rrbracket
=\mathbb{E}\left\llbracket \sum_{\xi=1}^{M} \frac{\psi(x_{\xi})}{\|\psi(\bullet_{\xi})\|_{L_{p}^{\Sigma}(\mathbf{Q})}}
\frac{\mathscr{J}(x_{\xi})}{\|\mathscr{J}(\bullet_{\xi})\|_{L_{p}^{\Sigma}(\mathbf{Q})}}
\right\rrbracket\nonumber\\&
\le \mathbb{E}\left\llbracket \sum_{\xi=1}^{M}
\frac{|\psi(x_{\xi})|^{q}}{\|q\psi(\bullet_{\xi})\|_{L_{p}^{\Sigma}(\mathbf{Q})}}
+\sum_{\xi=1}^{M}|\frac{\mathscr{J}(x_{\xi})|^{p}}{p\|
\mathscr{J}(\bullet_{\xi})\|_{L_{p}^{\Sigma}(\mathbf{Q})}}\right\rrbracket\nonumber\\&
= \sum_{\xi=1}^{M}
\frac{|\psi(x_{\xi})|^{q}}{\|q\psi(\bullet_{\xi})\|_{L_{p}^{\Sigma}(\mathbf{Q})}}
+\mathbb{E}\left\llbracket\sum_{\xi=1}^{M}|\frac{|\mathscr{J}(x_{\xi})|^{p}}{p\|
\mathscr{J}(\bullet_{\xi})\|_{L_{p}^{\Sigma}(\mathbf{Q})}}\right\rrbracket\nonumber\\&
=\frac{(\|\psi(\bullet_{\xi})\|_{L_{p}^{\Sigma}(\mathbf{Q})})^{q}}
{q\left(\|\psi(\bullet_{\xi})\|_{L_{p}^{\Sigma}(\mathbf{Q})}\right)^{q}}
+\frac{(\|\mathscr{J}(\bullet_{\xi})\|_{L_{p}^{\Sigma}(\mathbf{Q})})^{p}}
{p\left(\|\mathscr{J}(\bullet_{\xi})\|_{L_{p}^{\Sigma}(\mathbf{Q})}\right)^{p}}
=\frac{1}{q}+\frac{1}{p}=1
\end{align}
and so (5.14) follows.
\end{proof}
The following preliminary lemma will be required, which is a Holders integral inequality or estimate.
\begin{lem}
Let ${\mathbf{Q}}\subset {\mathbf{R}}^{n}$ and let $\psi:\mathbf{Q}\rightarrow
{\mathbf{R}}$ be a smooth, continuous and deterministic function. Let
$\mathscr{J}(x)$ be a regulated GRF with $\mathbb{E}\llbracket
\mathscr{J}(x)\rrbracket=0$ and $\mathbb{E}\llbracket\mathscr{J}(x)\otimes
\mathscr{J}(x)\rrbracket=\xi$. The integrals $\int_{\mathbf{Q}}\mathscr{J}(x)d\mu_{n}(x)$
and $\int_{\mathbf{Q}}|\mathscr{J}(x)|^{p}d\mu_{n}(x) $ exist and well defined as
Riemann sums over a partition of $\bm{\mathbf{Q}}$. If $\tfrac{1}{p}+\tfrac{1}{q}=1$
then the Holder inequality is
\begin{align}
&\mathbb{E}\left\llbracket\int_{\mathbf{Q}}\psi(x)\otimes\mathscr{J}(x)
d\mu_{n}(x)\right\rrbracket
\le\left(\int_{\mathbf{Q}}|\psi(x)|^{q}d\mu_{n}(x)\right)^{1/q}
\left(\int_{\mathbf{Q}}\mathbb{E}|
\mathscr{J}(x)|^{p}d\mu_{n}(x)\right)^{1/p}\nonumber\\&
=\frac{1}{2}[\zeta^{p/2}+(-1)^{p}\zeta^{p/2}]v(\bm{\mathbf{Q}})
\left(\int_{\mathbf{Q}}|\psi(x)|^{p/p-1}d\mu_{n}(x)\right)^{(p-1)/p}
\end{align}
where $V(\mathbf{Q})=\int_{\mathbf{Q}}d\mu_{n}(x) $.
\end{lem}
\begin{proof}
Define the $L_{p}$ and $\mathscr{L}_{p}$ norms
\begin{align}
&\|\psi(\bullet)\|_{L_{p}}=\left(\int_{\mathbf{Q}}|\psi(x)|^{p}d\mu_{n}(x)\right)^{1/p}\\&
\|\mathscr{J}(\bullet)\|_{\mathscr{L}_{p}}=\left(\int_{\mathbf{Q}}
\mathbb{E}\llbracket|
\mathscr{J}(x)|^{p}\rrbracket d\mu_{n}(x)\right)^{1/p}
\end{align}
then one must establish that
\begin{align}
&\|\psi(\bullet)\|_{L_{p}}=\left(\int_{\mathbf{Q}}|\psi(x)|^{p}d\mu_{n}(x)\right)^{1/p}\\&
\|\mathscr{J}(\bullet)\|_{\mathscr{L}_{p}}=\left(\int_{\mathbf{Q}}
\mathbb{E}\llbracket|
\mathscr{J}(x)|^{p}\rrbracket d\mu_{n}(x)\right)^{1/p}
\end{align}
\begin{align}
&\mathbb{E}\left\llbracket\int_{\mathbf{Q}}\psi(x)\otimes\mathscr{J}(x)d\mu_{n}(x)
\right\rrbracket
\le\|\psi(\bullet)\|_{L_{q}}\circ \|\mathscr{J}(\bullet)\|_{\mathscr{L}_{p}}
\end{align}
Next, Young's inequality states that for two functions $Y(x),Z(x)$ and
$\tfrac{1}{p}+\tfrac{1}{q}=1$ one has $
Y(x)Z(x)\le \frac{|Y(x)|^{q}}{q}+\frac{|Z(x)|^{p}}{p}$. If $Z(x)$ is replaced by a GRF $\mathscr{Y}(x)$ then
\begin{align}
Y(x)\mathscr{Z}(x)\le \frac{|Y(x)|^{q}}{q}+\frac{|\mathscr{Z}(x)|^{p}}{p}
\end{align}
and the expectation is
\begin{align}
\mathbb{E}\left\llbracket Y(x)\mathscr{Z}(x)\right\rrbracket \le
\frac{|Y(x)|^{q}}{q}+\frac{\mathbb{E}
\left\llbracket|\mathscr{Z}(x)|^{p}\right\rrbracket}{p}
\end{align}
Now let
\begin{align}
Y(x)=\frac{\psi(x)}{\|\psi(\bullet)\|_{L_{q}(\mathbf{Q})}},~~~
\mathscr{Z}(x)=\frac{\mathscr{J}(x)}{\|\mathscr{J}(\bullet)\|_{\mathscr{L}_{p}(\mathbf{Q})}}
\end{align}
Equation (5.24) then becomes
\begin{align}
Y(x)\mathscr{Z}(x)\le \frac{|\psi(x)|^{q}}{q\|\psi(\bullet)\|^{q}_{L_{q}(\mathbf{Q})}} +
\frac{|\mathscr{J}(x)|^{p}}{p\|\mathscr{J}(\bullet)\|^{p}_{\mathscr{L}_{p}(\mathbf{Q})}}
\end{align}
with expectation
\begin{align}
\mathbb{E}\llbracket Y(x)\mathscr{Z}(x)\rrbracket \le \frac{|\psi(x)|^{q}}{q\|\psi(\bullet)\|^{q}_{L_{q}(\mathbf{Q})}} +
\frac{\mathbb{E}\llbracket |\mathscr{J}(x)|^{p}\rrbracket}{p\|\mathscr{J}(\bullet)\|^{p}_{\mathscr{L}_{p}(\mathbf{Q})}}
\end{align}
Integrating and using the Fubini Theorem
\begin{align}
&\int_{\mathbf{Q}}\mathbb{E}\llbracket Y(x)\mathscr{Z}(x)\rrbracket d\mu_{n}(x)
=\mathbb{E}\left\llbracket\int_{\mathbf{Q}} Y(x)\mathscr{Z}(x)d\mu_{n}(x)\right\rrbracket\nonumber\\&
\le\int_{\mathbf{Q}}\frac{|\psi(x)|^{q}}{q\|\psi(\bullet)\|^{q}_{L_{q}(\mathbf{Q})}}
\mu_{n}(x) +
\int_{\mathbf{Q}}\frac{\mathbb{E}\llbracket |\mathscr{J}(x)|^{p}\rrbracket}{p\|\mathscr{J}
(\bullet)\|^{p}_{\mathscr{L}_{p}(\mathbf{Q})}}d\mu_{n}(x)\nonumber\\&
\frac{1}{q\|\psi(\bullet)\|^{q}_{L_{q}(\mathbf{Q})}}\int_{\mathbf{Q}}|\psi(x)|^{q}d\mu_{n}(x)+
\frac{1}{p\|\mathscr{J}(\bullet)\|^{p}_{\mathscr{L}_{p}(\mathbf{Q})}}\int_{\mathbf{Q}}
\mathbb{E}\llbracket |\mathscr{J}(x)|^{p}\rrbracket d\mu_{n}(x)\nonumber\\&
\equiv \frac{\|\psi(\bullet)\|^{q}_{L_{q}(\mathbf{Q})}}{q\|\psi(\bullet)\|^{q}_{L_{q}(\mathbf{Q})}}
+\frac{\|\mathscr{J}(\bullet)\|^{p}_{\mathscr{L}_{p}(\mathbf{Q})}}{p\|
\mathscr{J}(\bullet)\|^{p}_{\mathscr{L}_{p}(\mathbf{Q})}}=\frac{1}{q}+\frac{1}{p}=1
\end{align}
Now using (5.26)
\begin{align}
\mathbb{E}\left\llbracket\int_{\mathbf{Q}}
\frac{\psi(x)}{\|\psi(\bullet)\|_{L_{q}(\mathbf{Q})}}\frac{\mathscr{J}(x)}{\|\mathscr{J}(\bullet)\|_{\mathscr{L}_{p}(\mathbf{Q})}}
d\mu_{n}(x)\right\rrbracket \le 1
\end{align}
and so the Holder inequality (5.19) is established. The stochastic integral or norm is
\begin{align}
\|\mathscr{J}(\bullet)\|_{\mathscr{L}_{p}}&=\left(\int_{\mathbf{Q}}
\mathbb{E}\llbracket|\mathscr{J}(x)|^{p}\rrbracket d\mu_{n}(x)\right)^{1/p}=
\left(\frac{1}{2}[\zeta^{p/2}+(-1)^{p}\zeta^{p/2}]\int_{{\mathbf{Q}}} d\mu_{n}(x)\right)^{1/p}\nonumber\\&
=\left(\frac{1}{2}[\zeta^{p/2}+(-1)^{p}\zeta^{p/2}]v({\mathbf{Q}})\right)^{1/p}
\end{align}
and $q=p/(p-1)$ so (5.19) is established.
\end{proof}
\begin{thm}
The moments $\mathbb{E}\llbracket |\widehat{u(x,t)}^{p}|\rrbracket$ are estimated
(with $u(x,t)=0$ as  and are smoothed out or dissipated since
\begin{align}
\lim_{t\uparrow \infty}\bm{\mathcal{M}}_{p}(x,t)&=\lim_{t\uparrow \infty}
\mathbb{E}\llbracket |\widehat{u(x,t)}^{p}|\rrbracket=0
\end{align}
\end{thm}
The moments estimate can also be deduced via the Riemann-Stieltjes sum representation
and the Holder inequality for sums the p-moments are
\begin{proof}
Let ${\mathbf{Q}}=\bigcup_{\xi=1}^{M}{\mathbf{Q}}_{\xi}$ be a partition of
${\mathbf{Q}}\subset{\mathbf{R}}^{n}$ with $\omega\in\Omega$, and
$\tfrac{1}{p}+\tfrac{1}{q}=1$. The heat kernel is $h(x-y,t)$. As a Riemann-Stieltjes sum
\begin{align}
&\mathbb{E}\llbracket|\widehat{u(x,t)}|^{p}
\rrbracket\le\mathbb{E}\left\llbracket\left|\int_{\mathbf{Q}}h(x-y,t)
\otimes{\mathscr{J}}(y)
d\mu_{n}(y)\right|^{p}\right\rrbracket\nonumber\\&=\lim_{v(\mathbf{Q}_{\xi})\uparrow
0}\lim_{M\uparrow \infty}
\int_{\Omega}\left|\sum_{\xi=1}^{M}h(x-y_{\xi},t)\otimes\mathscr{J}(y_{\xi};\omega)
v(\mathbf{Q}_{\xi})\right|^{p}d\bm{\mathrm{I\!P}}(\omega)\nonumber\\&
=\lim_{v(\mathbf{Q}_{\xi})\uparrow 0}\lim_{M\uparrow \infty}
\int_{\Omega}\left|\sum_{\xi=1}^{M}h(x-y_{\xi},t)\otimes\mathscr{J}(y_{\xi};\omega)
|v(\mathbf{Q}_{\xi})|^{\tfrac{1}{p}+\tfrac{1}{q}}\right|^{p}
d\bm{\mathrm{I\!P}}(\omega)\nonumber\\&
=\lim_{v({\mathbf{Q}}_{\xi})\uparrow 0}\lim_{M\uparrow \infty}
\int_{\Omega}\left|\sum_{\xi=1}^{M}h(x-y_{\xi},t)\otimes\mathscr{J}(y_{\xi};\omega)
|v({\mathbf{Q}}_{\xi})|^{\tfrac{1}{p}}|v({\mathbf{Q}}_{\xi})|^{\tfrac{1}{q}}|\right|^{p}
d\bm{\mathrm{I\!P}}(\omega)
\le\nonumber\\&
=\lim_{v({\mathbf{Q}}_{\xi})\uparrow 0}\lim_{M\uparrow \infty}
\int_{\Omega}\left|\underbrace{\left(\sum_{\xi=1}^{M}|h(x-y_{\xi},t)|^{q}|
v({\mathbf{Q}}_{\xi})|\right)^{1/q}\otimes\left(\sum_{\xi=1}^{M}|\mathscr{J}(y_{\xi};\omega)|^{p}|v({\mathbf{Q}}_{\xi})|
\right)^{1/p}}_{via~Holder~ineq.~for~sums}\right|^{p}d\bm{\mathrm{I\!P}}(\omega)\nonumber\\&
=\lim_{v({\mathbf{Q}}_{\xi})\uparrow 0}\lim_{M\uparrow \infty}
\int_{\Omega}\left(\sum_{\xi=1}^{M}|h(x-y_{\xi},t)|^{q}|v({\mathbf{Q}}_{\xi})|\right)^{p/q}
\otimes\left(\sum_{\xi=1}^{M}|\mathscr{J}(y_{\xi};\omega)|^{p}|v({\mathbf{Q}}_{\xi})|\right)
d\bm{\mathrm{I\!P}}(\omega)\nonumber\\&
=\lim_{v({\mathbf{Q}}_{\xi})\uparrow 0}\lim_{M\uparrow \infty}
\left(\sum_{\xi=1}^{M}|h(x-y_{\xi},t)|^{q}|v({\mathbf{Q}}_{\xi})|\right)^{p/q}
\otimes\left(\sum_{\xi=1}^{M}\int_{\Omega}|\mathscr{J}(y_{\xi};\omega)|^{p}
d\bm{\mathrm{I\!P}}(\omega)|v({\mathbf{Q}}_{\xi})|\right)\nonumber\\&
\equiv\lim_{v({\mathbf{Q}}_{\xi})\uparrow 0}\lim_{M\uparrow \infty}
\left(\sum_{\xi=1}^{M}|h(x-y_{\xi},t)|^{q}|v({\mathbf{Q}}_{\xi})|\right)^{p/q}
\otimes\left(\sum_{\xi=1}^{M}\mathbb{E}\llbracket|\mathscr{J}(y_{\xi};\omega)|^{p}\big
\rrbracket|v({\mathbf{Q}}_{\xi})|\right)\nonumber\\&
\equiv\lim_{v({\mathbf{Q}}_{\xi})\uparrow 0}\lim_{M\uparrow
\infty}\frac{1}{2}[\zeta^{p/2}+(-1)^{p}\zeta^{p/2}]
\left(\sum_{\xi=1}^{M}|h(x-y_{\xi},t)|^{q}|v(\mathbf{Q}_{\xi})|\right)^{p/q}\nonumber\\&
\equiv\lim_{v({\mathbf{Q}}_{\xi})\uparrow 0}\lim_{M\uparrow
\infty}\frac{1}{2}[\zeta^{p/2}+(-1)^{p}\zeta^{p/2}]v(\mathbf{Q})
\left(\sum_{\xi=1}^{M}|h(x-y_{\xi},t)|^{q}|v(\mathbf{Q}_{\xi})|\right)^{p/q}\nonumber\\&
\equiv\lim_{v(\mathbf{Q}_{\xi})\uparrow 0}\lim_{M\uparrow
\infty}\frac{1}{2}[\zeta^{p/2}+(-1)^{p}\zeta^{p/2}]v(\mathbf{Q})
\left(\int_{{\mathbf{Q}}}|h(x-y,t)|^{p}d\mu_{n}(y)\right)^{p/q}                                                                                                                \end{align}
which agrees exactly with (5.32)
\end{proof}
The final preliminary lemma is
\begin{lem}
Let $\Lambda:\mathbf{Q}\rightarrow {\mathbf{R}}$ and and let $p\ge 2$. If $\tfrac{p}{p-1}+\tfrac{\ell}{\ell-1}=1$ then
\begin{align}
&\int_{\mathbf{Q}}|\Lambda(x)|^{\frac{p}{p-1}}d\mu_{n}(x)\le \bigg(\int_{\mathbf{Q}}\Lambda(x)d\mu_{n}(x)\bigg)^{\frac{p}{p-1}}(v(\mathbf{Q}))^{\frac{\ell}{\ell-1}}
\bigg)\\&
\bigg(\int_{\mathbf{Q}}|\Lambda(x)|^{\frac{p}{p-1}}d\mu_{n}(x)\bigg)^{p-1}\le \bigg(\int_{\mathbf{Q}}\Lambda(x)d\mu_{n}(x)\bigg)^{p}(v(\mathbf{Q}))^{\frac{\ell(p-1)}{\ell-1}}
\bigg)
\end{align}
\end{lem}
\begin{proof}
Let $\Psi(x)=|\Lambda(x)|^{\frac{p}{p-1}}$ and let $\tfrac{p}{p-1}=\tfrac{\ell}{\ell-1}$ then
using the Holder inequality
\begin{align}
\int_{\mathbf{Q}}|\Lambda(x)|^{\frac{p}{p-1}}d\mu_{n}(x)&=\int_{\mathbf{Q}}(\Psi(x).1)
d\mu_{n}(x)\nonumber\\&
\le \left(\int|\Psi(x)|^{\frac{p-1}{p}}\right)^{\frac{p}{p-1}}
\left(\int_{\mathbf{Q}}1.d\mu_{n}(x)\right)^{\frac{\ell}{\ell-1}}
=\left(\int_{\mathbf{Q}}|\Psi(x)|^{\frac{p-1}{p}}d\mu_{n}(x)\right)^{\frac{p}{p-1}}
({v}(\mathbf{Q})^{\frac{\ell}{\ell-1}}\nonumber\\&
=\left(\int_{\mathbf{Q}}f(x)d\mu_{n}(x)\right)^{\frac{p}{p-1}}
({v}(\mathbf{Q}))^{\frac{\ell}{\ell-1}}
\end{align}
where ${v}(\mathbf{Q})=\int_{\mathbf{Q}}d\mu_{n}(x)<\infty $.
\end{proof}
\subsection{stochastic perturbations and estimates}
The main theorem for the general global stochastic Cauchy problem with random initial data is then as follows.
\begin{thm}
The global Cauchy initial value problem for the homogenous heat equation, with Gaussian randomized initial data is
\begin{align}
{\square}
\widehat{u(x,t)}&=\left(\frac{\partial}{\partial_{x}}-{\Delta}\right) \widehat{u(x,t)}=0,~(x\in\bm{\mathrm{R}}^{n},t\in (0,\infty))\\&
\widehat{u(x,0)}=\phi(x)+\mathscr{J}(x),~(x\in\bm{\mathrm{R}}^{n},t=0)
\end{align}
where $\mathscr{J}(x)$ is a Gaussian random scalar field (GRSF)existing for all $x\in\bm{\mathrm{R}}^{n}$. The GRSF has the statistical properties:
\begin{enumerate}
\item The stochastic expectation vanishes so that
    $\mathbb{E}\llbracket[\mathscr{J}(x)\rrbracket=0$. The covariance is
\begin{equation}
\mathbb{K}(\mathscr{J}(x)\otimes \mathscr{J}(y)\otimes y)=\mathbb{E}\llbracket\mathscr{J}(x)\otimes \mathscr{J}(y)\rrbracket=\zeta J(x,y;\ell)
\end{equation}
The GRF is regulated so that $\mathbb{K}\llbracket\mathscr{J}(x)\otimes \mathscr{J}(x)\rrbracket=\mathbb{E}\llbracket\mathscr{J}(x)\otimes \mathscr{J}(x)\rrbracket=\zeta<\infty.$
\item The stochastic integrals and convolution integrals exist, namely
\begin{equation}
\left|\int_{\bm{\mathrm{R}}^{n}}\mathscr{J}(x)d\mu_{n}(x)\right|,~~
\left|\int_{\bm{\mathrm{R}}^{n}}h(x-y,t)\otimes\mathscr{J}(y)
d\mu_{n}(y)\right|
\end{equation}
where $k(x-y)$ is any kernel dependent on the separation $\|x-y\|$, which includes the heat kernel such that $k(x-y)=h(x-y,t)$. Integrals over GRFs are also GRFs. Their expectations vanish by Fubini's theorem and the Gaussian property.
\begin{align}
&\mathbb{E}\left\llbracket\int_{\bm{\mathrm{R} }^{n}}\mathscr{J}(x)d\mu_{n}(x)\right\rrbracket
=\int_{\mathrm{R}^{n}}\mathbb{E}\big\llbracket\mathscr{J}(x)\big\rrbracket d\mu_{n}(x)=0\\&
\mathbb{E}\bigg\llbracket\int_{\bm{\mathrm{R}}^{n}}\mathscr{J}(y)
K(x-y)
d\mu_{n}(y)\bigg\rrbracket=\int_{\mathrm{R}^{n}}
\mathbb{E}\big\llbracket\mathscr{J}(y)\big\rrbracket
h(x-y,t)d\mu_{n}(y)=0
\end{align}
\item The initial Cauchy random data is now correlated as
\begin{align}
&\mathbb{E}\llbracket\widehat{u(z,0)}\otimes\widehat{u(z^{\prime},0)}
\rrbracket=\phi(x)\phi(y)+\mathbb{E}\llbracket
\mathscr{J}(x)\otimes\mathscr{J}(y)\rrbracket\nonumber\\&=\phi(x)\phi(y)+
\zeta J(x,y;\ell)
\end{align}
\end{enumerate}
The solution of the Cauchy problem is then
\begin{align}
\widehat{u(x,t)}&=(4\pi t)^{-n/2}\int_{\mathrm{R}^{n}} \exp\left(-\frac{\|x-y\|^{2}}{4t}\right)\widehat{\phi(y)}d\mu_{n}(y)\nonumber\\&=
(4\pi t)^{-n/2}\int_{\mathrm{R}^{n}} \exp\left(-\frac{\|x-y\|^{2}}{4t}\right)\phi(y)d\mu_{n}(y)\nonumber\\&+{(4\pi t)^{-n/2}}
\int_{\mathrm{R}^{n}}\exp\left(-\frac{\|x-y\|^{2}}{4t}\right)\mathscr{J}(y)d\mu_{n}(y)
\nonumber\\&
\equiv\int_{\mathrm{R}^{n}}h(x-y,t)\phi(y)d\mu_{n}(y)+\int_{\mathrm{R}^{n}}
h(x-y,t)\mathscr{J}(y)d\mu_{n}(y)\equiv u(x,t)+\mathscr{U}(x,t)
\end{align}
\end{thm}
\begin{proof}
The heat kernel $h(x-y,t)$ solves the heat equation so that $\square h(x-y,t)=\frac{\partial}{\partial t }h(x-y,t)-{\Delta}h(x-y,t)=0$. The heat kernel $h(x-y,t)$ is also smooth and infinitely differentiable so that integrals and derivative can be interchanged. Then from (5.44)
\begin{align}
&\frac{\partial}{\partial t}\widehat{u(x,t)}=\int_{\bm{\mathrm{R}}^{n}}\frac{\partial}{\partial t}h(x-y,t)\phi(y)d\mu_{n}(y)+
\int_{\bm{\mathrm{R}}^{n}}\frac{\partial}{\partial t}h(x-y,t)\otimes\mathscr{J}(y)d\mu_{n}(y)\\&
{\Delta}_{x}\widehat{u(x,t)}=\int_{\bm{\mathrm{R}}^{n}}\Delta_{x}h(x-y,t)
\phi(y)d\mu_{n}(y)+
\int_{\bm{\mathrm{R}}^{n}}\Delta_{x}h(x-y,t)\otimes\mathscr{J}(y)d\mu_{n}(y)
\end{align}
so that
\begin{align}
{\square} \widehat{u(x,t)}
=&\left(\frac{\partial}{\partial t}-{\Delta}_{x}\right)\widehat{u(x,t)}\nonumber\\&
=\int_{\mathrm{R}}\big[\frac{\partial}{\partial t}h(x-y,t)
+{\Delta}_{x}h(x-y,t)\big]\phi(y)d\mu_{n}(y)\nonumber\\&+\int_{\mathrm{R}}
\big[\frac{\partial}{\partial_{x}}h(x-y,t)+\Delta_{x} h(x-y,t)\big]\otimes\mathscr{J}(y)d\mu_{n}(y)\nonumber\\&
=\int_{\mathrm{R}^{n}}{\square}h(x-y,t)\phi(y)d\mu_{n}(y)+
\int_{\mathrm{R}^{n}}{\square}h(x-y,t)\mathscr{J}(y)d\mu_{n}(y)=0
\end{align}
since $\square u(x,t)=0$. Hence (5.45) solves (5.38)
\end{proof}
\begin{lem}
The stochastic average of (5.45) is $\mathbb{E}\llbracket\widehat{u(x,t)}\rrbracket=u(x,t)$ since $\mathbb{E}\llbracket\mathscr{J}(x)\rrbracket=0$.
\end{lem}
From (5.44)
\begin{proof}
\begin{align}
&\mathbb{E}\llbracket\widehat{u(x,t)}\rrbracket=\frac{1}{(4\pi t)^{n/2}}\int_{\bm{\mathrm{R}}^{n}}\exp\left(-\frac{\|x-y\|^{2}}{4t}\right)\phi(y)d\mu_{n}(y)
\nonumber\\&+\frac{1}{(4\pi t)^{n/2}}\mathbb{E}\bigg\llbracket\int_{\mathrm{R}^{n}} \exp\left(-\frac{\|x-y\|^{2}}{4t}\right)\mathscr{J}(y)d\mu_{n}(y)\bigg\rrbracket\nonumber\\&
={(4\pi t)^{-n/2}}\int_{\bm{\mathrm{R}}^{n}}\exp(\left(-\frac{\|x-y\|^{2}}{4t}\right)\phi(y)d\mu_{n}(y)\nonumber\\&+\frac{1}{(4\pi t)^{3/2}}\int_{\bm{\mathrm{R}}^{n}}\exp(\left(-\frac{\|x-y\|^{2}}{4t}\right)
\mathbb{E}\big\llbracket\mathscr{J}(y)\big\rrbracket d\mu_{n}(y)\nonumber\\&={(4\pi t)^{-n/2}}\int_{\mathrm{R}^{n}} \exp\left(-\frac{\|x-y\|^{2}}{4t}\right)\phi(y)d\mu_{n}(y)
\end{align}
In terms of a stochastic Riemann-Stieljes sum
\begin{align}
\mathbb{E}\llbracket\widehat{u(x,t)}\rrbracket&=
\int_{{\mathbf{Q}}}h(x-y,t)\phi(y)d\mu_{n}(y)\nonumber\\&+
\lim_{v({\mathbf{Q}}_{\xi})\uparrow 0}
\lim_{M\uparrow\infty}\sum_{\xi=1}^{M}\int_{\Omega}{(4\pi t)^{-n/2}}\exp\left(-\frac{|x-y|^{2}}{4t}\right)\otimes\mathscr{J}(y_{\xi};\omega)
v({\mathbf{Q}}_{\xi})
d\bm{\mathrm{I\!P}}(\omega)\nonumber\\&=\int_{{\mathbf{Q}}}h(x-y,t)\phi(y)d\mu_{n}(y)\nonumber\\&+
\lim_{v({\mathbf{Q}}_{\xi})\uparrow 0}\lim_{M\uparrow\infty}\sum_{\xi=1}^{M}{(4\pi t)^{-n/2}}\exp\left(-\frac{|x-y|^{2}}{4t}\right)\int_{\Omega}\mathscr{J}(y_{\xi};\omega)
d\bm{\mathrm{I\!P}}(\omega)v({\mathbf{Q}}_{\xi})=0
\end{align}
since $\int_{\Omega}\mathscr{J}(y_{\xi};\omega)d\bm{\mathrm{I\!P}}(\omega)=0$
\end{proof}
\begin{rem}
It is well established that a GRF $\mathscr{J}(x)$ can be 'Gaussian smoothed' on a scale
$\zeta$ with respect to a Gaussian kernel $k(x-y,\zeta)$ such that
\begin{align}
\mathscr{J}(x,\zeta)=\int_{\mathbf{Q}}k(x-y,\zeta)\otimes
\mathscr{J}(y)d\mu_{n}(y)=\int_{\mathbf{Q}}\exp\left(-\frac{\|x-y\|^{2}}{\zeta^{2}}\right)
\mathscr{J}(y)d\mu_{n}(y)
\end{align}
and this stochastic convolution integral is well defined as a Riemann sum.(Appendix A.) As an example, in cosmology, Gaussian random fields can be used to represent density
fluctuations in the early (expanding) universe $\mathbf {[30,31]}$. A GRF $\mathscr{J}(r)$ can be Gaussian smoothed on a (comoving) scale $R$ such that
\begin{align}
\mathscr{J}(r,R)=(2\pi R^{2})^{-3/2}\int
\exp\left(-\frac{\|r-r^{\prime}\|^{2}}{2R^{2}}\right)\mathscr{J}
(r^{\prime}) d^{3}r^{\prime}
\end{align}
For the solution (5.44), the heat kernel $h(x-y,t)$ is also a Gaussian and so increasingly Gaussian smooths the GRF $\mathscr{J}$ as t increases. The stochastic convolution integral gives a time-dependent random field $
\mathscr{G}(x,t)$
\begin{align}
\mathscr{G}(x,t)=\int_{\mathbf{Q}}h(x-y,t)
\mathscr{J}(y)d\mu_{n}(y)=\frac{1}{(4\pi t)^{n/2}}
\int_{\bm{\mathbf{Q}}}\exp\left(-\frac{\|x-y\|^{2}}{4t}\right)\otimes
\mathscr{J}(y)d\mu_{n}(y),~~t>0
\end{align}
which solves the SCIVP ${\square}\mathscr{G}(x,t)=
\frac{\partial}{\partial t}\mathscr{G}(x,t)-\Delta_{x}\mathscr{G}(x,t)=0$ with initial random data $\mathscr{G}(x,0)=\mathscr{G}(x)$. The heat kernel smooths the initial-value random field to zero so that
\begin{align}
\lim_{t\uparrow 0}\mathscr{G}(x,t)=\lim_{t\uparrow 0}\frac{1}{(4\pi t)^{n/2}}\int_{\mathbf{Q}}\exp\left(-\frac{\|x-y\|^{2}}{4t}\right)
\otimes\mathscr{J}d\mu_{n}(y)=0
\end{align}
\end{rem}
\begin{lem}
Since the heat equation is a linear PDE, the expectation of the randomly perturbed equation
gives the deterministic heat equation.
\begin{align}
\mathbb{E}\left\llbracket\left(\frac{\partial}{\partial t}-{\Delta}\right)\widehat{u(x,t)}\right\rrbracket=
\left(\frac{\partial}{\partial t}-{\Delta}\right)u(x,t)=0
\end{align}
\end{lem}
\begin{proof}
Since $\mathbb{E}\llbracket\widehat{u(x,t)}\rrbracket=u(x,t)$ then
\begin{align}
&\mathbb{E}\left\llbracket\left(\frac{\partial}{\partial t}-{\Delta}\right)\widehat{u(x,t)}\right\rrbracket
=\left(\frac{\partial}{\partial t}-{\Delta}\right)
\mathbb{E}\llbracket\widehat{u(x,t)}\rrbracket\nonumber\\&
=\left(\frac{\partial}{\partial t}-{\Delta}\right)u(x,t)=0
\end{align}
One can also take the derivatives since $\mathbb{E}\llbracket\frac{\partial}{\partial t}\widehat{u(x,t)}\rrbracket
=\frac{\partial}{\partial t}u(x,t)$ and $\mathbb{E}\llbracket{\Delta}_{x}\widehat{u(x,t)}\rrbracket
={\Delta}_{x}u(x,t)$
\end{proof}
\begin{lem}
Using (2.63), the randomly perturbed solution the SCIVP can be expressed in terms of the
eigenvalues and eigenfunctions of the Laplacian such that
\begin{align}
&\widehat{u(x,t)}=\int_{{\mathbf{Q}}}\sum_{k=0}^{\infty}\exp(-\vartheta_{k}t)
\chi_{k}(x)\chi_{k}(y)\phi(y)d\mu_{n}(y)
+\int_{{\mathbf{Q}}}\sum_{k=0}^{\infty}\exp(-\vartheta_{k}t)
\chi_{k}(x)\chi_{k}(y)\otimes\mathscr{J}(y)d\mu_{n}(y)
\end{align}
\end{lem}
\begin{proof}
The derivatives are
\begin{align}
&\frac{\partial}{\partial t}\widehat{u(x,t)}=\int_{{\mathbf{Q}}}\sum_{k=0}^{\infty}
\exp(-\vartheta_{k}t)(-\vartheta_{k}\chi_{k}(x))\chi_{k}(y)\phi(y)d\mu_{n}(y)\nonumber\\&+
\int_{{\mathbf{Q}}}\sum_{k=0}^{\infty}
\exp(-\vartheta_{k}t)(-\vartheta_{k}\chi_{k}(x))\chi_{k}(y)\otimes\mathscr{J}(y)d\mu_{n}(y)
\end{align}
\begin{align}
{\Delta}_{x}\widehat{u(x,t)}&\equiv \nabla^{2}\widehat{u(x,t)}= \int_{{\mathbf{Q}}}\sum_{k=0}^{\infty}\exp(-\vartheta_{k}t)
({\Delta}_{x}\chi_{k}(x))\chi_{k}(y)\phi(y)d\mu_{n}(y)\nonumber\\&
+\int_{{\mathbf{Q}}}\sum_{k=0}^{\infty}exp(-\vartheta_{k}t)
({\Delta}_{x}\chi_{k}(x))\chi_{k}(y)\otimes\mathscr{J}(y)d\mu_{n}(y)
\nonumber\\&=\int_{{\mathbf{Q}}}\sum_{k=0}^{\infty}exp(-\vartheta_{k}t)
(-\vartheta_{k}\chi_{k}(x))\chi_{k}(y)\phi(y)d\mu_{n}(y)\nonumber\\&
+\int_{{\mathbf{Q}}}\sum_{k=0}^{\infty}exp(-\vartheta_{k}t)
(-\vartheta_{k}\chi_{k}(x))\chi_{k}(y)\otimes\mathscr{J}(y)d\mu_{n}(y)
\end{align}
 since ${\Delta}_{x}\chi_{k}(x)=-\vartheta_{k}\chi_{k}(x)$. Hence
 $(\tfrac{\partial}{\partial t}-{\Delta}_{x})u(x,t)=0$ as required.
\end{proof}
The $L_{p}$ norm of the stochastic convolution integral dissipates as $t\rightarrow\infty$.
\begin{lem}
Let ${\mathbf{Q}}\subset\bm{\mathrm{R}}^{n}$. Given the stochastic convolution integral solution of the SCIVP on ${\mathbf{Q}}$ then
\begin{align}
&\lim_{t\uparrow\infty}\|\widehat{u(x,t)}\|_{L_{p}(\mathbf{Q})}
=\lim_{t\uparrow\infty}\left(\|\widehat{u(x,t)}
\|_{L_{p}(\mathbf{Q})}\right)^{p}=0\\&
\lim_{t\uparrow\infty}\mathbb{E}\llbracket\|
\widehat{u(x,t)}\|_{L_{p}(\mathbf{Q})}\rrbracket=
\lim_{t\uparrow\infty}\mathbb{E}\llbracket
\left(\|\widehat{u(x,t)}\|_{L_{p}(\mathbf{Q})}\right)^{p}\rrbracket=0
\end{align}
where $\|\mathscr{J}(x)\|_{L_{p}(\mathbf{Q}}<\infty$
\end{lem}
\begin{proof}
Young's inequality states that for a convolution $f*g$ of two functions with $1+\tfrac{1}{p}=\tfrac{1}{q}+\tfrac{1}{r}$ then
\begin{equation}
\|f*g\|_{L_{p}}\le \|f\|_{L_{r}}\|g\|_{L_{q}}
\end{equation}
so that
\begin{equation}
\left\|\int_{\mathbf{Q}}h(\bullet-y,t)\phi(y)d\mu_{n}(y)
\right\|_{L_{p}}\le \|h(\bullet-y,t)\|_{L_{r}}\|\phi(y)\|_{L_{q}}
\end{equation}
Then
\begin{align}
\lim_{t\uparrow\infty}\|\widehat{u(x,t)}\|_{L_{p}(\mathbf{Q})}&=\lim_{t\uparrow\infty}\left(\int_{\mathbf{Q}}
\left|\int_{\mathbf{Q}}h(x-y,t)\phi(y)d\mu_{n}(y)\right|^{p}d\mu_{n}(x)\right)^{1/p}\nonumber\\&+
\left(\int_{\mathbf{Q}}\left|\int_{\mathbf{Q}}
h(x-y,t)\otimes\mathscr{J}(y)d\mu_{n}(y)\right|^{p}d\mu_{n}(x)\right)^{1/p}\nonumber\\&
\equiv\lim_{t\uparrow\infty}\left\|\int_{\mathbf{Q}}h(\bullet-y,t)\phi(y)d\mu_{n}(y)
\right\|_{L_{p}(\mathbf{Q})}+\lim_{t\uparrow\infty}\left\|\int_{\mathbf{Q}}h(\bullet-y,t)
\otimes\mathscr{J}(y)d\mu_{n}(y)\right\|_{L_{p}(\mathbf{Q})}\nonumber\\&\le \lim_{t\uparrow\infty}\|h(\bullet-y,t)\|_{L_{r}(\mathbf{Q}}
\|\phi(y)\|_{L_{q}(\mathbf{Q})}+\lim_{t\uparrow\infty}\|h(\bullet-y,t)\|_{L_{r}}
\otimes\|\mathscr{J}(y)\|_{L_{q}(\mathbf{Q})}\nonumber\\&
=\lim_{t\uparrow\infty}\|
h(\bullet-y,t)\|_{L_{r}}\big(\|\phi(y)\|_{L_{q}
(\mathbf{Q})}+\|\mathscr{J}(y)\|_{L_{q}(\mathbf{Q})}\big)\nonumber\\&
=\lim_{t\uparrow\infty}\left(\frac{1}{q^{n/2q}(4\pi t)^{\tfrac{n}{2}(1-t\frac{1}{p})}}\right)
\big(\|\phi(y)\|_{L_{q}(\mathbf{Q})}+\|\mathscr{J}(y)\|_{L_{q}(\mathbf{Q})}\big)=0
\end{align}
Taking the expectation is superfluous, however the result is equivalent. Since
\begin{align}
\mathbb{E}\llbracket\|\mathscr{J}(\bullet)\|
\rrbracket_{L_{p}({\mathbf{Q}})}&\mathbb{E}\left\llbracket\left( \int_{\mathbf{Q}}|\mathscr{J}(x)|^{p}d\mu_{n}(x)\right)^{1/p}
\right\rrbracket\equiv\left( \int_{\mathbf{Q}}|\mathbb{E}\big\llbracket\mathscr{J}(x)|^{p}\big\rrbracket d\mu_{n}(x)\right)^{1/p}\nonumber\\&
=\left(\frac{1}{2}v({\mathbf{Q}})[\zeta^{p/2}+(-1)^{p}\zeta^{p/2}]\right)^{1/p}
\end{align}
one has
\begin{align}
\lim_{t\uparrow\infty}\|\widehat{u(x,t)}\|_{L_{p}(\mathbf{Q})}&=\lim_{t\uparrow\infty}
\left(\int_{\mathbf{Q}}\left|\int_{\mathbf{Q}}
h(x-y,t)\phi(y)d\mu_{n}(y)\right|^{p}d\mu_{n}(x)\right)^{1/p}\nonumber\\&+
\mathbb{E}\left\llbracket\left(\int_{\mathbf{Q}}\left|\int_{\mathbf{Q}}
h(x-y,t)\otimes\mathscr{J}(y)d\mu_{n}(y)\right|^{p}d\mu_{n}(x)\right)^{1/p}\right\rrbracket
\nonumber\\&
\equiv\lim_{t\uparrow\infty}\left\|\int_{\mathbf{Q}}h(\bullet-y,t)\phi(y)d\mu_{n}(y)
\right\|_{L_{p}(\mathbf{Q})}+\lim_{t\uparrow\infty}\mathbb{E}
\left\llbracket\left\|\int_{\mathbf{Q}}h(\bullet-y,t)\otimes\mathscr{J}(y)d\mu_{n}(y)
\right\|_{L_{p}(\mathbf{Q})}\right\rrbracket\nonumber\\&\le \lim_{t\uparrow\infty}\|h(\bullet-y,t)\|_{L_{r}(\mathbf{Q}}
\|\phi(y)\|_{L_{q}(\mathbf{Q})}+\lim_{t\uparrow\infty}\|h(\bullet-y,t)\|_{L_{r}}
\otimes \mathbb{E}\llbracket\|\mathscr{J}(y)\|_{L_{q}(\mathbf{Q})}\rrbracket\nonumber\\&
=\lim_{t\uparrow\infty}\|h(\bullet-\bullet,t)\|_{L_{r}}\big(\|\phi(\bullet)\|_{L_{q}
(\mathbf{Q})}+\mathbb{E}\llbracket\|\mathscr{J}(\bullet)\|_{L_{q}(\mathbf{Q})}
\rrbracket\|_{L_{q}(\mathbf{Q})}\big)\nonumber\\&
=\lim_{t\uparrow\infty}\left(\frac{1}{q^{n/2q}(4\pi t)^{\tfrac{n}{2}(1-t\frac{1}{p})}}\right)\big(\|\phi(\bullet)\|_{L_{q}(\mathbf{Q})}+
\mathbb{E}\llbracket\|\mathscr{J}(\bullet)\|_{L_{q}(\mathbf{Q})}\rrbracket\big)=0
\nonumber\\&=\lim_{t\uparrow\infty}\left(\frac{1}{q^{n/2q}(4\pi t)^{\tfrac{n}{2}(1-t\frac{1}{p})}}\right)\big(\|\phi(\bullet)\|_{L_{q}(\mathbf{Q})}
+\left(\frac{1}{2}v(\mathbf{Q})[\zeta^{p/2}+(-1)^{p}\zeta^{p/2}]\right)^{1/p}
\big)=0
\end{align}
\end{proof}
\begin{lem}
Let $u\in C^{\infty}(\mathrm{R}^{n}\times (0,\infty ))$ satisfy Lemma (2.11) so that
$|\nabla u(x,t)|\le (c(n)/\sqrt{t})\|\phi\|_{\infty}$. For the Cauchy problem with randomized initial data $\widehat{u(x,0)}=\phi(x)+\mathscr{J}(x)$ then
\begin{equation}
\mathbb{E}\llbracket |\widehat{u(x,t)}|\rrbracket \le \frac{c(n)}{\sqrt{t}}\|\phi\|_{\infty}
\end{equation}
\end{lem}
\begin{proof}
The solution is given by (5.44) so that
\begin{align}
&{\nabla}\widehat{u(x,t)}={\nabla} u(x,t)+\frac{1}{(4\pi t)^{n/2}}\int_{\bm{\mathrm{R}}^{n}}{\nabla}_{x} \exp\left(-\frac{\|x-y\|^{2}}{4t}\right)\otimes\mathscr{J}(y)d\mu_{n}(y)\nonumber\\&
={\nabla} u(x,t) - {(4\pi t)^{-n/2}}\int_{\mathrm{R}^{n}}
\frac{\|x-y\|}{2t}\exp\left(-\frac{\|x-y\|^{2}}{2t}\right)\otimes\mathscr{J}(y)d\mu_{n}(y)
\end{align}
Hence, taking the expectation
\begin{align}
&\mathbb{E}\llbracket |{\nabla}\widehat{u(x,t)}|\rrbracket={\nabla}u(x,t)
+\frac{1}{(4\pi t)^{n/2}}\int_{\mathrm{R}^{n}}{\nabla}_{x} \exp\left(-\frac{\|x-y\|^{2}}{4t}\right)\mathbb{E}\llbracket
\otimes\mathscr{J}(y)\rrbracket d\mu_{n}(y)\nonumber\\&
={\nabla} u(x,t)\le \frac{c(n)}{\sqrt{t}}\|\phi\|_{\infty}
\end{align}
\end{proof}
\begin{lem}
If Thm 2.14 holds such that $\sup|u|_{\bm{\mathrm{R}}^{n}}\times (0,\infty)=\sup|\phi|_{\mathrm{R}^{n}}$ and $ u(x,t)\le A \exp(a|x|^{2})$ then for the Cauchy problem one has
\begin{equation}
\sup \mathbb{E}\llbracket\widehat{u(x,t)}\rrbracket_{\mathrm{R}^{n}}\times (0,\infty) = \sup \mathbb{E}\llbracket|\widehat{\phi}|\rrbracket_{\mathrm{R}^{n}}
\end{equation}
\end{lem}
\begin{proof}
Writing $\widehat{u(x,t)}=u(x,t)+\mathscr{U}(x,t)$ then $\mathbb{E}\llbracket\widehat{u(x,t)}\rrbracket=u(x,t)\le A \exp(a|x|^{2})$ since $\mathbb{E}\llbracket {\mathscr{U}(x,t)}\rrbracket=0$.
\begin{equation}
\sup \mathbb{E}\llbracket u(x,t)+\mathscr{U}(x,t)\rrbracket_{\mathrm{R}^{n}\times
(0,\infty)}=\sup\mathbb{E}\llbracket|\phi+\mathscr{J}(x)|\rrbracket_{\mathrm{R}^{n}}=
\sup|u|_{\mathrm{R}^{n}\times (0,\infty)}=\sup|\phi|_{\mathrm{R}^{n}}
\end{equation}
\end{proof}
The following preliminary lemma appears in Davies $\mathbf{[55]}$.
\begin{lem}
Let $({\mathbf{M}},\bm{g})$ be a Riemannian manifold and let
${\mathbf{Q}},{\mathbf{Q}}^{\prime}$ be two compact sets or domains in
${\mathbf{M}}$. If $v(\mathbf{Q}),v(\mathbf{Q}^{\prime})$ are the volumes of ${\mathbf{Q}}$ and ${\mathbf{Q}}^{\prime}$. If
$d^{2}({\mathbf{Q}},{\mathbf{Q}}^{\prime})$ is the distance between the centres of
$({\mathbf{Q}},{\mathbf{Q}}^{\prime})$ then the following estimate holds
\begin{align}
\iint_{{\mathbf{Q}},\mathbf{Q}^{\prime}}h(x-y,t)d\mu_{n}(x) d\mu_{n}(y)\le
\sqrt{v({\mathbf{Q!}})v({\mathbf{Q}}^{\prime})}
\exp\left(-\frac{d^{2}({\mathbf{Q}},{\mathbf{Q}}^{\prime})}{4t}\right)
\end{align}
On Euclidean space ${\mathbf{M}}={\mathbf{R}}^{n}$ the distance $d^{2}({\mathbf{Q}},{\mathbf{Q}}^{\prime})$
is just the usual Euclidean distance between the centres of the domains. If
${\mathbf{Q}}=\bm{\mathbf{Q}}^{\prime}$ and the centres coincide then
\begin{align}
\iint_{{\mathbf{Q}},{\mathbf{Q}}}h(x-y,t)d\mu_{n}(x)d\mu_{n}(y)\le
\sqrt{v({\mathbf{Q}})v({\mathbf{Q}})}=v(\mathbf{Q})
\end{align}
\end{lem}
The following lemma is then the stochastic version of Lemma 5.16 for $L_{2}$ estimates of the solution of the heat equation.
\begin{thm}
Let $u(x,t)$ solve the Cauchy IVP for the heat equation such that
${\square}u(x,t)=(\frac{\partial}{\partial t}-         {\Delta})u(x,t)=0,x\in{\mathbf{Q}},t>0$ with initial data $u(x,0)=\phi(x),x\in{\mathbf{Q}},t=0$. The $L_{2}({\mathbf{Q}})$ norms satisfy the estimate $\|u(\bullet,t)\|_{L_{2}({\mathbf{Q}})}\le \|\phi(\bullet)\|_{L_{2}({\mathbf{Q}})}$ or equivalently $(\|u(\bullet,t)\|_{L_{2}({\mathbf{Q}})})^{2}\le(\|\phi(\bullet)\|_{L_{2}({\mathbf{Q}})})^{2}$. If $\widehat{u(x,t)}$ solves the randomly perturbed heat equation ${\square}\widehat{u(x,t)}=0$  with random initial data $\widehat{\phi(x)}=\phi(x)+\mathscr{J}(x)$ then $\widehat{u(x,t)}
=\int_{\bm{\mathbf{Q}}}h(x-y,t)\otimes\mathscr{J}(y)d\mu_{n}(y)$ the following stochastically averaged estimate
\begin{align}
\mathbb{E}\llbracket(\|\widehat{(u(\bullet,t)}\|_{L_{2}({\mathbf{Q}})})^{2}\rrbracket
\le \mathbb{E}\llbracket(\|\widehat{\phi(\bullet)}\|_{L_{2}({\mathbf{Q}})})^{2}\rrbracket
\end{align}
or equivalently
\begin{align}
\mathbb{E}\left\llbracket\int_{\bm{\mathbf{Q}}}|\widehat{(u(x,t)}|^{2}d\mu_{n}(x)
\right\rrbracket\le\mathbb{E}\left\llbracket
\int_{\bm{\mathbf{Q}}}|\widehat{(\phi(x)}|^{2}d\mu_{n}(x)
\right\rrbracket
\end{align}
or
\begin{align}
\mathbb{E}\left\llbracket\left(\int_{{\mathbf{Q}}}
\left|\int_{{\mathbf{Q}}}h(x-y,t)
\otimes\widehat{\phi(x)}d\mu_{n}(y)\right|^{2}d\mu_{n}(x)\right)^{1/2}
\right\rrbracket\le\mathbb{E}\left\llbracket\left(\int_{{\mathbf{Q}}}
|\widehat{\phi(x)}|^{2}d\mu_{n}(x)\right)^{1/2}\right\rrbracket
\end{align}
holds provided that
\begin{align}
\iint_{{\mathbf{Q}},v{\mathbf{Q}}}h(x-y,t)d\mu_{n}(x)d\mu_{n}(y)\le
\sqrt{v({\mathbf{Q}})v({\mathbf{Q}})}=v(\bm{\mathbf{Q}})
\end{align}
which is Lemma (5.16)
\end{thm}
\begin{proof}
First, apply the Cauchy Schwarz inequality to $\int_{{\mathbf{Q}}}
h(x-y,t)d\mu_{n}(y) $ to derive the estimate
\begin{align}
&\int_{\bm{\mathbf{Q}}}h(x-y,t)d\mu_{n}(y) \equiv
\int_{\bm{\mathbf{Q}}}|h(x-y,t)\times \mathds{1}|d\mu_{n}(y)\nonumber\\&\le
\left(\int_{{\mathbf{Q}}}|h(x-y,t)|^{2}d\mu_{n}(x)\right)^{1/2}
\left(\int_{{\mathbf{Q}}}|\mathds{1}|^{2})d\mu_{n}(x)\right)^{1/2}\nonumber\\&=
\left(\int_{{\mathbf{Q}}}|h(x-y,t)|^{2}d\mu_{n}(x)\right)^{1/2}
\left(\int_{{\mathbf{Q}}}d\mu_{n}(x)\right)^{1/2}\nonumber\\&=
\left(\int_{{\mathbf{Q}}}|h(x-y,t)|^{2}d\mu_{n}(x)\right)^{1/2}
(v({\mathbf{Q}}))^{1/2}\le \left(\int_{{\mathbf{Q}}}|h(x-y,t)|^{2}d\mu_{n}(x)\right)^{1/2}
(v({\mathbf{Q}}))\nonumber\\&
=\left(\int_{{\mathbf{Q}}}|h(x-y,t)|^{2}d\mu_{n}(x)\right)
(v({\mathbf{Q}}))
\end{align}
Then if (5.74) holds
\begin{align}
&\int_{{\mathbf{Q}}}\mathbb{E}\llbracket(\|\widehat{(u(\bullet,t)}\|_{L_{2}({\mathbf{Q}})})^{2}\equiv
\mathbb{E}\left\llbracket\int_{{\mathbf{Q}}}|\widehat{(u(x,t)}|^{2}d\mu_{n}(x)
\right\rrbracket
\le\int_{{\mathbf{Q}}}|\phi(x)|^{2}d\mu_{n}(x)+\int_{{\mathbf{Q}}}\mathbb{E}\llbracket |\mathscr{J}(x)|^{2}\rrbracket d\mu_{n}(x)
\end{align}
or
\begin{align}
(\|u(\bullet,t)\|_{L_{2}({\mathbf{Q}})})^{2}
+\int_{{\mathbf{Q}}}\mathbb{E}\left\llbracket
\left|\int_{{\mathbf{Q}}}h (x-y,t)\otimes\mathscr{J}(y)d\mu_{n}(y)
\right|^{2}\right\rrbracket d\mu_{n}(x) \le\left (\|\phi(\bullet)\|_{L_{2}({\mathbf{Q}})}\right)^{2}
+\zeta v({\mathbf{Q}})
\end{align}
Applying Cauchy Schwartz to the stochastic integral on the rhs
\begin{align}
&\mathbb{E}\left\llbracket
\left|\int_{{\mathbf{Q}}}h(x-y,t)\otimes\mathscr{J}(y)
d\mu_{n}(y)\right|^{2}\right\rrbracket d\mu_{n}(x)
\nonumber\\&\le
\int_{{\mathbf{Q}}}\left(\int_{{\mathbf{Q}}}|h(x-y,t)|^{2}d\mu_{n}(y)\right)
\left(\mathbb{E}\left\llbracket\int_{{\mathbf{Q}}}|\mathscr{J}(y)|^{2}
d\mu_{n}(y)\right\rrbracket\right)d\mu_{n}(x)\nonumber\\&=
\int_{{\mathbf{Q}}}\left(\int_{{\mathbf{Q}}}|h(x-y,t)|^{2}d\mu_{n}(y)\right)
\left(\int_{{\mathbf{Q}}}\mathbb{E}\llbracket|\mathscr{J}(y)|^{2}\rrbracket
d\mu_{n}(y)\right)d\mu_{n}(x)\nonumber\\&=
\zeta\int_{{\mathbf{Q}}}\left(\int_{{\mathbf{Q}}}|h(x-y,t)|^{2}d\mu_{n}(y)\right)
v(\bm{\mathbf{Q}})d\mu_{n}(x)
\end{align}
The estimate now becomes
\begin{align}
\zeta\int_{{\mathbf{Q}}}\left(\int_{{\mathbf{Q}}}|
h(x-y,t)|^{2}d\mu_{n}(y)\right)v({\mathbf{Q}})d\mu_{n}(x)\le\left (\|\phi(\bullet)\|_{L_{2}({\mathbf{Q}})}\right)^{2}
+\zeta v({\mathbf{Q}})
\end{align}
so it is required that
\begin{align}
\int_{{\mathbf{Q}}}\left(\int_{{\mathbf{Q}}}|h(x-y,t)|^{2}d\mu_{n}(y)\right)
v({\mathbf{Q}})d\mu_{n}(x)\le+v({\mathbf{Q}})
\end{align}
Now the estimate can be expressed as
\begin{align}
\iint_{\bm{\mathbf{Q}}}h(x-y,t)d\mu_{n}(x) d\mu_{n}(y)\le
\int_{\bm{\mathbf{Q}}}\left(\int_{{\mathbf{Q}}}|h(x-y,t)|^{2}d\mu_{n}(y)\right)
v(\bm{\mathbf{Q}})d\mu_{n}(x)\le+v({\mathbf{Q}})
\end{align}
which is (5.73).
\end{proof}
The general global solution for the stochastic initial Cauchy problem for the inhomogenous equation is as follows.
\begin{thm}
Let $f\in C^{2}(\mathbf{Q}\times(0,\infty)),\phi\in C^{2}({\mathbf{Q}})$ and consider the stochastic initial-value problem for the inhomogeneous equation
\begin{align}
&{\mathbf{Q}}\widehat{u(x,t)}=\frac{\partial}{\partial t}\widehat{u(x,t)}-{\Delta}_{x}\widehat{u(x,t)}
+f(x,t),~~(x\in \bm{\mathrm{R}}^{n},t>0)\\&
\widehat{\psi(x,0)}=\phi(x)+\mathscr{J}(x),~~(x\in\bm{\mathrm{R}}^{n},t=0)
\end{align}
where $\mathscr{J}(x)$ is the GRSF with the same properties as before. If $u(x,t)$ solves the PDE $\square u(x,t)=(\tfrac{\partial}{\partial t}-{\Delta}_{x})u(x,t)=f(x,t)$,
then the solution to this problem is given by the stochastic convolution integral
\begin{align}
&\widehat{u(x,t)}=\frac{1}{(4\pi)^{3/2}}\int_{0}^{R}\int_{\mathbf{R}^{n}}
\frac{e^{\frac{-\|x-y\|^{2}}{|t-s|^{2}}}}{|t-s|^{2}} f(y,s)d\mu_{n}(y)ds\nonumber\\&+\frac{1}{(4\pi t)^{n/2}}\int_{\mathbf{R}^{n}}\exp\left(-\frac{\|x-y\|^{2}}{4t}\right)\phi(y)d\mu_{n}(x)+
\frac{1}{(4\pi t)^{n/2}}\int_{\mathbf{R}^{n}}\exp\left(-\frac{\|x-y\|^{2}}{4t}\right)\otimes\mathscr{J}(y)
d\mu_{n}(y)\nonumber\\&\equiv \int_{0}^{t}\int_{\mathbf{R}^{n}}h(x-y,t-s)ds d\mu_{n}(y)+\int_{\mathrm{R}^{n}}\phi(y)d\mu_{n}(y)+\int_{\mathbf{R}^{n}}
h(x-y,t)\otimes\mathscr{J}(y)d\mu_{n}(y)\nonumber\\&\equiv u(x,t)+\int_{\mathbf{R}^{n}}h(x-y,t)\otimes\mathscr{J}(y)d\mu_{n}(y)
\end{align}
\end{thm}
\begin{proof}
\begin{align}
&\frac{\partial}{\partial t}\widehat{u(x,t)}=\frac{\partial}{\partial t}u(x,t)+\int_{\mathbf{R}^{n}}\frac{\partial}{\partial t}
h(x-y,t)\mathscr{J}(y)d\mu_{n}(y)\\&
{\Delta}_{x}\widehat{u(x,t)}={\Delta}_{x}u(x,t)+
\int_{\mathbf{R}^{n}}{\Delta}_{x}h(x-y,t)\otimes\mathscr{J}(y)d\mu_{n}(y)
\end{align}
so that
\begin{align}
&{\square}\widehat{u(x,t)}=\left(\frac{\partial}{\partial t}-
{\Delta}_{x}\right)\widehat{u(x,t)}=\frac{\partial}{\partial t}u(x,t)
-\Delta_{x}u(x,t)\nonumber\\&+\int_{\mathbf{R}^{n}}[\frac{\partial}{\partial t}
h(x-y,t)-{\Delta}_{x}h(x-y,t)]\otimes\mathscr{J}(y)d\mu_{n}(y)\nonumber\\&
=\mathcal{F}(x,t)\nonumber\\&+\int_{\mathbf{R}^{n}}[\frac{\partial}{\partial t}
h(x-y,t)-\Delta_{x}h(x-y,t)]\mathscr{J}(y)d\mu_{n}(y)=f(x,t)
\end{align}
since ${\square}h(x-y,t)=\tfrac{\partial}{\partial t}
h(x-y,t)-{\Delta}_{x}h(x-y,t)=0$ for the heat kernel or fundamental solution.
\end{proof}
\subsection{Stochastically averaged mean value theorem in a heat ball}
\begin{prop}
For fixed $x\in\bm{\mathrm{R}}^{n},t\in\bm{\mathrm{R}}$ and $R>0$ define a 'noisy' or stochastic heat ball as the region
\begin{equation}
\widehat{\pmb{\mathscr{B}}(x,t;R)}=\bigg{\lbrace}(y,s)\in\mathbf{R}^{n}\times[0,\infty)|s\le t,h(x-y,t-s)\ge \frac{1}{R^{n}},\mathscr{J}(x),\mathscr{J}(y)\bigg{\rbrace}
\end{equation}
\end{prop}
where $\mathscr{J}(x)$ is a time-independent GRSF existing at all points in the heat ball and all points in $\mathbf{R}^{n}$.
\begin{prop}
Let $u(x,t)\in C^{2}(\mathbf{U}_{T})$ solve the initial Cauchy value problem for the heat equation and let $\widehat{(x,t)}=u(x,t)+\mathscr{U}(x,t)$ solve the Cauchy problem with random initial data and where $\mathscr{U}(x,t)$ is a stochastic convolution. Then if $u(x,t)$ satisfies the MVP, the stochastically averaged mean-value property is
\begin{equation}
\mathbb{E}\big\llbracket\widehat{u(x,t)}\big\rrbracket=\frac{1}{4R^{n}}
\iint_{\pmb{\mathscr{B}}(x,t;R)}\mathbb{E}
\big\llbracket\widehat{u(y,s)}\big\rrbracket\frac{\|x-y\|^{2}}{|t-s|^{2}}d\mu_{n}(y)ds
\end{equation}
\end{prop}
\begin{proof}
\begin{align}
\mathbb{E}\big\llbracket\widehat{u(x,t)}\big\rrbracket&
=u(x,t)+\mathbb{E}\big\llbracket\mathscr{U}(x,t)\big\rrbracket
=u(x,t)\nonumber\\&=\frac{1}{4R^{n}}\iint_{\pmb{\mathscr{B}}(x,t;R)}\mathbb{E}\big\llbracket\widehat{u(y,s)}\big\rrbracket\frac{\|x-y\|^{2}}{|t-s|^{2}}d^{3}yds\nonumber\\&
=\frac{1}{4R^{n}}\iint_{\pmb{\mathscr{B}}(x,t;R)}u(y,s)\frac{\|x-y\|^{2}}{|t-s|^{2}}d\mu_{n}(y)ds+\frac{1}{4R^{n}}\iint_{{\mathrm{I\!H}}(x,t;R)}
\mathbb{E}\big\llbracket{\mathscr{U}(y,s)}\big\rrbracket\frac{\|x-y\|^{2}}{|t-s|^{2}}d^{3}yds\nonumber\\&
=\frac{1}{4R^{n}}\iint_{\pmb{\mathscr{B}}(x,t;R)}u(y,s)\frac{\|x-y\|^{2}}{|t-s|^{2}}
d\mu_{n}(y)ds
\end{align}
since $\mathbb{E}\big\llbracket{\mathscr{U}(y,s)}\big\rrbracket=0$.
\end{proof}
\begin{thm}
Given the SCIVP on ${\mathbf{Q}}\subset {\mathbf{R}}^{n}$ with random initial data $\widehat{u(x,0)}=u(x)+\mathscr{J}(x)$ then the perturbed heat equation is
${\square}\widehat{u(x,t)}=0$. Then the stochastic version of (2.56) is
\begin{align}
\frac{1}{2}\frac{d}{dt}\mathbb{E}\llbracket\widehat{\bm{\mathrm{I\!E}}}_{I}(t)
\rrbracket=\frac{1}{2}\frac{d}{dt}\mathbb{E}\left
\llbracket\int_{\mathbf{Q}}|\widehat{u(x,t)}|^{2}d\mu_{n}(x)\right\rrbracket<0
\end{align}
and is also non decreasing.
\end{thm}
\begin{proof}
The randomly perturbed solution is the convolution integral
\begin{align}
\widehat{u(x,t)}=u(x,t)+\int_{\mathbf{Q}}|
h(x-y,t)\otimes\mathscr{J}(y)d\mu_{n}(y)
\end{align}
Then
\begin{align}
&\frac{1}{2}\frac{d}{dt}\mathbb{E}\llbracket
{\widehat{\bm{\mathrm{I\!E}}}_{I}(t)}\rrbracket=\frac{1}{2}\frac{d}{dt}
\mathbb{E}\left\llbracket\int \left|u(x,t)+\int_{\mathbf{Q}}
h(x-y,t)\otimes\mathscr{J}(y)d\mu_{n}(y)
\right|^{2}d\mu_{n}(x)\right\rrbracket\nonumber\\&
=\frac{1}{2}\frac{d}{dt}\mathbb{E}\llbracket
\int_{\mathbf{Q}}|u(x,t)|^{2}d\mu_{n}(x)\rrbracket+\frac{1}{2}\frac{d}{dt}
\mathbb{E}\left\llbracket \int_{\mathbf{Q}}u(x,t)\int_{\mathbf{Q}}h(x-y,t)\otimes
\mathscr{J}(y)d^{n}y\right\rrbracket d\mu_{n}(x)\nonumber\\&
+\frac{1}{2}\frac{d}{dt}\int_{\mathbf{Q}}\mathbb{E}\left\llbracket
\left|\int_{\mathbf{Q}}h\otimes\mathscr{J}(y)d\mu_{n}(y)\right|^{2}\right\rrbracket d\mu_{n}(x)
\nonumber\\&
=\frac{1}{2}\frac{d}{dt}{\bm{\mathrm{I\!E}}}_{I}(t)+\frac{1}{2}\frac{d}{dt}\int_{\mathbf{Q}}
\mathbb{E}\left\llbracket
\left|\int_{\mathbf{Q}}h(x-y,t)\otimes\mathscr{J}(y)d\mu_{n}(y)\right|^{2}\right\rrbracket d\mu_{n}(x)\nonumber\\&
\le\frac{1}{2}\frac{d}{dt}{\bm{\mathrm{I\!E}}}_{I}(t)
+\frac{1}{2}\frac{d}{dt}\int_{\mathbf{Q}}
\underbrace{\bigg(\int_{\mathbf{Q}}|h(x-y,t)|^{2}d\mu_{n}(y)\bigg)\bigg(\int_{\mathbf{Q}}
\mathbb{E}\llbracket|\mathscr{J}(y)|^{2}
\rrbracket d\mu_{n}(y)\bigg)}_{via~CS~ineq.}d\mu_{n}(x)\nonumber\\&
=\frac{1}{2}\frac{d}{dt}{\bm{\mathrm{I\!E}}}_{I}(t)+
\frac{1}{2}\zeta v({\mathbf{Q}})\int_{\mathbf{Q}}\frac{d}{dt}
\left(\int_{\mathbf{Q}}|h(x-y,t)|^{2}d\mu_{n}(y)\right)d\mu_{n}(x)\nonumber\\&
=\frac{1}{2}\frac{d}{dt}{\bm{\mathrm{I\!E}}}_{I}(t)+
\frac{1}{2}\zeta v(\bm{\mathbf{Q}})\int_{\mathbf{Q}}
\left(\int_{\mathbf{Q}}\frac{d}{dt}|h(x-y,t)|^{2}d^{n}y\right)d\mu_{n}(x)\nonumber\\&
\frac{1}{2}\frac{d}{dt}{\bm{\mathrm{I\!E}}}_{I}(t)+
\frac{1}{2}\zeta v(\mathbf{Q})\int_{\mathbf{Q}}
\bigg(\int_{\mathbf{Q}}\bigg(\frac{-t^{-(n+2)}(2nt-(y-x)^{2}))e^{-\frac{\|x-y\|^{2}}{2t}}
}{2(4\pi)^{n}}\bigg)d\mu_{n}(y)\bigg)d\mu_{n}(x)<0
\end{align}
\end{proof}
\section{Stochastic Li-Yau estimate and parabolic Harnack inequalities}
The Li-Yau estimate and parabolic Harnack inequalities for solutions of the
heat equation were presented in Section 3. Here, stochastic versions are presented for
a randomly perturbed heat equation where the initial data is a Gaussian random field.
\begin{lem}
Let $u=u(x,t)$ be a solution of ${\square}u\equiv
(\tfrac{\partial}{\partial t}-\Delta)u=0$ for all $x\in{\mathbf{Q}}\subset{\mathbf{R}}^{n}$.
with initial data $u(x,0)=\phi(x)$. From Lemma (3.1), the following equalities hold such that ${\square}(\log u)=\tfrac{|\nabla u|^{2}}{|u|^{2}}$ and
${\square}(u\log u)=-\tfrac{|\nabla u|^{2}}{|u|}$. Suppose now that
the initial data is randomised as $\widehat{u(x,0)}= \widehat{\phi(x)}=\phi(x)+\mathscr{J}(x)$. Then the randomly perturbed solution is $\widehat{u(x,t)}=u(x,t)+\int_{\mathbf{Q}}h(x-y,t)\otimes\mathscr{J}(y)d^{n}y\equiv
u(x,t)+\mathscr{U}(x,t)$. Then $\widehat{u}=u+\mathscr{U}=u(x,t)+\mathscr{U}(x,t)$
satisfies the inequalities:
\begin{align}
&{\square}(\log\widehat{u})=(\frac{\partial}{\partial t}-\Delta)(\log\widehat{u})=\frac{|\nabla\widehat{u}|^{2}}{|\widehat{u}|^{2}}\\&
{\square}(\widehat{u}\log\widehat{u})
=(\frac{\partial}{\partial t}-\Delta)(\widehat{u}\log\widehat{u})=
-\frac{|\nabla\widehat{u}|^{2}}{|\widehat{u}|}
\end{align}
or equivalently
\begin{align}
&{|\widehat{u}|^{2}}{\square}(\log\widehat{u})={|\widehat{u}|^{2}}(\frac{\partial}{\partial t}-\Delta)(\log\widehat{u})=
{|\nabla\widehat{u}|^{2}}\\&
-|\widehat{u}|{\square}(\widehat{u}\log\widehat{u})
=-|\widehat{u}|(\frac{\partial}{\partial t}-\Delta)(\widehat{u}
\log\widehat{u})={|\nabla\widehat{u}|^{2}}
\end{align}
which is
\begin{align}
&|{u+\mathscr{U}|^{2}}{\square}(\log(u+\mathscr{U})
={|{u}+\mathscr{U}|^{2}}(\frac{\partial}{\partial t}-\Delta)(\log(u+\mathscr{U})
={|\nabla u+\nabla\mathscr{U}|^{2}}\\&
-{|{u+\mathscr{U}}|}{\square}(\log(u+\mathscr{U})
=-{|{u}+\mathscr{U}}|(\frac{\partial}{\partial t}-\Delta)(\log(u+\mathscr{U})
={|\nabla u+\nabla\mathscr{U}|^{2}}
\end{align}
If $\phi(x)=0$ then $\widehat{u}=\mathscr{U}$ and so
\begin{align}
&{|{\mathscr{U}}|^{2}}{\square}(\log(\mathscr{U})
={|\mathscr{U}|^{2}}\left(\frac{\partial}{\partial t}-\Delta\right)(\log(\mathscr{U})
={|\nabla\mathscr{U}|^{2}}\\&
-{|{\mathscr{U}}|}{\square}(\log(\mathscr{U})
=-{|\mathscr{U}}|\left(\frac{\partial}{\partial t}-\Delta\right)(\log(\mathscr{U})
={|\nabla\mathscr{U}|^{2}}
\end{align}
\end{lem}
\begin{proof}
As in Lemma (3.1), the proof follows easily by direct calculation. The derivates
$\frac{\partial}{\partial t}\mathscr{U}$ and $\nabla\mathscr{U}$ exist so that
\begin{align}
&{\square}(\log\widehat{u})=(\frac{\partial}{\partial t}\widehat{u}-\Delta\widehat{u})(\log u)\equiv
\frac{\frac{\partial}{\partial t}\widehat{u}}{\widehat{u}}
-\nabla(\nabla\log\widehat{u})\nonumber\\&
=\frac{\Delta\widehat{u}}{\widehat{u}}-\nabla\left(\frac{\nabla\widehat{u}}
{\widehat{u}}\right)=\frac{\Delta\widehat{u}}{\widehat{u}}
-\frac{(\nabla\nabla u)u-\nabla\widehat{u}\nabla\widehat{u}}{|\widehat{u}|^{2}}\nonumber\\&=
\frac{\Delta\widehat{u}}{\widehat{u}}-\frac{\Delta\widehat{u}}
{\widehat{u}}+\frac{|\nabla\widehat{u}|^{2}}{|\widehat{u}|^{2}}
=\frac{|\nabla\widehat{u}|^{2}}{|\widehat{u}|^{2}}
\end{align}
and similarly for (6.2).
\end{proof}
\begin{cor}
\begin{align}
|\widehat{u}|^{2}{\square}(\log\widehat{u})
=-|\widehat{u}|{\square}(\widehat{u}\log\widehat{u})
\end{align}
\end{cor}
\begin{cor}
The expectations of (6.1) and (6.2) are
\begin{align}
&\mathbb{E}\llbracket{|\widehat{u}|^{2}}{\square}(\log\widehat{u})\rrbracket=
\mathbb{E}\left\llbracket{|\widehat{u}|^{2}}(\frac{\partial}{\partial t}-\Delta)(\log\widehat{u})\right\rrbracket=
\mathbb{E}\llbracket{|\nabla\widehat{u}|^{2}}\rrbracket\\&
-\mathbb{E}\llbracket|\widehat{u}|{\square}(\widehat{u}\log\widehat{u})=
-\mathbb{E}\left\llbracket|\widehat{u}|(\frac{\partial}{\partial t}-\Delta)(\widehat{u}
\log\widehat{u})\right\rrbracket=\mathbb{E}\llbracket{|\nabla\widehat{u}|^{2}}\rrbracket
\end{align}
and for $\phi(x)=0$
\begin{align}
&\mathbb{E}\llbracket{|{\mathscr{U}}|^{2}}{\square}(\log(\mathscr{U})\rrbracket
=\mathbb{E}\left\llbracket{|\mathscr{U}|^{2}}(\frac{\partial}{\partial t}-\Delta)(\log(\mathscr{U})\right\rrbracket
=\mathbb{E}\llbracket{|\nabla\mathscr{U}|^{2}}\\&
-\mathbb{E}\llbracket{|{\mathscr{U}}|}{\square}(\mathscr{U}\log(\mathscr{U})\rrbracket
=-\mathbb{E}\left\llbracket{|\mathscr{U}}|(\frac{\partial}{\partial t}-\Delta)(\mathscr{U}\log(\mathscr{U})\right\rrbracket
=\mathbb{E}\llbracket{|\nabla\mathscr{U}|^{2}}\rrbracket
\end{align}
\end{cor}
\begin{cor}
\begin{align}
&\mathbb{E}\llbracket{|{\mathscr{U}}|^{2}}{\square}(\log(\mathscr{U})\rrbracket
=-\mathbb{E}\llbracket{|{\mathscr{U}}|}{\square}(\mathscr{U}\log(\mathscr{U})\rrbracket\nonumber
\\&=\mathbb{E}\left\llbracket\left|
\int_{{\mathbf{Q}}}h(x-y,t)\otimes\mathscr{J}(y)d\mu_{n}(y)\right|^{2}
\right\rrbracket\nonumber\\&\le
\left(\int_{{\mathbf{Q}}}|h(x-y,t)|^{2}d\mu_{n}(y)\right)
\left(\int_{{\mathbf{Q}}}\mathbb{E}\llbracket|\mathscr{J}(y)|^{2}d\mu_{n}(y)\right)\nonumber\\&=
\zeta v(\mathbf{Q})\left(\int_{{\mathbf{Q}}}|h(x-y,t)|^{2}d\mu_{n}(y)\right)\nonumber\\&
=\zeta v(\mathbf{Q})\Lambda_{2}^{2}t^{-n}
\int_{\mathbf{Q}}e^{-\frac{\|x-y\|^{2}}{\Lambda_{2}t}}d\mu_{n}(y)
\end{align}
where the Cauchy Schwarz inequality and the bound (2.18) have been used.
\end{cor}
The main theorem for the stochastic extension of the Li-Yau inequality is as follows.
\begin{thm}
Let ${\mathbf{Q}}\subset{\mathbf{R}}^{n}$ be a compact set or
domain with boundary $\partial {\mathbf{Q}}$. Let $u(x,t)$ be a solution
of the deterministic heat equation ${\square}u(x,t)=0$ for
$x\in{\mathbf{Q}}$ and $t>0$, with initial Cauchy data
$u(x,0)=\phi(x)$, for some function $\phi(x)$. The solution for t>0 is the
convolution integral $u(x,t)=\int_{{\mathbf{Q}}}
h(x-y,t)\phi(y)d\mu_{n}(y)$. The Li-Yau estimate is then
\begin{align}
\frac{|\nabla_{x}u(x,t)|^{2}}{|u(x,t)|^{2}}-\frac{|\frac{\partial}{\partial t}u(x,t)|}{|u(x,t)|}\le \frac{1}{2}nt^{-1}
\end{align}
or
\begin{align}
|\nabla_{x}u(x,t)|^{2}-|\frac{\partial}{\partial t}u(x,t)u(x,t)|
\le \frac{1}{2}nt^{-1}|u(x,t)|^{2}
\end{align}
which from Thm 3.4 this is equivalent to
\begin{align}
&\left|\int_{{\mathbf{Q}}}|\nabla_{x}h(x-y,t)|^{2}d^{n}y\right|
\le\left(\int_{{\mathbf{Q}}}|\frac{\partial}{\partial t}h(x-y,t)|^{2}d\mu_{n}(y)\right)^{1/2}
\left(\int_{\bm{\mathbf{Q}}}|h(x-y,t)|^{2}d^{n}y\right)^{1/2}\nonumber\\&
\le \frac{1}{2}nt^{-1}\int_{{\mathbf{Q}}}|\frac{\partial}{\partial t}h(x-y,t)|^{2}d\mu_{n}(y)
\end{align}
and also
\begin{align}
&\Lambda_{1}^{2}t^{-n}\int_{\mathbf{Q}}\left(-\frac{2|x-y|}{\Lambda_{1}t}\right)^{2}
e^{-\frac{2|x-y|^{2}}{\Lambda_{1}t}}d^{n}y\le\left(\int_{{\mathbf{Q}}}|
\frac{\partial}{\partial t}h(x-y,t)|^{2}d^{n}y\right)^{1/2}
\left(\int_{{\mathbf{Q}}}|h(x-y,t)|^{2}d\mu_{n}(y)\right)^{1/2}\nonumber\\&
\le\Lambda_{2}^{2}t^{-n}\int_{\mathbf{Q}}\exp\left(-\frac{2|x-y|^{2}}{\Lambda_{2}t}
\right)d\mu_{n}(y)
\end{align}
If the initial data is now a GRF such that $\widehat{u(x,0)}=\phi(x)+\mathscr{J}(x)$ with $\mathbb{E}\llbracket \mathscr{J}(x)\otimes\mathscr{J}(x)\rrbracket=\zeta$ then
the randomly perturbed solution $\widehat{u(x,t)}$ satisfies
${\square}\widehat{u(x,t)}=0$, for $x\in{\mathbf{Q}}$ and $t>0$. Then
$\widehat{u(x,t)}$ satisfies a stochastic Li-Yau inequality of the form
\begin{align}
\mathbb{E}\left \llbracket\frac{|\nabla_{x}\widehat{u(x,t)}|^{2}}{|\widehat{u(x,t)}|^{2}}\right
\rrbracket-\mathbb{E}\left \llbracket
\frac{|\frac{\partial}{\partial t}\widehat{u(x,t)}|}{|\widehat{u(x,t)}|}\right\rrbracket \le \frac{1}{2}nt^{-1}
\end{align}
or the equivalent forms
\begin{align}
\mathbb{E}\llbracket|\nabla_{x}\widehat{u(x,t)}|^{2}\rrbracket-
\mathbb{E}\llbracket|\frac{\partial}{\partial t}\widehat{u(x,t)}\widehat{u(x,t)}|\rrbracket
\le \frac{1}{2}nt^{-1}|\mathbb{E}\llbracket\widehat{u(x,t)}|^{2}\rrbracket
\end{align}
\begin{align}
\mathbb{E}\llbracket\nabla\log(\widehat{u(x,t)})|\rrbracket
-\mathbb{E}\llbracket\frac{\partial}{\partial t}\log(\widehat{u(x,t)})|\rrbracket
\end{align}
only if (6.18) or (6.19) are also satisfied.
\end{thm}
\begin{proof}
The simplest proof utilises the stochastic Jensen inequality for concave functions. If
$\Psi(x)$ is a concave function and $\hat{X}$ a random variable with expectation $
\mathbb{E}\llbracket X \rrbracket$ then $\mathbb{E}\llbracket[\Psi(X)]\le \Psi(\mathbb{E}\llbracket X\rrbracket)$. Hence
\begin{align}
\mathbb{E}\llbracket{\nabla}_{i}\log\widehat{u(x,t)}\rrbracket-
\mathbb{E}\llbracket\frac{\partial}{\partial t}\log \widehat{u(x,t)}\rrbracket\le
|{\nabla}\log(\mathbb{E}\llbracket \widehat{u(x,t)}\rrbracket|^{2}-\frac{\partial}{\partial t}\log\mathbb{E}\llbracket\widehat{u(x,t)}\rrbracket|
\end{align}
But $\mathbb{E}\llbracket\widehat{u(x,t)}
\rrbracket=u(x,t)$ since $\mathbb{E}\llbracket\widehat{u(x,t)}\rrbracket=u(x,t)+
\int_{{\mathbf{Q}}}h(x-y,t)\otimes\mathbb{E}\llbracket\mathscr{J}(y)\rrbracket d\mu_{n}(y)=u(x,t)$.
Hence
\begin{align}
\mathbb{E}\llbracket{\nabla}_{i}\log\widehat{u(x,t)}\rrbracket-
\mathbb{E}\llbracket\frac{\partial}{\partial t}\log \widehat{u(x,t)}\rrbracket\le
|{\nabla}\log u(x,t)|^{2}-\frac{\partial}{\partial t}\log u(x,t)\le \frac{1}{2}nt^{-1}
\end{align}
A second proof follows from a direct 'brute force' computation of the derivatives. First
\begin{align}
&\mathbb{E}\llbracket |{\nabla}_{x}\widehat{u(x,t)}|^{2}\rrbracket=
|{\nabla}{x}u(x,t)|^{2}+\mathbb{E}\left\llbracket\left|\int_{{\mathbf{Q}}}
{\nabla}_{x}h(x-y,t)\otimes\mathscr{J}(y)d\mu_{n}(y)\right|^{2}
\right\rrbracket\\&
\mathbb{E}\llbracket |\frac{\partial}{\partial t}\widehat{u(x,t)}|^{2}\rrbracket=
|\frac{\partial}{\partial t}u(x,t)|^{2}+\mathbb{E}\left\llbracket\left|\int_{{\mathbf{Q}}}
\frac{\partial}{\partial t}h(x-y,t)\otimes\mathscr{J}(y)d\mu_{n}(y)\right|^{2}
\right\rrbracket
\end{align}
Using Cauchy Schwartz
\begin{align}
&\mathbb{E}\llbracket|{\nabla}_{x}\widehat {u(x,t)}|^{2}\rrbracket
=|u(x,t)|^{2}+\mathbb{E}\left\llbracket\left|\int_{{\mathbf{Q}}}
h(x-y,t)\otimes\mathscr{J}(y)d\mu_{n}(y)\right|^{2}\right\rrbracket\nonumber\\&
\le |u(x,t)|^{2}+\left(\int_{{\mathbf{Q}}}|h(x-y,t)|^{2}d\mu_{n}(y)\right)\left(
\int_{{\mathbf{Q}}}|\mathbb{E}\llbracket\mathscr{J}(y)|^{2}\rrbracket d^{n}y\right)
\nonumber\\&
=|u(x,t)|^{2}+\left(\int_{{\mathbf{Q}}}|h(x-y,t)|^{2}d\mu_{n}(y)\right)\zeta v({\mathbf{Q}})
\end{align}
Similarly
\begin{align}
&\mathbb{E}\llbracket|{\nabla}_{x}\widehat{u(x,t)}|^{2}\rrbracket\le
{\nabla} u(x,t)|^{2}+\left(\int_{{\mathbf{Q}}}|{\nabla}_{x}h(x-y,t)|^{2}d\mu_{n}(y)\right)\mu v({\mathbf{Q}}\\&
\mathbb{E}\left\llbracket|\frac{\partial}{\partial t}\widehat{u(x,t)}|^{2}\right\rrbracket\le
|\frac{\partial}{\partial t} u(x,t)|^{2}+\left(\int_{{\mathbf{Q}}}|\frac{\partial}{\partial t}h(x-y,t)|^{2}d\mu_{n}(y)\right)\mu v({\mathbf{Q}})
\end{align}
The estimate (6.21) can be expressed as
\begin{align}
\mathbb{E}\left\llbracket{|{\nabla} \widehat{u(x,t)}|^{2}}{}
\right\rrbracket-\mathbb{E}\left\llbracket{\frac{\partial}{\partial t}\widehat{u(x,t)}}\otimes{\widehat{u(x,t)}}\right\rrbracket
\le \frac{1}{2}nt^{-1}\mathbb{E}\llbracket |\widehat{u(x,t)}|^{2} \rrbracket
\end{align}
The term $\mathbb{E}\left\llbracket\tfrac{\partial}{\partial t}\widehat{u(x,t)}\otimes{\widehat{u(x,t)}}\right\rrbracket$
can be decomposed as
\begin{align}
&\mathbb{E}\left\llbracket\frac{\partial}{\partial t}\widehat{u(x,t)}\otimes{\widehat{u(x,t)}}\right\rrbracket=
\left((\frac{\partial}{\partial t}u(x,t)\right)u(x,t)\nonumber\\&+\mathbb{E}\left\llbracket
\underbrace{\left(\int_{\mathbf{Q}}
\frac{\partial}{\partial t}h(x-y,t)\otimes\mathscr{J}(y)d\mu_{n}(y)\right)}_{Apply~Cauchy~Schwartz}
\underbrace{\left(\int_{{\mathbf{Q}}}h(x-y,t)\otimes\mathscr{J}(y)d\mu_{n}(y)
\right)}_{Apply~Cauchy~Schwartz}\right\rrbracket\nonumber\\&\le
(\frac{\partial}{\partial t}u(x,t))u(x,t)
+\mathbb{E}\bigg\llbracket\left[\left(\int_{{\mathbf{Q}}}|
\frac{\partial}{\partial t}h(x-y,t)|^{2}d\mu_{n}(y)\right)^{1/2}
\left(\int_{{\mathbf{Q}}}|\mathscr{J}(y)|^{2}d\mu_{n}(y)\right)^{1.2}\right]\nonumber\\&\times
\left[\left(\int_{{\mathbf{Q}}}|h(x-y,t)|^{2}d\mu_{n}(y)\right)^{1/2}
\left(\int_{{\mathbf{Q}}}|\mathscr{J}(y)|^{2}d\mu_{n}(y)\right)^{1/2}\right]\bigg\rrbracket
\nonumber\\&
\le (\frac{\partial}{\partial t}u(x,t))u(x,t)
+\left[\left(\int_{{\mathbf{Q}}}|\frac{\partial}{\partial t}h(x-y,t)|^{2}d\mu_{n}(y)\right)^{1/2}
\right]\nonumber\\&
\left[\left(\int_{{\mathbf{Q}}}|h(x-y,t)|^{2}d^{n}y\right)^{1/2}
\mathbb{E}\bigg\llbracket\left(\int_{{\mathbf{Q}}}|\mathscr{J}(y)|^{2}d\mu_{n}(y)\right)\right]\bigg\rrbracket
\nonumber\\&
= (\frac{\partial}{\partial t}u(x,t))u(x,t)
+\left[\left(\int_{{\mathbf{Q}}}|\frac{\partial}{\partial t}h(x-y,t)|^{2}d\mu_{n}(y)\right)^{1/2}
\right]\nonumber\\&\left[\left(\int_{{\mathbf{Q}}}|h(x-y,t)|^{2}d\mu_{n}(y)\right)^{1/2}
\left(\int_{{\mathbf{Q}}}\mathbb{E}\big\llbracket|\mathscr{J}(y)|^{2}\big\rrbracket d\mu_{n}(y)\right)\right]\bigg\rrbracket
\nonumber\\&= (\frac{\partial}{\partial t}u(x,t))u(x,t)
+\left[\left(\int_{{\mathbf{Q}}}|\frac{\partial}{\partial t}h(x-y,t)|^{2}d\mu_{n}(y)\right)^{1/2}
\right]\nonumber\\&\left[\left(\int_{{\mathbf{Q}}}|h(x-y,t)|^{2}d\mu_{n}(y)\right)^{1/2}
\left(\int_{{\mathbf{Q}}}\zeta d\mu_{n}(y)\right)\right]\bigg\rrbracket\nonumber\\&= (\frac{\partial}{\partial t}u(x,t))u(x,t)\nonumber\\&
+\left[\left(\int_{{\mathbf{Q}}}|\frac{\partial}{\partial t}h(x-y,t)|^{2}d\mu_{n}(y)\right)^{1/2}
\right]\left(\int_{{\mathbf{Q}}}|h(x-y,t)|^{2}d\mu_{n}(y)\right)^{1/2}
\zeta v(\mathbf{Q})
\end{align}
Hence, the estimate becomes
\begin{align}
&|\nabla u(x,t)|^{2}+\left(\int_{{\mathbf{Q}}}|{\nabla}_{x}h(x-y,t)|^{2}d\mu_{n}(y)\right)\zeta v({\mathbf{Q}})-(\frac{\partial}{\partial t}u(x,t))u(x,t)
\nonumber\\&-\left(\int_{{\mathbf{Q}}}|\frac{\partial}{\partial t}h(x-y,t)|^{2}d\mu_{n}(y)\right)^{1/2}
\left(\int_{{\mathbf{Q}}}|h(x-y,t)|^{2}d\mu_{n}(y)\right)^{1/2}
\zeta v({\mathbf{Q}})\nonumber\\&\le \frac{1}{2}n t^{-1}+\frac{1}{2}nt^{-1}\int_{{\mathbf{Q}}}
|h(x-y,t)|^{2}d\mu_{n}(y)\zeta v({\mathbf{Q}})
\end{align}
Since $u(x,t)$ satisfies the estimate (6.18), it is required that
\begin{align}
&\left(\int_{{\mathbf{Q}}}|{\nabla}_{x}h(x-y,t)|^{2}d\mu_{n}(y)\right)\zeta v({\mathbf{Q}}
\nonumber\\&-\left(\int_{{\mathbf{Q}}}|\frac{\partial}{\partial t}h(x-y,t)|^{2}d\mu_{n}(y)\right)^{1/2}
\left(\int_{{\mathbf{Q}}}|h(x-y,t)|^{2}d\mu_{n}(y)\right)^{1/2}
\zeta v({\mathbf{Q}})\nonumber\\&\le \frac{1}{2}nt^{-1}\int_{{\mathbf{Q}}}
|h(x-y,t)|^{2}d\mu_{n}(y)\zeta v({\mathbf{Q}})
\end{align}
which is (cancelling the $v(\mathbf{Q})$ terms)
\begin{align}
&\left(\int_{{\mathbf{Q}}}|{\nabla}_{x}h(x-y,t)|^{2}d\mu_{n}(y)\right)
\nonumber\\&-\left(\int_{{\mathbf{Q}}}|\frac{\partial}{\partial t}h(x-y,t)|^{2}d\mu_{n}(y)\right)^{1/2}
\left(\int_{{\mathbf{Q}}}|h(x-y,t)|^{2}d\mu_{n}(y)\right)^{1/2}
\nonumber\\&\le \frac{1}{2}nt^{-1}\int_{{\mathbf{Q}}}
|h(x,y,t)|^{2}d\mu_{n}(y)
\end{align}
which also agrees exactly with (6.18).
\end{proof}
\begin{thm}
The randomly perturbed solution $\widehat{u(x,t)}$ also satisfies the following Li-Yau inequality
\begin{align}
\frac{\mathbb{E}\llbracket|\nabla\widehat{u(x,t)}|^{2}\rrbracket}
{\mathbb{E}\llbracket|\widehat{u(x,t)}|^{2}\rrbracket}-
\frac{\mathbb{E}\llbracket|\frac{\partial}{\partial t}\widehat{u(x,t)}|\rrbracket}
{\mathbb{E}\llbracket|\widehat{u(x,t)}|\rrbracket}\le \frac{1}{2}nt^{-1}
\end{align}
provided that
\begin{align}
\frac{|\nabla{u(x,t)}|^{2}}{|{u(x,t)}|^{2}}-\frac{|\frac{\partial}{\partial t}{u(x,t)}|}
{|{u(x,t)}|}\le \frac{1}{2}nt^{-1}
\end{align}
\end{thm}
\begin{proof}
The perturbed solution $\widehat{u(x,t)}$ has the derivatives
\begin{align}
&\frac{\partial}{\partial t}\widehat{u(x,t)}=\frac{\partial}{\partial t}u(x,t)+
\int_{\mathbf{Q}}\frac{\partial}{\partial t}h(x-y,t)\otimes \mathscr{J}(y)d\mu_{n}(y)\\&
\nabla\widehat{u(x,t)}=\nabla u(x,t)+
\int_{\mathbf{Q}}\nabla h(x-y,t)\otimes \mathscr{J}(y)d\mu_{n}(y)
\end{align}
Clearly, $\mathbb{E}\llbracket\frac{\partial}{\partial t}\widehat{u(x,t)}\rrbracket
=\frac{\partial}{\partial t}u(x,t)$ and $\mathbb{E}\llbracket\nabla\widehat{u(x,t)}\rrbracket
=\nabla u(x,t)$. Using the Cauchy Schwarz inequality (or the Holder inequality with $p=q=2$)
gives
\begin{align}
&\mathbb{E}\llbracket|\frac{\partial}{\partial t}\widehat{u(x,t)}|^{2}\rrbracket
=|\nabla u(x,t)|^{2}+
\mathbb{E}\left\llbracket \left|\int_{\mathbf{Q}}\nabla
h(x-y,t)\otimes\mathscr{J}(y)d\mu_{n}(y)\right|^{2}\right\rrbracket\nonumber\\&
\le |\nabla u(x,t)|^{2}+\left(\int_{\mathbf{Q}}
|\nabla h(x-y,t)|^{2}d\mu_{n}(y)\right)\left(\int_{\mathbf{Q}}
\mathbb{E}\llbracket|\mathscr{J}(y)|^{2}\rrbracket d\mu_{n}(y)\right)\nonumber\\&
=|\nabla u(x,t)|^{2}+\left(\int_{\mathbf{Q}}
|\nabla h(x-y,t)|^{2}d\mu_{n}(y)\right)\zeta v({\mathbf{Q}})
\end{align}
and
\begin{align}
&\mathbb{E}\llbracket|\widehat{u(x,t)}|^{2}\rrbracket
=|\nabla u(x,t)|^{2}+
\mathbb{E}\left\llbracket \left|\int_{\mathbf{Q}}
h(x-y,t)\otimes\mathscr{J}(y) d\mu_{n}(y)\right|^{2}
\right\rrbracket\nonumber\\&\le
|u(x,t)|^{2}+\left(\int_{\mathbf{Q}}
h(x-y,t)|^{2}d\mu_{n}(y)\right)\left(\int_{\mathbf{Q}}
\mathbb{E}\llbracket|\mathscr{J}(y)|^{2}\rrbracket d\mu_{n}(y)\right)\nonumber\\&
=|u(x,t)|^{2}+\left(\int_{\mathbf{Q}}
h(x-y,t)|^{2}d\mu_{n}(y)\right)\zeta v(\mathbf{Q})
\end{align}
If the Li-Yau-Harnack (6.36) inequality holds then
\begin{align}
\frac{|\nabla{u(x,t)}|^{2}+\zeta v(\mathbf{Q})
\int_{\mathbf{Q}}\nabla h(x-y,t)|^{2}d\mu_{n}(y)}{
|{u(x,t)}|^{2}+\zeta v(\mathbf{Q})
\int_{\mathbf{Q}}h(x-y,t)|^{2}d\mu_{n}(y)}-\frac{\frac{\partial}{\partial t}u(x,t)}{u(x,t)}
\le \frac{1}{2}nt^{-1}
\end{align}
Since $\zeta v({\mathbf{Q}})\int_{\mathbf{Q}}h(x,y-t)d^{n}y\ge 0 $ for all
$(x,y)\in {\mathbf{Q}}$ and $t>0$
\begin{align}
&\frac{|\nabla{u(x,t)}|^{2}+\zeta V(\mathbf{Q})
\int_{\mathbf{Q}}\nabla h(x-y,t)|^{2}d\mu_{n}(y)}{
|{u(x,t)}|^{2}+\zeta v(\mathbf{Q})
\int_{\mathbf{Q}}h(x-y,t)|^{2}d\mu_{n}(y)}-\frac{\frac{\partial}{\partial t}u(x,t)}{u(x,t)}\nonumber\\&
\le
\frac{|\nabla{u(x,t)}|^{2}+\zeta v(\mathbf{Q})
\int_{\mathbf{Q}}\nabla h(x-y,t)|^{2}d\mu_{n}(y)}{
|{u(x,t)}|^{2}}-\frac{\frac{\partial}{\partial t}u(x,t)}{u(x,t)}\nonumber\\&
=\frac{|\nabla{u(x,t)}|^{2}}{
|{u(x,t)}|^{2}}+\underbrace{\frac{\zeta v(\mathbf{Q})
\int_{\mathbf{Q}}\nabla h(x-y,t)|^{2}d\mu_{n}(y)}{|u(x,t)|^{2}}}_{positive~term}
-\frac{\frac{\partial}{\partial t}u(x,t)}{u(x,t)}\le\frac{1}{2}n t^{-1}
\end{align}
Hence
\begin{align}
\frac{|\nabla{u(x,t)}|^{2}}{
|{u(x,t)}|^{2}}-\frac{\frac{\partial}{\partial t}u(x,t)}{u(x,t)}\le\frac{1}{2}n t^{-1}-\underbrace{\frac{\zeta v(\mathbf{Q})
\int_{\mathbf{Q}}\nabla h(x-y,t)|^{2}d\mu_{n}(y)}{|u(x,t)|^{2}}}\le \frac{1}{2}nt^{-1}
\end{align}
where we have used the simple inequality $\frac{A+B}{C+D}\le \frac{A+B}{C}=\frac{A}{C}+\frac{B}{C}$ if $A,B,C,D\ge 0$. Note also that
\begin{align}
&\int_{\mathbf{Q}}|\nabla|h|^{2}d\mu_{n}(y)=\frac{1}{4\pi t}^{n/2}
\int_{\mathbf{Q}}\left|\nabla\exp\left(-\frac{\|x-y\|^{2}}{4t}\right)\right|^{2}
d\mu_{n}(y)\nonumber\\&=\frac{1}{4\pi t}^{n/2}
\int_{\mathbf{Q}}\left|\frac{-\|x-y\|}{2t}\exp\left(-\frac{\|x-y\|^{2}}{4t}\right)\right|^{2}
d\mu_{n}(y)\ge 0
\end{align}
\end{proof}
\begin{thm}($\underline{Stochastic~parabolic~Harnack inequality}$)
Let $u(x,t)$ and $u(y,t)$ be solutions of the heat equations ${\square}u(x,t)=0=\frac{\partial}{\partial t}-{\Delta}_{x}u(x,t)=0$ and ${\square}u(y,t)=0=\frac{\partial}{\partial t}-{\Delta}_{y}u(y,t)=0$ with initial data $u(x,0)=u(y,0)=\phi(x)$. If $0<t_{1}<t_{2}$ then the parabolic Harnack estimate is given by Thm 3.5 so that $|u(y,t_{2})|^{2}\ge
|u(x,t_{1})|^{2}\left|\tfrac{t_{1}}{t_{2}}\right|^{n}
\exp(-\tfrac{|x-y|^{2}}{2t})$. If the initial Cauchy data is randomly perturbed
as $\widehat{\phi(x)}=\phi(x)+\mathscr{J}(x)$ then the perturbed solution
$\widehat{u(x,t)}$ solves $\left(\tfrac{\partial}{\partial t}
-{\Delta}_{x}\right)\widehat{u(x,t)}=0$ and similarly $\widehat{u(y,t)}$ solves $\left(\frac{\partial}{\partial t}-\Delta_{y}\right)\widehat{u(y,t)}=0$. Then the stochastic parabolic Harnack inequality is
\begin{align}
\mathbb{E}\llbracket|\widehat{u(y,t_{2})}|^{2}\rrbracket\ge
\mathbb{E}\llbracket|\widehat{u(x,t_{1})}|^{2}\rrbracket\left|\frac{t_{1}}{t_{2}}\right|^{n}
\exp\left(-\tfrac{|x-y|^{2}}{2t}\right)
\end{align}
provided that $\exp\left(-\tfrac{-|x-y|^{2}}{2t}\right)\le 1$ which is always true since
$\exp(-X)\in[0,1]$ for $X\in[0,\infty]$.
\end{thm}
\begin{proof}
The expectations of the squares (volatility) of the stochastic convolution integrals are
\begin{align}
&\mathbb{E}\llbracket |\widehat{u(x,t)}|^{2}\rrbracket =|u(x,t)|^{2}
+\mathbb{E}\left\llbracket
\left|\int_{{\mathbf{Q}}}h(x,z,t)\otimes\mathscr{J}(z)d^{n}z\right|^{2}
\right\rrbracket
\\&\mathbb{E}\llbracket |\widehat{u(y,t)}|^{2}\rrbracket =|u(y,t)|^{2}
+\mathbb{E}\left\llbracket
\left|\int_{{\mathbf{Q}}}h(y,z,t)\otimes\mathscr{J}(z)d^{n}z\right|^{2}
\right\rrbracket
\end{align}
Using Cauchy Schwartz
\begin{align}
&\mathbb{E}\llbracket |\widehat{u(x,t)}|^{2}\rrbracket =|u(x,t)|^{2}
+\mathbb{E}\left\llbracket
\left|\int_{{\mathbf{Q}}}h(x,z,t)\otimes\mathscr{J}(z)d^{n}z\right|^{2}
\right\rrbracket\nonumber\\&
\le |u(x,t)|^{2}+\left(\int_{{\mathbf{Q}}}|h(x,z,t)|^{2}d^{n}z\right)
\mathbb{E}\left\llbracket\left(\int_{{\mathbf{Q}}}|\mathscr{J}(z)|^{2}
d^{n}z\right)\right\rrbracket\nonumber\\&
\equiv |u(x,t)|^{2}+\left(\int_{{\mathbf{Q}}}|h(x,z,t)|^{2}d^{n}z\right)
\left(\int_{{\mathbf{Q}}}\mathbb{E}
\llbracket|\mathscr{J}(z)|^{2}\rrbracket d^{n}z\right)\nonumber\\&
\equiv |u(x,t)|^{2}+\left(\int_{{\mathbf{Q}}}|h(x,z,t)|^{2}d^{n}z\right)
\zeta\left(\int_{{\mathbf{Q}}}d^{n}z\right)\nonumber\\&
\equiv |u(x,t)|^{2}+\left(\int_{{\mathbf{Q}}}|h(x,z,t)|^{2}d^{n}z\right)
\zeta v({\mathbf{Q}})\nonumber\\&
\le |u(x,t)|^{2}+\zeta h(x,x,2t)v({\mathbf{Q}})= |u(x,t)|^{2}
+\frac{\zeta v({\mathbf{Q}})}{|8\pi t|^{n/2}}
\end{align}
using the semi-group property from Lemma 2.21. Similarly,
\begin{align}
&\mathbb{E}\llbracket |\widehat{u(y,t)}|^{2}\rrbracket \le |u(y,t)|^{2}
+\frac{\zeta v({\mathbf{Q}})}{|8\pi t|^{n/2}}
\end{align}
The estimate (1.37) is then
\begin{align}
\left(|u(y_{2},t)|^{2}+\frac{\zeta v({\mathbf{Q}})}{|8\pi t|^{n/2}}
\right)\ge\left|\frac{t_{1}}{t_{2}}\right|^{n}\left(|u(x_{2},t)|^{2}+\frac{\zeta v(\bm{\mathbf{Q}})}{|8\pi t|^{n/2}}\right)\exp\left(-\frac{|x-y|}{2|t_{2}-t_{1}|}\right)
\end{align}
Since $|u(y,t_{2})|^{2}\ge |u(x,t_{1})|^{2}\left|\tfrac{t_{1}}{t_{2}}\right|^{n}
\exp(-\tfrac{|x-y|^{2}}{2t})$, this is equivalent to
\begin{align}
\frac{|8\pi t_{1}|^{n/2}}{|8 \pi t_{2}|^{n/2}}
\ge\left|\frac{t_{1}}{t_{2}}\right|^{n}\exp\left(-\frac{|x-y|}{2|t_{2}-t_{1}|}\right)
\end{align}
or
\begin{align}
\left|\frac{t_{1}}{t_{2}}\right|^{n/2}
\ge\left|\frac{t_{1}}{t_{2}}\right|^{n}\exp\left(-\frac{|x-y|}{2|t_{2}-t_{1}|}\right)
> \left|\frac{t_{1}}{t_{2}}\right|^{n/2}\exp\left(-\frac{|x-y|}{2|t_{2}-t_{1}|}\right)
\end{align}
which is simply $\exp\left(-\frac{|x-y|}{2|t_{2}-t_{1}|}\right)\le 1$ which is always true for all $|x-y|\ge 0$ and $|t_{2}-t_{1}|>0$. Hence, the estimate (6.45) is valid.
\end{proof}
\section{Estimates and decay properties for p-moments and volatility on $\mathbf{Q}$, $\mathrm{R}^{+}$,$\mathbf{B}_{R}(0)$}
In this section, the p-moments $\mathbb{E}\llbracket|\widehat{u(x,t)}|^{p}\rrbracket $ and volatility $\mathbb{E}\llbracket|\widehat{u(x,t)}|^{2}\rrbracket $ are estimated from the general solution $\widehat{u(x,t)}$ to the stochastic Cauchy initial-value problem on a finite domain $\mathbf{Q}\subset\mathbf{R}^{n}$. The estimates for a general domain $\mathbf{Q}$ is first derived, then the specific cases are calculated explicitly for a segment on the real line $\mathbf{Q}=\mathbf{L}\subset{\mathbf{R}}^{+}$ with $\mathbf{L}=[0,L]$, a ball $\mathbf{B}_{R}(0)\subset{\mathbf{R}}^{n}$ of radius R. The Cauchy evolution or development from the random initial data, is then stable or unstable and will either decay or grow without bound, or blow up with time.
\begin{prop}
(Reprising Proposition 1.1). Given the SCIVP  on $\mathbf{Q}\subset\bm{\mathbf{R}}^{n}$ with Gaussian random initial data $\widehat{u(x,0)}=\widehat{\phi(x)}=\phi(x)+\mathscr{J}(x)$ and solution $\widehat{u(x,t)}$ then:
\begin{enumerate}
\item The solution is stable and is smoothed out and decays if
\begin{align}
\lim_{t\uparrow\infty}\mathbb{E}\llbracket|\widehat{u(x,t)}|^{p}=0,~~
\lim_{t\uparrow\infty}\mathbb{E}\llbracket|\widehat{u(x,t)}|^{2}=0
\end{align}
or
\begin{align}
\lim_{t\uparrow \infty}\mathbb{E}\llbracket|\widehat{u(x,t)}|^{p}\le K,~~
\lim_{t\uparrow \infty}\mathbb{E}\llbracket|\widehat{u(x,t)}|^{2}\le K
\end{align}
for some $K>0$
\item The solution is unstable and grows without bound if
\begin{align}
\lim_{t\uparrow \infty}\mathbb{E}\llbracket|\widehat{u(x,t)}|^{p}=\infty,~~
\lim_{t\uparrow \infty}\mathbb{E}\llbracket|\widehat{u(x,t)}|^{2}=\infty
\end{align}
\item The solution blows up if for some finite $t_{B}\in[0,\infty)$
\begin{align}
\lim_{t\uparrow t_{B}}\mathbb{E}\llbracket|\widehat{u(x,t)}|^{p}=\infty,~~
\lim_{t\uparrow t_{B}}\mathbb{E}\llbracket|\widehat{u(x,t)}|^{2}=\infty
\end{align}
\end{enumerate}
\end{prop}
One expects scenario (1) as the heat equation should smooth out the initial randomness so that the moments and volatility decay to zero.
\begin{thm}
Let $\widehat{u(x,t)}$ satisfy the stochastic Cauchy initial-value problem on a finite domain $\mathbf{Q}\subset\mathbf{R}^{n}$ with random initial Cauchy data, for the heat equation
\begin{align}
&\frac{\partial}{\partial t}\widehat{u(x,t)}={\Delta} \widehat{u(x,t)},~~(x\in\mathbf{Q},t\in[0,\infty)\nonumber\\&
\widehat{u(x,0)}=\phi(x)+\mathscr{J}(x,~~(x\in\mathbf{Q},t=0)
\end{align}
with Gaussian random scalar field $\mathscr{J}(x)$ and $\mathbb{E}\llbracket|\mathscr{J}(x)|^{p}\rrbracket =\frac{1}{2}[\zeta^{\frac{p}{2}}+(-1)^{p}\zeta^{\frac{p}{2}}]$. Then the moments are finite and bounded with the estimate
\begin{align}
&\mathbb{E}\big\llbracket|\widehat{u(x,t)}|^{p}\big\rrbracket\le 2^{p-1}\left(\int_{\mathbf{Q}}|\phi(y)|^{p}d^{n}p\right)+\frac{1}{2}
[\zeta^{p/2}+(-1)^{p}\zeta^{p/2}]
v(\mathbf{Q}\left(\int_{\mathbf{Q}}{(4\pi t))^{-n/2}}e^{-\frac{|x-y|^{2}}{4t}}d\mu_{n}(y)\right)^{p}\nonumber\\&
\equiv 2^{p-1}\left(\int_{\mathbf{Q}}|\phi(y)|^{p}d^{n}p\right)+
\frac{1}{2}[\zeta^{p/2}+(-1)^{p}\zeta^{p/2}]
v(\mathbf{Q}\left(\int_{\mathbf{Q}} h(x-y,t)d\mu_{n}(y)
\right)^{p}
\end{align}
If $\phi(x)=C$ then
\begin{align}
&\mathbb{E}\big\llbracket|\widehat{u(x,t)}|^{p}\big\rrbracket\le 2^{p-1}|C|^{p}v(\mathbf{Q})+\frac{1}{2}[\zeta^{p/2}+(-1)^{p}\zeta^{p/2}]
v(\mathbf{Q}\left(\int_{\mathbf{Q}}\frac{1}{(4\pi t))^{n/2}}\exp\left(-\frac{|x-y|^{2}}{4t}\right)d\mu_{n}(y)\right)^{p}\nonumber\\&
\equiv 2^{p-1}|C|^{p}v(\mathbf{Q})+\frac{1}{2}[\zeta^{p/2}+(-1)^{p}\zeta^{p/2}]
\left(\int_{\mathbf{Q}} h(x-y,t)d\mu_{n}(y)\right)^{p}
\end{align}
\end{thm}
\begin{proof}
Using the basic estimate $|a+b|^{p}\le 2^{p-1}|a|^{p}+2^{p-1}|b|^{p}$ and the results of Lemma (5.5)
\begin{align}
\bm{\mathcal{M}}_{p}(x,t)&=\mathbb{E}\llbracket|\widehat{u(x,t)}|^{p}\rrbracket
=\mathbb{E}\left\llbracket\left|\int_{\mathbf{Q}} h(x-y,t)\phi(y)d\mu_{n}(y)+\int_{\mathbf{Q}} h(x-y,t)\otimes
\mathscr{J}(x)d\mu_{n}(y)\right|^{p}\right\rrbracket\nonumber\\&
\le\mathbb{E}\left\llbracket\frac{1}{2^{p}}\left|\int_{\mathbf{Q}}
h(x-y,t)\phi(y)d\mu_{n}(y)\right|^{p}+
\frac{1}{2^{p}}\left|\int_{\mathbf{Q}} h(x-y,t)\otimes
\mathscr{J}(x)d\mu_{n}(y)\right|^{p}\right\rrbracket\nonumber\\&
\le\frac{1}{2^{p}}\underbrace{\left|\int_{\mathbf{Q}}h(x-y,t)\phi(y)
d\mu_{n}(y)\right|^{p}}_{apply~Holder~inequality}+
\frac{1}{2^{p}}\underbrace{\mathbb{E}\left\llbracket\left|\int_{\mathbf{Q}}
h(x-y,t)\otimes\mathscr{J}(x)d\mu_{n}(y)\right|^{p}\right\rrbracket}_{apply~stoc.~Holder~inequality}
\nonumber\\&
\le \frac{1}{2^{p}}\left(\int_{\mathbf{Q}}|h(x-y,t)|^{q}
\right)^{p/q}\int_{\mathbf{Q}}|\phi(y)|^{p}d\mu_{n}(y)\nonumber\\&+\frac{1}{2^{p}}
\left(\int_{\mathbf{Q}}|h(x-y,t)|^{q}
\right)^{p/q}\int_{\mathbf{Q}}\mathbb{E}\llbracket |\mathscr{J}(y)|^{p}\rrbracket d\mu_{n}(y)
\nonumber\\&\equiv
\frac{1}{2^{p}}\left(\int_{\mathbf{Q}}|h(x-y,t)|^{q}
\right)^{p/q}\left(\int_{\mathbf{Q}}|\phi(y)|^{p}d\mu_{n}(y)+\int_{\mathbf{Q}}
\mathbb{E}\llbracket |\mathscr{J}(y)|^{p}\rrbracket d\mu_{n}(y)\right)\nonumber\\&=
\frac{1}{2^{p}}\left(\int_{\mathbf{Q}}|h(x-y,t)|^{q}
\right)^{p/q}\left(\int_{\mathbf{Q}}|\phi(y)|^{p}d\mu_{n}(y)+\frac{1}{2}[\zeta^{p/2}+f
\zeta^{p/2}]
\int_{\mathbf{Q}}d\mu_{n}(y)\right)\nonumber\\&=
\frac{1}{2^{p}}\left(\int_{\mathbf{Q}}|h(x-y,t)|^{q}
\right)^{p/q}\left(\int_{\mathbf{Q}}|\phi(y)|^{p}d\mu_{n}(y)+\frac{1}{2}[\zeta^{p/2}+(-1)^{p}\zeta^{p/2}]
v({\mathbf{Q}})\right)\nonumber\\&\equiv
\left(\|h(x-\bullet,t)\|_{L_{q}(\mathbf{Q}}\right)^{p}\left(
\|\phi(\bullet)\|_{L_{p}(\mathbf{Q})}+\frac{1}{2}[\zeta^{p/2}+(-1)^{p}\zeta^{p/2}]
v({\mathbf{Q}})
\right)
\end{align}
where $\tfrac{1}{p}+\tfrac{1}{q}=1$. Since $\lim_{t\uparrow\infty}\|
(x-\bullet,t)\|_{L_{q}(\mathbf{Q}}=0$ it follows that the p-moments are smoothed and decay as $t\rightarrow\infty$ so that $\lim_{t\uparrow\infty}\mathbb{E}\llbracket
|\widehat{u}(x,t)|^{p}\rrbracket=0$.
\end{proof}
\begin{cor}
If $\phi(x)=C$ then
\begin{equation}
\mathbb{E}\big\llbracket|\widehat{u(x,t)}|^{p}\big\rrbracket \le 2^{p-1}|C|^{p}V(\bm{\mathbf{Q}})+
\frac{1}{2}[\zeta^{p/2}+(-1)^{p}\zeta^{p/2}]v(\mathbf{Q})
\left(\int_{\mathbf{Q}}{(4 \pi t)^{-n/2}}\exp\left(-\frac{|x-y|^{2}}{4t}\right)d\mu_{n}(y)\right)^{p}
\end{equation}
\end{cor}
A second estimate can be derived via the binomial theorem.
\begin{thm}
Let $\widehat{u(x,t)}$ satisfy the stochastic Cauchy initial-value problem on a finite domain $\mathbf{Q}\subset\mathbf{R}^{n}$ with random initial Cauchy data, for the heat equation
\begin{align}
&\frac{\partial}{\partial t}\widehat{u(x,t)}={\Delta} \widehat{u(x,t)},~~(x\in\mathbf{Q},t\in[0,\infty)\nonumber\\&
\widehat{u(x,0)}=\phi(x)+\mathscr{J}(x)=C+\mathscr{J}(x),~~(x\in\mathbf{Q},t=0)
\end{align}
with Gaussian random scalar field $\mathscr{J}(x)$ and $\mathbb{E}\llbracket|\mathscr{J}(x)|^{p}\rrbracket =\frac{1}{2}[\zeta^{\frac{p}{2}}+(-1)^{p}\zeta^{\frac{p}{2}}]$. Then the moments are finite and bounded with the estimate
\begin{align}
\mathbb{E}\llbracket|\widehat{u(x,t)}|^{p}\rrbracket&\le \frac{1}{2}\sum_{\beta=0}^{p}\binom{p}{\beta}|C|^{p-\beta}(\zeta^{\frac{p}{2}}+(-1)^{p}
\zeta^{\frac{p}{2}})|v(\bm{\mathbf{Q}})|^{p}\left|\int_{\mathbf{Q}}{(4\pi t)^{-n/2}}e^{-\frac{|x-y|^{2}}{4t}}d\mu_{n}(y)\right|^{2p-\beta}\nonumber\\&\equiv \frac{1}{2}\sum_{\beta=0}^{p}\binom{p}{\beta}|C|^{p-\beta}(\zeta^{\frac{p}{2}}+(-1)^{p}
\zeta^{\frac{p}{2}})|{v}({\mathbf{Q}})|^{p}\left|\int_{\mathbf{Q}}
h(x-y,t)|d\mu_{n}(y)\right|^{2p-\beta}
\end{align}
and it is known that $\int_{\mathbf{Q}}h(x-y,t)|d^{n}y<\infty$ for any finite domain ${\mathbf{Q}}$.
\end{thm}
\begin{proof}
Using the binomial theorem
\begin{align}
(X+Y)^{p}=\sum_{\beta=1}^{p}\binom{p}{\beta}X^{p-\beta}Y^{p}\nonumber
\end{align}
with $\tfrac{1}{p}+\tfrac{1}{q}=1$.
\begin{align}
&\mathbb{E}\llbracket|\widehat{u(x,t)}|^{p}\rrbracket
=\mathbb{E}\left\llbracket\left|\int_{\mathbf{Q}}{C}
h(x-y,t)d\mu_{n}(y)+\int_{\mathbf{Q}}{C}h(x-y,t)\otimes\mathscr{J}(y)
d\mu_{n}(y)\right|^{p}\right\rrbracket
\nonumber\\&
=\mathbb{E}\left\llbracket \sum_{\beta=1}^{p}\binom{p}{\beta}\left|\int_{\mathbf{Q}}{C}h(x-y,t)
d\mu_{n}(y)\right|^{p-\beta}\left|\int_{\mathbf{Q}}h(x-y,t)\otimes\mathscr{J}(y)
d\mu_{n}(y)\right|^{p}\right\rrbracket
\nonumber\\&
\equiv \sum_{\i=1}^{p}\binom{p}{\beta}\left|\int_{\mathbf{Q}}{C}h(x-y,t)d\mu_{n}(y)\right|^{p-\xi}
\mathbb{E}\left\llbracket \left|\int_{\mathbf{Q}}h(x-y,t)\otimes \mathscr{J}(y)d\mu_{n}(y)\right|^{p}\right\rrbracket
\nonumber\\&
\equiv \sum_{\i=1}^{p}\binom{p}{\beta}C^{p-\beta}\left|\int_{\mathbf{Q}}h(x-y,t)d\mu_{n}(y)\right|^{p-\beta}
\mathbb{E}\left\llbracket \left|\int_{\mathbf{Q}}h(x-y,t)\otimes\mathscr{J}(y)d\mu_{n}(y)\right|^{p}\right\rrbracket
\nonumber\\&=
\sum_{\beta=1}^{p}\binom{p}{\beta}C^{p-\beta}\left|\int_{\mathbf{Q}}h(x-y,t))
d\mu_{n}(y)\right|^{p-\beta}\nonumber\\&\times\bigg( \int_{\mathbf{Q}}\mathbb{E}\llbracket|\mathscr{J}(y)|^{p}\rrbracket d\mu_{n}(y)\bigg)\int_{\mathbf{Q}}\bigg(\bigg|h(x-y,t)d\mu_{n}(y)\bigg|^{q}\bigg)^{p/q}\nonumber\\&\equiv \sum_{\i=1}^{p}\binom{p}{\beta}C^{p-\beta}\left|\int_{\mathbf{Q}}h(x-y,t)d\mu_{n}(y)\right|^{p-\beta}\nonumber\\&\times\bigg( \int_{\mathbf{Q}}\mathbb{E}\llbracket|\mathscr{J}(y)|^{p}\rrbracket d\mu_{n}(y)\bigg)\int_{\mathbf{Q}}\bigg(\bigg|h(x-y,t)|^{\frac{p}{p-1}}d\mu_{n}(y)\bigg)^{p-1}
\nonumber\\&\equiv \frac{1}{2}\sum_{\beta=1}^{p}\binom{p}{\beta}C^{p-\beta}\left|\int_{\mathbf{Q}}h(x-y,t)
d\mu_{n}(y)\right|^{p-\beta}\nonumber\\&\times[\zeta^{\frac{p}{2}}+(-1)^{p}
\zeta^{\frac{p}{2}}]\bigg(\int_{\mathbf{Q}} d\mu_{n}(y)\bigg)^{p}\int_{\mathbf{Q}}\bigg(\bigg|h(x-y,t)|^{\frac{p}{p-1}}d\mu_{n}(y)\bigg)^{p-1}\nonumber\\&\equiv \frac{1}{2}\sum_{\beta=1}^{p}\binom{p}{\beta}C^{p-\beta}\left|\int_{\mathbf{Q}}h(x-y,t)
d\mu_{n}(y)\right|^{p-\beta}\nonumber\\&\times[\beta^{\frac{p}{2}}+(-1)^{p}
\beta^{\frac{p}{2}}]|v(\mathbf{Q})|^{p}\bigg(\bigg|h(x-y,t)|^{p/p-1}d\mu_{n}(y)\bigg)^{p-1}\nonumber\\&\equiv \frac{1}{2}\sum_{\beta=1}^{p}\binom{p}{\beta}C^{p-\beta}\left|\int_{\mathbf{Q}} h(x-y,t)d^{n}y\right|^{p-\beta}\nonumber\\&\times[\zeta^{\frac{p}{2}}+(-1)^{p}
\zeta^{\frac{p}{2}}]|v(\mathbf{Q})|^{p}\bigg(\int_{\mathbf{Q}}h(x-y,t)d\mu_{n}(y)\bigg)^{p}\nonumber\\&\equiv \frac{1}{2}\sum_{\beta=1}^{p}\binom{p}{\beta}C^{p-\beta}[\zeta^{\frac{p}{2}}+(-1)^{p}
\zeta^{\frac{p}{2}}]|v(\mathbf{Q})|^{p}\bigg(\int_{\mathbf{Q}}h(x-y,t)d\mu_{n}(y)\bigg)^{2p-\beta}\nonumber\\&\equiv \frac{1}{2}\sum_{\beta=1}^{p}\binom{p}{\beta}C^{p-\beta}[\zeta^{\frac{p}{2}}+(-1)^{p}
\zeta^{\frac{p}{2}}]|v(\mathbf{Q})|^{p}\bigg(\int_{\mathbf{Q}}
h(x-y,t)d\mu_{n}(y)\bigg)^{2p-\beta}
\end{align}
\end{proof}
\begin{cor}
The volatility follows for $p=2$ so that
\begin{align}
\mathbb{E}\llbracket|\widehat{u(x,t)}|^{2}\rrbracket&\le \frac{1}{2}\sum_{\beta=0}^{2}\binom{2}{\beta}|C|^{2-\beta}(\zeta^{\frac{2}{2}}+(-1)^{2}
\zeta^{\frac{2}{2}})|v(\mathbf{Q})|^{2}\left|\int_{\mathbf{Q}}\frac{1}{(4\pi t)^{n/2}}\exp\left(-\frac{|x-y|^{2}}{4t}\right)d\mu_{n}(y)\right|^{4-\beta}\nonumber\\&\equiv \frac{1}{2}\sum_{\beta=0}^{2}\binom{2}{\beta}|C|^{2-\beta}(\zeta^{\frac{2}{2}}+(-1)^{2}
\zeta^{\frac{2}{2}})|v(\mathbf{Q})|^{2}\left|\int_{\mathbf{Q}}
h(x-y,t)|d\mu_{n}(y)\right|^{4-\beta}\nonumber\\&\equiv\sum_{\beta=0}^{2}
\binom{2}{\beta}|C|^{2-\beta}\zeta|v(\mathbf{Q})|^{2}\left|\int_{\mathbf{Q}}
h(x-y,t)|d\mu_{n}(y)\right|^{4-\beta}
\end{align}
\end{cor}
\begin{cor}
Since
\begin{equation}
\lim_{t\uparrow \infty}\frac{1}{(4\pi t)^{n/2}}\int_{\mathbf{Q}}\exp\left(-\frac{|x-y|^{2}}{4t}\right)d^{n}y=0
\end{equation}
it follows that the moments decay to zero as $t\rightarrow \infty$ so the convolution integral solution is stable
\begin{align}
&\lim_{t\uparrow \infty}\mathbb{E}\llbracket|\widehat{u(x,t)}|^{p}\rrbracket\nonumber\\&\le \frac{1}{2}\sum_{\beta=0}^{p}\binom{p}{\beta}|C|^{p-\beta}(\zeta^{\frac{p}{2}}+(-1)^{p}
\zeta^{\frac{p}{2}})|v(\mathbf{Q})|^{p}\lim_{t\uparrow \infty}\left|\int_{\mathbf{Q}}\frac{1}{(4\pi t)^{n/2}}\exp\left(-\frac{|x-y|^{2}}{4t}\right)d\mu_{n}(y)\right|^{2p-\beta}\nonumber\\&\equiv \frac{1}{2}\sum_{\beta=0}^{p}\binom{p}{\beta}|C|^{p-\beta}(\zeta^{\frac{p}{2}}+(-1)^{p}
\zeta^{\frac{p}{2}})|v(\mathbf{Q})|^{p}\lim_{t\uparrow \infty}\left|\int_{\mathbf{Q}}h(x-y,t)|d\mu_{n}(y)
\right|^{2p-\beta}=0
\end{align}
Similarly, for the volatility (p=2)
\begin{align}
&\lim_{t\uparrow \infty}\mathbb{E}\llbracket|\widehat{u(x,t)}|^{2}\rrbracket\nonumber\\&\le \frac{1}{2}\sum_{\beta=0}^{2}\binom{2}{\beta}|C|^{2-\beta}(\zeta^{\frac{2}{2}}+(-1)^{2}
\zeta^{\frac{2}{2}})|v(\mathbf{Q})|^{2}\lim_{t\uparrow \infty}\left|\int_{\mathbf{Q}}\frac{1}{(4\pi t)^{n/2}}\exp\left(-\frac{|x-y|^{2}}{4t}\right)d\mu_{n}(y)\right|^{4-\beta}\nonumber\\&\equiv \frac{1}{2}\sum_{\beta=0}^{2}\binom{2}{\beta}|C|^{2-\beta}(\zeta^{\frac{2}{2}}+(-1)^{2}
\zeta^{\frac{2}{2}})|v(\mathbf{Q})|^{2}\lim_{t\uparrow \infty}\left|\int_{\mathbf{Q}}h(x-y,t)|d\mu_{n}(y)
\right|^{4-\beta}\nonumber\\&\equiv\sum_{\beta=0}^{2}\binom{2}{\beta}|C|^{2-\beta}\zeta
|v(\mathbf{Q})|^{2}\left|\int_{\mathbf{Q}}h(x-y,t)|d\mu_{n}(y)
\right|^{4-\beta}=0
\end{align}
\end{cor}
The moments and volatility can be computed on an interval of the real line since the integral is simply Gaussian.
\begin{cor}
If ${\mathbf{Q}}={\mathbf{L}}=[0,L]$
\begin{align}
\mathbb{E}\llbracket|\widehat{u(x,t)}|^{p}\rrbracket&\le \frac{1}{2}\sum_{\beta=0}^{p}\binom{p}{\beta}|C|^{p-\beta}(\zeta^{\frac{p}{2}}+(-1)^{p}
\zeta^{\frac{p}{2}})|L|^{p}\left|\int_{0}^{L}{(4\pi t)^{-1/2}}\exp\left(-\frac{|x-y|^{2}}{4t}\right)dy\right|^{2p-\beta}\nonumber\\&\equiv \frac{1}{2}\sum_{\beta=0}^{p}\binom{p}{\beta}|C|^{p-\beta}(\zeta^{\frac{p}{2}}+(-1)^{p}
\zeta^{\frac{p}{2}})|L|^{p}\left|\int_{0}^{L}h(x-y,t)|dy\right|^{2p-\beta}\nonumber\\&
=\frac{1}{2}\sum_{\beta=0}^{p}\binom{p}{\beta}|C|^{p-\beta}(\zeta^{\frac{p}{2}}+(-1)^{p}
\zeta^{\frac{p}{2}})|L|^{p}\left(\frac{erf(\frac{x}{2\sqrt{t}})-erf(\frac{x-L}{2\sqrt{t}})}
{4^{1/2}\pi t}\right)
\end{align}
where $erf(x)$ is the error function with $erf(0)=0$ and $erf(\infty)=1$. And so
\begin{align}
&\lim_{t\uparrow \infty}\mathbb{E}\llbracket|\widehat{u(x,t)}|^{p}\rrbracket=\lim_{t\uparrow \infty}\frac{1}{2}\sum_{\beta=0}^{p}\binom{p}{\beta}|C|^{p-\beta}(\zeta^{\frac{p}{2}}+(-1)^{p}
\zeta^{\frac{p}{2}})|L|^{p}\left(\frac{erf(\frac{x}{2\sqrt{t}})-erf(\frac{x-L}{2\sqrt{t}})}
{4^{1/2}\pi t}\right)=0\\&
\lim_{t\uparrow \infty}\mathbb{E}\llbracket|\widehat{u(x,t)}|^{2}\rrbracket=\lim_{t\uparrow \infty}\sum_{\beta=0}^{2}\binom{2}{\beta}|C|^{2-\beta}\zeta |L|^{2}\left(\frac{erf(\frac{x}{2\sqrt{t}})-erf(\frac{x-L}{2\sqrt{t}})}
{4^{1/2}\pi t}\right)=0
\end{align}
\end{cor}
\begin{thm}(\underline{Double-sided bound on volatility and p-moments})
Let $u(x,t)$ be a solution of the (deterministic) heat equation $\square
u(x,t)=0$ for $x\in{\mathbf{Q}},t>0$, with initial Cauchy data $\psi(x,0)=\phi(x)$. Then $u(x,t)=\int_{{\mathbf{Q}}}h(x-y,t)\phi(y)d\mu_{n}(y)$. Now consider the randomly
perturbed solution $\widehat{u}(x,t)$ which satisfies $\square\widehat{u(x,t)}=0$ for random initial Cauchy data $\widehat{\phi(x)}=\mathscr{J}(x)$. The perturbed solution is
$\widehat{u(x,t)}=\int_{{\mathbf{Q}}}h(x-y,t)\otimes\mathscr{J}(y)d\mu_{n}(y)$.
If the following double-sided integral bounds from Lemma (2.5) hold
\begin{align}
&\Lambda_{1}^{2}t^{-n}\int_{{\mathbf{R}}}\exp\left(-\frac{2|x-y|^{2}}{\Lambda_{1}t}\right)d\mu_{n}(y)~
\le~\int_{{\mathbf{R}}^{n}}|h(x-y;t)|^{2}d^{n}y~
\le~\Lambda_{2}^{2}t^{-n}\int_{{\mathbf{R}}}\exp\left(-2\frac{|x-y|^{2}}{\Lambda_{2}t}\right)
d\mu_{n}(y)
\end{align}
\begin{align}
&\Lambda_{1}^{2}t^{-pn/2}\int_{\bm{\mathrm{R}}}\exp\left(-\frac{p|x-y|^{2}}{\Lambda_{1}t}\right)d\mu_{n}(y)~
\le~\int_{{\mathbf{R}}^{n}}|h(x-y;t)|^{p}d\mu_{n}(y)~
\nonumber\\&\le~\Lambda_{2}^{2}t^{-pn/2}\int_{\bm{\mathrm{R}}}\exp\left(-p\frac{|x-y|^{2}}{\Lambda_{2}t}\right)d\mu_{n}(y)
\end{align}
then the volatility and p-moments satisfy the double-sided bounds
\begin{align}
&\mathbb{E}\left\llbracket\left| \Lambda_{1}t^{-n/2}\int_{\bm{\mathrm{R}}}\exp\left(-\frac{|x-y|^{2}}{\Lambda_{1}t}\right)
\otimes\mathscr{J}(y)d^{n}y\right|^{2}\right\rrbracket~
\le~\mathbb{E}\left\llbracket\left|\int_{\bm{\mathrm{R}}^{n}}|h(x-y;t)|\otimes\mathscr{J}(y)d^{n}y\right|^{2}\right\rrbracket~
\nonumber\\&\le~\mathbb{E}\left\llbracket\left|\Lambda_{2}^{2}t^{-n/2}\int_{\bm{\mathrm{R}}}
\exp\left(-\frac{|x-y|^{2}}{\Lambda_{2}t}\right)\otimes\mathscr{J}(y)d\mu_{n}(y)
\right|^{2}\right\rrbracket
\end{align}
\begin{align}
&\mathbb{E}\left\llbracket\left| \Lambda_{1}t^{-n/2}\int_{\bm{\mathrm{R}}}\exp\left(-\frac{|x-y|^{2}}{\Lambda_{1}t}\right)
\otimes\mathscr{J}(y)d\mu_{n}(y)\right|^{p}\right\rrbracket~
\le~\mathbb{E}\left\llbracket\left|\int_{\mathrm{R}^{n}}|h(x-y;t)|\otimes\mathscr{J}(y)d^{n}y\right|^{p}\right\rrbracket~
\nonumber\\&\le~\mathbb{E}\left\llbracket\left|\Lambda_{2}^{2}t^{-n/2}\int_{\bm{\mathrm{R}}}e^{-p\frac{|x-y|^{2}}{\Lambda_{2}t}}\otimes\mathscr{J}(y)
d\mu_{n}(y)\right|^{p}\right\rrbracket
\end{align}
\end{thm}
\begin{proof}
Apply the Cauchy-Schwarz inequality to each term to get
\begin{align}
&\left(\int_{{\mathbf{Q}}}\Lambda_{1}^{2}t^{-n}\exp\left(-2\frac{|x-y|^{2}}{\Lambda_{1}t}\right)d^{n}y\right)
\left(\int_{\mathbf{Q}}\mathbb{E}\llbracket|\mathscr{J}(x)|^{2}
\rrbracket d\mu_{n}(y)\right)\nonumber\\&\le
\left(\int_{\bm{\mathrm{R}}^{n}}|h(x-y;t)|^{2}d^{n}y\right)\left(\int_{{\mathbf{Q}}}
\mathbb{E}\llbracket|\mathscr{J}(x)|^{2}\rrbracket
d\mu_{n}(y) \right)\nonumber\\&
\le \left(\int_{{\mathbf{Q}}}\Lambda_{2}^{2}t^{-n}\exp\left(-2\frac{|x-y|^{2}}{\Lambda_{2}t}\right)d^{n}y\right)
\left(\int_{{\mathbf{Q}}}\mathbb{E}\llbracket|\mathscr{J}(x)|^{2}
\rrbracket d\mu_{n}(y)\right)
\end{align}
Cancelling the integrals $\int_{{\mathbf{Q}}}\mathbb{E}\llbracket|\mathscr{J}(x)|^{2}
\rrbracket d^{n}y$ then gives (6.20). The bound (6.21) is established in
 the same manner.
\end{proof}
\subsection{Multiplicative random perturbations or noise}
The previous estimates considered additive random perturbations or noise on the initial data such that $\widehat{u(x,0)}=u(x,0)+\mathscr{J}(x)$. Here, the case of
the stochastic CIVP is considered for multiplicative noise such that $\widehat{u(x,t)}=u(x,0)\mathscr{J}(x)$.
\begin{thm}
Let $\widehat{u(x,t)}$ satisfy the stochastic Cauchy initial-value problem on a finite domain ${\mathbf{Q}}\subset{\mathbf{R}}^{n}$ with random initial Cauchy data, for the heat equation
\begin{align}
&\frac{\partial}{\partial t}\widehat{u(x,t)}={\Delta } \widehat{u(x,t)},~~(x\in {\mathbf{Q}},t\in[0,\infty)\nonumber\\&
\widehat{u(x,0)}=\phi(x)\otimes\mathscr{J}(x),~~(x\in{\mathbf{Q}},t=0)
\end{align}
with Gaussian random scalar field $\mathscr{J}(x)$ and $\mathbb{E}\big\llbracket|\mathscr{J}(x)|^{p}\big\rrbracket =\frac{1}{2}[\zeta^{\frac{p}{2}}+(-1)^{p}\zeta^{\frac{p}{2}}]$. Then:
\begin{enumerate}
\item The solution is the stochastic convolution integral
\begin{equation}
\widehat{u(x,t)}={(4 \pi t)^{-n/2}}\int_{{\mathbf{Q}}}\exp\left(-\frac{|x-y|^{2}}{4p}\right)\phi(y)\mathscr{J}(y)
d\mu_{n}(y)\equiv \int_{{\mathbf{Q}}}h(x-y,t)\phi(y)\mathscr{J}(y)d\mu_{n}(y)
\end{equation}
\item The p-moments are finite and bounded with the estimate
\begin{align}
\mathbb{E}\big\llbracket|\widehat{u(x,t)}|^{p}\big\rrbracket&\le \frac{1}{2}[\zeta^{p/2}+(-1)^{p}\zeta^{p/2}]v(\mathbf{Q})\nonumber\\&
\left(\int_{{\mathbf{Q}}}{(4 \pi t)^{-n/2}}\exp\left(-\frac{|x-y|^{2}}{4t}\right)d\mu_{n}(y)
\right)^{p}\left(\int_{\mathbf{Q}}|\phi(y)|d\mu_{n}(y)\right)^{p}<\infty
\end{align}
\item The volatility is given for $p=2$ so that
\begin{equation}
\mathbb{E}\big\llbracket|\widehat{u(x,t)}|^{2}\big\rrbracket\le |C|^{2}\beta v({\mathbf{Q}})
\left(\int_{\mathbf{Q}}{(4 \pi t)^{-n/2}}\exp\left(-\frac{|x-y|^{2}}{4t}\right)d^{n}y
\right)^{2}\left(\int_{{\mathbf{Q}}}|\phi(y)|d\mu_{n}(y)\right)^{2}<\infty
\end{equation}
\end{enumerate}
\end{thm}
\begin{proof}
To prove (1)
\begin{align}
&\frac{\partial}{\partial t}\widehat{u(x,t)}=\int_{{\mathbf{Q}}}\frac{\partial}{\partial t}h(x-y,t)
\phi(y)\otimes\mathscr{J}(y)d\mu_{n}(y)
\nonumber\\&
{\Delta}_{x}\widehat{u(x,t)}
=\int_{\mathbf{Q}}{\Delta}_{x}h(x-y,t)
(y)\otimes\mathscr{J}(y)d\mu_{n}(y)
\end{align}
so that
\begin{equation}
\left(\frac{\partial}{\partial t}-{\Delta}_{x}\right)\widehat{u(x,t)}
=\int_{\mathbf{Q}}\left(\frac{\partial}{\partial t}-{\Delta}_{x}\right)h(x-y,t)\phi(y)\otimes\mathscr{J}(y)d^{n}y=0
\end{equation}
since the fundamental solution or heat kernel $h(x-y,t)$ satisfies the heat equation
To prove (2),use the holder inequality for three functions with $\tfrac{1}{p}+\tfrac{1}{q}+\tfrac{1}{\ell}=1$ and $\ell=q$.
\begin{align}
&\mathbb{E}\big\llbracket\widehat{u(x,t)}\big\rrbracket=\left\llbracket
\left|{(4 \pi t)^{-n/2}}\int_{\mathbf{Q}}\left(-\frac{|x-y|^{2}}{4p}\right)\phi(y)\otimes\mathscr{J}(y)d^{n}y\right|^{p}\right\rrbracket
\nonumber\\&
\le \mathbb{E}\left\llbracket\int_{\mathbf{Q}}|\otimes\mathscr{J}(y)|^{p}d\mu_{n}(y)
\right\rrbracket\left(\int_{{\mathbf{Q}}}\left|{(4\pi t)^{-n/2}}\exp\left(-\frac{|x-y|^{2}}{4t}\right)\right|^{q}d\mu_{n}(y)\right)^{p/q}
\left(\int_{{\mathbf{Q}}}|\phi(y)|^{\ell}
d\mu_{n}(y)
\right)^{p/\ell}\nonumber\\&
\equiv \mathbb{E}\left\llbracket\int_{{\mathbf{Q}}}|\mathscr{J}(y)|^{p}d\mu_{n}(y)
\right\rrbracket\left(\int_{{\mathbf{Q}}}\left|{(4\pi t)^{-n/2}}\exp\left(-\frac{|x-y|^{2}}{4t}\right)\right|^{q}d\mu_{n}(y)\right)^{p/q}\left(\int_{{\mathbf{Q}}}
|\phi(y)|^{q}d\mu_{n}(y)
\right)^{p/q}\nonumber\\&
\equiv \mathbb{E}\left\llbracket\int_{{\mathbf{Q}}}|\mathscr{J}(y)|^{p}d\mu_{n}(y)
\right\rrbracket\left(\int_{\mathbf{Q}}\left|{(4\pi t)^{-n/2}}\exp\left(-\frac{|x-y|^{2}}{4t}\right)\right|^{p/(p-1)}d\mu_{n}(y)\right)^{p-1}
\nonumber\\&\left(\int_{{\mathbf{Q}}}|\phi(y)|^{p/p-1}d\mu_{n}(y)
\right)^{p-1}\nonumber\\&
\equiv\int_{{\mathbf{Q}}}\mathbb{E}\left\llbracket|\otimes\mathscr{J}(y)|^{p}\right\rrbracket d^{n}y\left(\int_{{\mathbf{Q}}}\left|{(4\pi t)^{-n/2}}\exp\left(-\frac{|x-y|^{2}}{4t}\right)\right|d^{n}y\right)^{p}
\left(\int_{{\mathbf{Q}}}|\phi(y)|d\mu_{n}(y)
\right)^{p}\nonumber\\&
\equiv \frac{1}{2}[\zeta^{p/2}+(-1)^{p}\zeta^{p/2}]v(\mathbf{Q})
\left(\int_{{\mathbf{Q}}}{(4\pi t)^{-n/2}}\exp\left(-\frac{|x-y|^{2}}{4t}\right)d\mu_{n}(y)\right)^{p}
\left(\int_{{\mathbf{Q}}}|\phi(y)|d\mu_{n}(y)
\right)^{p}
\end{align}
\end{proof}
\begin{cor}
If $\phi(x)=C$ then
\begin{align}
&\mathbb{E}\big\llbracket\widehat{u(x,t)}\big\rrbracket\le \frac{1}{2}|C|^{p}[\zeta^{p/2}+(-1)^{p}\zeta^{p/2}]|v({\mathbf{Q}})|^{p+1}
\left(\int_{\mathbf{Q}}\left|{(4\pi t)^{-n/2}}\exp\left(-\frac{|x-y|^{2}}{4t}\right)\right|d\mu_{n}(y)\right)^{p}
\end{align}
\end{cor}
\begin{cor}
The moments and volatility decay to zero as $t\rightarrow\infty$ so that
\begin{align}
&\lim_{t\uparrow\infty}\mathbb{E}\big\llbracket\widehat{|u(x,t)|^{p}}\big\rrbracket\le \frac{1}{2}[\zeta^{p/2}+(-1)^{p}\zeta^{p/2}]v(\mathbf{Q})\nonumber\\&
\otimes\lim_{t\uparrow\infty}\left(\int_{{\mathbf{Q}}}{(4\pi t)^{-n/2}}\exp\left(-\frac{|x-y|^{2}}{4t}\right)d\mu_{n}(y)\right)^{p}
\left(\int_{\bm{\mathbf{Q}}}|\phi(y)|d\mu_{n}(y)
\right)^{p}=0\nonumber\\&\lim_{t\uparrow\infty}\mathbb{E}\big\llbracket
\widehat{|u(x,t)|^{2}}\big\rrbracket\le \zeta v(\mathbf{Q})\lim_{t\uparrow\infty}\left(\int_{\mathbf{Q}}{(4\pi t)^{-n/2}}\exp\left(-\frac{|x-y|^{2}}{4t}\right)d\mu_{n}(y)\right)^{2}
\left(\int_{\mathbf{Q}}|\phi(y)|d\mu_{n}(y)
\right)^{2}=0
\end{align}
which is guaranteed since $\lim_{t\uparrow\infty}
\left(\int_{\mathbf{Q}}{(4\pi t)^{-n/2}}\exp\left(-\frac{|x-y|^{2}}{4t}\right)d\mu_{n}(y)\right)^{2}=0$ for all $(x,y)\in{\mathbf{Q}}$.
\end{cor}
\subsection{Moments estimates in a ball $\mathbf{B}_{R}(0)\subset\mathbf{R}^{3}$}
The moments and volatility estimates can be evaluated specifically for an Euclidean ball${\mathbf{Q}}={\mathbf{B}}_{R}(0)\subset{\mathbf{R}}^{3}$ of radius $R$.
\begin{thm}
Let ${\mathbf{B}}_{R}(0)\subset{\mathbf{R}}^{3}$ be an Euclidean ball of radius $R$. Then for the following stochastic Cauchy initial value problem
\begin{align}
&\left(\frac{\partial}{\partial t}-\Delta\right)\widehat{u(x,t)}=0,~~(x\in{\mathbf{B}}_{R}(0),t>0)\\&
\widehat{\psi(x,0)}=\phi(x)+\mathscr{J}(x)=C+\mathscr{J}(x),~~(x\in{\mathbf{B}}_{R}(0),t=0)\\&
\widehat{\psi(x,R)}=0,~~(x\in{\mathbf{B}}_{R}(0),t>0)
\end{align}
where $\mathscr{J}(x)$ is the usual GRSF. If $x=(0,0,a)$ with $\|x\|=a$ and $a\le R$ then the p-moment are estimated as
\begin{align}
\mathbb{E}\big\llbracket|\widehat{u(x,t)}|^{p}\big\rrbracket&\le \frac{2^{p+1}}{3}\pi R^{3}(|C|^{2}+\frac{1}{2}[\zeta^{p/2}+(-1)^{p}\zeta^{p/2}]\bigg\lbrace
\frac{1}{2}\bigg(erf\left(\frac{R+a}{2\sqrt{t}}\right)+erf\left(\frac{R+a}{2\sqrt{t}}\right)\bigg)\nonumber\\&
-\sqrt{\frac{t}{\pi}}\frac{1}{a}\big(\exp\left(\frac{at}{t}\right)-1\big)\exp\left(-\frac{|R+a|^{2}}{4t}\right)
\bigg\rbrace^{p}
\end{align}
\begin{align}
\mathbb{E}\big\llbracket|\widehat{\psi(a,t)}|^{p}\big\rrbracket&\le \frac{2}{3}\pi R^{3}\sum_{\beta=0}^{p}\binom{p}{\beta}|C|^{p-\beta}[\zeta^{p/2}+(-1)^{p}\zeta^{p/2}]\bigg\lbrace
\frac{1}{2}\bigg(erf\left(\frac{R+a}{2\sqrt{t}}\right)+erf\left(\frac{R+a}{2\sqrt{t}}\right)\bigg)\nonumber\\&
-\sqrt{\frac{t}{\pi}}\frac{1}{a}\big(\exp\left(\frac{at}{t}\right)-1\big)
\exp\left(-\frac{|R+a|^{2}}{4t}\right)
\bigg\rbrace^{2p-\beta}
\end{align}
and $\lim_{t\uparrow\infty}\mathbb{E}\big\llbracket|\widehat{u(x,t)}|^{p}\big\rrbracket=0$
\end{thm}
\begin{proof}
The estimate (6.37) is
\begin{align}
\mathbb{E}\big\llbracket|\widehat{\psi(a,t)}|^{p}\big\rrbracket& \le 2^{p-1}|C|^{p}v({\mathbf{B}}_{R}(0))+
\frac{1}{2}[\zeta^{p/2}+(-1)^{p}\zeta^{p/2}]v(\mathbf{B}_{R}(0))\nonumber\\&
\times\left(\int_{{\mathbf{B}}_{R}(0)}{(4 \pi t)^{-n/2}}\exp\left(-\frac{|x-y|^{2}}{4t}\right)d\mu_{n}(y)\right)^{p}
\end{align}
and the estimate(6.38)is
\begin{align}
\mathbb{E}\big\llbracket|\widehat{u(x,t)}|^{p}\big\rrbracket&\le \frac{1}{2}\sum_{\i=1}^{p}\binom{p}{\beta}C^{p-\beta}[\zeta^{\frac{p}{2}}+(-1)^{p}
\zeta^{\frac{p}{2}}]|v({\mathbf{B}}_{R}(0))|^{p}\nonumber\\&\times\left(\int_{\bm{\mathbf{B}}_{R}(0)}
{(4 \pi t)^{-n/2}}\exp\left(-\frac{|x-y|^{2}}{4t}\right)d\mu_{n}(y)\right)^{2p-\beta}
\end{align}
Choosing a point $x=(0,0,a)$ on the z-axis then $\|x\|=a\le L$ and $|x-y|
=\sqrt{(a^{2}-2ar\cos\theta+r^{2})}$. The volume integral is $\int_{\bm{\mathbf{B}}_{R}(0)}d\mu_{n}(y)\equiv 2\pi\int_{0}^{R}\int_{-\pi}^{\pi}r^{2}dr\sin\theta d\theta d\phi$. It is only necessary to evaluate the Gaussian integral of the heat kernel or fundamental solution, over the ball
\begin{equation}
\int_{{\mathbf{B}}_{R}(0)}{(4 \pi t)^{-n/2}}\exp\left(-\frac{|x-y|^{2}}{4t}\right)d\mu_{n}(y)\nonumber
\end{equation}
Then
\begin{align}
&\left(\int_{{\mathbf{B}}_{R}(0)}{(4 \pi t)^{-3/2}}\exp\left(-\frac{|x-y|^{2}}{4t}\right)d^{3}y\right)\nonumber\\&=2\pi{(4\pi t)^{-3/2}}\int_{0}^{R}\int_{-\pi}^{\pi}\exp\left(-(\frac{(a^{2}-2ar\cos\theta+r^{2}}{4t}\right))\sin d\theta  r^{2}dr\nonumber\\&
=2\pi{(4\pi t)^{-3/2}}\int_{0}^{R}\left|\int_{-1}^{1}\exp\left(-((a^{2}-2ar\zeta+r^{2})/{4t}\right))d\zeta\right| r^{2}dr\nonumber\\&
=2\pi{(4\pi t)^{-3/2}}\int_{0}^{R} \left|\frac{2t}{ar}e^{-r^{2}/4t+(ar/2t)-(a^{2}/4t)}-\frac{2t}{ar}e^{-(r^{2}/4t)-(ar/2t)-(a^{2}/4t)}
\right| r^{2}dr\nonumber\\&
=\frac{1}{2}\bigg(erf\left(\frac{R+a}{2\sqrt{t}}\right)+erf\left(\frac{R-a}{2\sqrt{t}}\right)
\bigg)-\sqrt{\frac{t}{\pi}}\frac{1}{a}\big(e^{\frac{at}{t}}-1\big)e^{-\frac{|R+a|^{2}}{4t}}
\end{align}
Where $\zeta=\cos\theta$ and the error function has the properties $erf(-\infty)=-1,erf(0)=0, erf(\infty)=1$. Note that the integral over all space is unity as required, so that
\begin{align}
&\left(\int_{\mathbf{R}}{(4 \pi t)^{-3/2}}e^{-\frac{|x-y|^{2}}{4t}}d^{3}y\right)=2\pi{(4\pi t)^{-3/2}}\int_{0}^{\infty}\int_{-\pi}^{\pi}e^{-((a^{2}-2ar\cos\theta+r^{2})/{4t}})\sin d\theta  r^{2}dr\nonumber\\&
=2\pi{(4\pi t)^{-3/2}}\int_{0}^{\infty}\left|\int_{-1}^{1}e^{-((a^{2}-2ar\zeta+r^{2})/{4t}})d\zeta\right| r^{2}dr\nonumber\\&
=2\pi{(4\pi t)^{-3/2}}\int_{0}^{R} \left|\frac{2t}{ar}e^{-r^{2}/4t+(ar/2t)-(a^{2}/4t)}-\frac{2t}{ar}e^{-(r^{2}/4t)-(ar/2t)-(a^{2}/4t)}
\right| r^{2}dr\nonumber\\&
=\lim_{R\uparrow\infty}\frac{1}{2}\bigg(erf\left(\frac{R+a}{2\sqrt{t}}\right)+erf\left(\frac{R-a}{2\sqrt{t}}\right)\bigg)-\sqrt{\frac{t}{\pi}}\frac{1}{a}\big(e^{\frac{at}{t}}-1\big)e^{-\frac{|R+a|^{2}}{4t}}=0
\end{align}
Since $v(\mathbf{Q})=\int_{{\mathbf{B}}_{R}(0)}d^{3}y=2\pi \int_{-\pi}^{\pi}\int_{0}^{R} r^{2}dr\sin\theta d\theta d\phi=\tfrac{4}{3}\pi R^{3}$, the estimates (6.37 )and(6.38) readily follow and the volatility is for $p=2$. The moments then decay rapidly to zero as $t\rightarrow \infty$.
\end{proof}
\subsection{Moments estimates for the inhomogeneous problem}
The estimates for the p-moments and volatility can be readily extended to the inhomogeneous problem.
\begin{thm}
Let $f\in C^{2}(\mathbf{Q}\times [0,\infty)$ and $g\in C^{2}(\mathbf{Q})$ and consider the stochastic IVP for the inhomogenous heat equation
\begin{equation}
\frac{\partial}{\partial t}\widehat{u(x,t)}=\Delta\widehat{u(x,t)}+f(x,t),~~(x\in\mathbf{Q},t>0)
\end{equation}
\begin{enumerate}
\item If $\widehat{\psi(x,0)}=\phi(x)+\mathscr{J}(x)$ is the initial random Cauchy data then
the moments are estimated as
\begin{align}
&\mathbb{E}\big\llbracket|\widehat{u(x,t)}|^{p}\big\rrbracket\le
3^{p-1}\left|\int_{0}^{t}\int_{\mathbf{Q}}\frac{e^{-\frac{|x-y|^{2}}{4t}}}{ (4\pi|t-s|)^{n/2}}f(y,s)d\mu_{n}(y)ds
\right|^{p}\nonumber\\&+3^{p-1}\left(\int_{\mathbf{Q}}|\phi(y)|^{p}d^{n}y+\frac{1}{2}[\zeta^{p/2}+(-1)^{p}\zeta^{p/2}]
v(\mathbf{Q})\right)\left(\int_{\mathbf{Q}}\frac{1}{(4 \pi t)^{-n/2}}e^{-\frac{|x-y|^{2}}{4t}}d\mu_{n}(y)\right)^{p}
\end{align}
and the volatility is
\begin{align}
&\mathbb{E}\big\llbracket\widehat{u(x,t)}|^{2}\big\rrbracket\le
3\left|\int_{0}^{t}\int_{\mathbf{Q}}\frac{e^{-\frac{|x-y|^{2}}{4t}}}{ (4\pi|t-s|)^{n/2}}f(y,s)d\mu_{n}(y)ds
\right|^{2}\nonumber\\&+3\left(\int_{\mathbf{Q}}|\phi(y)|^{2}d^{n}y+\beta
v(\mathbf{Q})\right)\left(\int_{\mathbf{Q}}\frac{1}{(4 \pi t)^{-n/2}}e^{-\frac{|x-y|^{2}}{4t}}d\mu_{n}(y)\right)^{2}
\end{align}
\item If $\widehat{u(x,0)}=\phi(x)\mathscr{J}(x)$ is the initial random Cauchy data then
the moments are estimated as
\begin{align}
&\mathbb{E}\big\llbracket|\widehat{u(x,t)}|^{p}\big\rrbracket\le
3^{p-1}\left|\int_{0}^{t}\int_{\mathbf{Q}}\frac{e^{-\frac{|x-y|^{2}}{4t}}}{ (4\pi|t-s|)^{n/2}}f(y,s)d\mu_{n}(y)ds
\right|^{p}\nonumber\\&+3^{p-1}\left(\int_{\mathbf{Q}}|\phi(y)|^{p}d^{n}\right)
[\zeta^{p/2}+(-1)^{p}\zeta^{p/2}]
v(\mathbf{Q})\left(\int_{\mathbf{Q}}\frac{1}{(4 \pi t)^{-n/2}}e^{-\frac{|x-y|^{2}}{4t}}d\mu_{n}(y)\right)^{p}
\end{align}
and the volatility is
\begin{align}
&\mathbb{E}\big\llbracket|\widehat{u(x,t)}|^{2}\big\rrbracket\le
3^{p-1}\left|\int_{0}^{t}\int_{\mathbf{Q}}\frac{e^{-\frac{|x-y|^{2}}{4t}}}{ (4\pi|t-s|)^{n/2}}f(y,s)d\mu_{n}(y)ds
\right|^{p}\nonumber\\&+3\left(\int_{\mathbf{Q}}|\phi(y)|^{p}d^{n}\right)\zeta
v(\mathbf{Q})\left(\int_{\mathbf{Q}}\frac{1}{(4 \pi t)^{-n/2}}e^{-\frac{|x-y|^{2}}{4t}}d\mu_{n}(y)\right)^{2}
\end{align}
\end{enumerate}
\end{thm}
\begin{proof}
The solution is the stochastic convolution integral so that for (1)
\begin{align}
&\mathbb{E}\bigg\llbracket|\widehat{u(x,t)}|^{p}\big\rrbracket\nonumber\\&=
\mathbb{E}\bigg\llbracket\bigg|\int_{0}^{t}\int_{\mathbf{Q}}\frac{e^{-\frac{|x-y|^{2}}{4t}}}{ (4\pi|t-s|)^{n/2}}f(y,s)d\mu_{n}(y)ds+\int_{\mathbf{Q}}\frac{1}{(4\pi t)^{n/2}}e^{-\frac{|x-y|^{2}}{4t}}\phi(y)d\mu_{n}(y)\nonumber\\&+\int_{\mathbf{Q}}\frac{1}{(4\pi t)^{n/2}}e^{-\frac{|x-y|^{2}}{4t}}\mathscr{J}(y)d\mu_{n}(y)\bigg|^{p}\bigg
\rrbracket\nonumber\\&\le
3^{p-1}\left|\int_{0}^{t}\int_{\mathbf{Q}}\frac{e^{-\frac{|x-y|^{2}}{4t}}}{ (4\pi|t-s|)^{n/2}}f(y,s)d\mu_{n}(y)ds\right|^{p}\nonumber\\&
+\underbrace{3^{p-1}\left|\int_{\mathbf{Q}}\frac{1}{(4\pi t)^{n/2}}e^{-\frac{|x-y|^{2}}{4t}}\phi(y)d\mu_{n}(y)\right|^{p}
+3^{p-1}\mathbb{E}\left\llbracket\left|\int_{\mathbf{Q}}\frac{1}{(4\pi t)^{n/2}}e^{-\frac{|x-y|^{2}}{4t}}\mathscr{J}(y)d\mu_{n}(y)\right|^{p}\right\rrbracket}
\end{align}
The underbraced term is then estimated as before and the volatility estimate follows for $p=2$. For (2)
\begin{align}
&\mathbb{E}\big\llbracket|\widehat{u(x,t)}|^{p}\big\rrbracket\nonumber\\&
\mathbb{E}\left\llbracket\left|\int_{0}^{t}\int_{\mathbf{Q}}
\frac{e^{-\frac{|x-y|^{2}}{4t}}}{(4\pi|t-s|)^{n/2}}f(y,s)d\mu_{n}(y)ds+\int_{\mathbf{Q}}\frac{1}{(4\pi t)^{n/2}}e^{-\frac{|x-y|^{2}}{4t}}\phi(y)\mathscr{J}(y)d\mu_{n}(y)\right|^{p}\right
\rrbracket\nonumber\\&\le
2^{p-1}\left|\int_{0}^{t}\int_{\mathbf{Q}}\frac{e^{-\frac{|x-y|^{2}}{4t}}}{ (4\pi|t-s|)^{n/2}}f(y,s)d\mu_{n}(y)ds\right|^{p}+\underbrace
{2^{p-1}\mathbb{E}\left\llbracket\left|\int_{\mathbf{Q}}\frac{1}{(4\pi t)^{n/2}}e^{-\frac{|x-y|^{2}}{4t}}\phi(y)\otimes\mathscr{J}(y)
d\mu_{n}(y)\right|^{p}\right\rrbracket}
\end{align}
and the underbraced term is estimated as before. The volatility estimates follows for $p=2$.
\end{proof}
\begin{cor}
If
\begin{equation}
\lim_{t\uparrow\infty}\left|\int_{0}^{t}\int_{\mathbf{Q}}\frac{e^{-\frac{|x-y|^{2}}{4t}}}{ (4\pi|t-s|)^{n/2}}f(y,s)d\mu_{n}(y)ds\right|=0
\end{equation}
it follows that $\lim_{t\uparrow \infty}\mathbb{E}\big\llbracket\widehat{u(x,t)}|^{p}\big\rrbracket=0$ for the estimates.
\end{cor}
\begin{thm}($\underline{Moments~estimate~for~a~ring}$)
Let ${\mathbf{S}}^{1}\subset{\mathbf{R}}^{1}$ be a circle or ring and consider the
CIVP for the randomly perturbed heat equation with GRF initial data and periodic boundary conditions
\begin{align}
&{\square}\widehat{\psi(\theta,t)}=\partial_{\theta}\widehat{\psi(\theta,t)}-\partial_{\theta\theta}\widehat{\psi(\theta,t)}=0,
~\theta\in {\mathbf{S}}^{1},t>0\nonumber\\&
\widehat{\psi(\theta,0)}=q(\theta)+\mathscr{J}(\theta),~\theta\in \bm{\mathbf{S}}^{1}\nonumber,t=0
\nonumber\\&
\psi(\theta,t)=\psi(\theta+2\pi,t)\nonumber
\end{align}
where $\mathbb{E}\llbracket \mathscr{J}(\theta)\rrbracket=0$ and
$\mathbb{E}\llbracket\mathscr{J}(|\theta)|^{p}\rrbracket=
\frac{1}{2}[\eta^{p/2}+(-1)^{p}\eta^{p/2}]$. The solution is the (random) Fourier series
\begin{align}
\widehat{\psi(\theta,t)}=A_{o}+\sum_{k=1}^{\infty}
e^{-k^{2}t}(\widehat{A_{k}}\cos(k\theta)+\widehat{B_{k}}\sin(k\theta)]
\end{align}
where the stochastic Fourier coefficients are
\begin{align}
&\widehat{A_{k}}=A_{k}+\mathscr{A}_{k}
=\pi^{-1}\oint q(\theta)\cos(k\theta)d\theta+\pi^{-1}\oint \mathscr{J}(\theta)\cos(k\theta)d\theta
\\&
\widehat{B_{k}}=B_{k}+\mathscr{B}_{k}=\pi^{-1}\oint Q(\theta)\sin(k\theta)d\theta+\pi^{-1}\oint\mathscr{J}(\theta)\sin(k\theta)d\theta
\end{align}
and where $\oint\equiv \int_{-\pi}^{\pi}$. Their expectations are then
\begin{align}
&\mathbb{E}\llbracket\widehat{A_{k}}\rrbracket=A_{k}+
\mathbb{E}\llbracket\mathscr{A}_{k}\rrbracket
=\pi^{-1}\oint q(\theta)\cos(k\theta)d\theta+\pi^{-1}\oint
\mathbb{E}\llbracket
\mathscr{J}(\theta)\rrbracket \cos(k\theta)d\theta=A_{k}
\\&
\mathbb{E}\llbracket\widehat{B_{k}}\rrbracket=B_{k}+
\mathbb{E}\llbracket
\mathscr{B}_{k}\rrbracket=\pi^{-1}\oint q(\theta)\sin(k\theta)d\theta
+\pi^{-1}\oint\mathbb{E}\llbracket\mathscr{J}(\theta)\rrbracket
\sin(k\theta)d\theta=B_{k}
\end{align}
The moments estimate is then
\begin{align}
\mathbb{E}\llbracket |\widehat{\psi(\theta,t)}|^{p}\rrbracket&\le
\frac{1}{2^{3p}}\left|\sum_{k=0}^{\infty}e^{-k^{2}t}A_{k}\cos(k\theta)\right|^{p}
+\frac{1}{2^{3p}}\left|\sum_{k=0}^{\infty}e^{-k^{2}t}B_{k}\sin(k\theta)\right|^{p}
\nonumber\\&+\frac{1}{2^{3p}}
\left(\sum_{k=0}^{\inf}\left|e^{-k^{2}t}\cos(k\theta)\right|^{q}\right)^{p/q}
\sum_{k=0}^{\infty}\left(\pi^{-1}\oint_{\mathbf{S}^{1}}|\cos(k\theta)|^{q}d\theta\right)^{p/q}
\left(\pi[\eta^{p/2}+(-1)^{p}\eta^{p/2}]\right)\nonumber\\&
+\frac{1}{2^{3p}}\left(\sum_{k=0}^{\inf}\left|e^{-k^{2}t}\sin(k\theta)\right|^{q}\right)^{p/q}
\sum_{k=0}^{\infty}\left(\pi^{-1}\oint_{\mathbf{S}^{1}}|\sin(k\theta)|^{q}d\theta\right)^{p/q}
\left(\pi[\eta^{p/2}+(-1)^{p}\eta^{p/2}]\right)
\end{align}
and decay smoothly to zero such that
\begin{align}
\lim_{t\uparrow\infty}\mathbb{E}\llbracket |\widehat{\psi(\theta,t)}|^{p}\rrbracket=0
\end{align}
Hence, the heat equation smooths out the initial randomness in the initial data on the ring.
\end{thm}
\begin{proof}
Using the basic estimate $|\alpha+\beta|^{p}\le\frac{1}{2^{p}}|\alpha|^{p}+ \frac{1}{2^{p}}|\beta|^{p}$ and the Holder inequality for sums, the p-moments can be estimated as
\begin{align}
&\mathbb{E}\left\llbracket|\widehat{\psi(\theta,t)}|^{p}\right \rrbracket
=\mathbb{E}\left\llbracket \left|\sum_{k=1}^{\infty}
e^{-k^{2}t}(\widehat{A_{k}}\cos(k\theta)+\widehat{B_{k}}\sin(k\theta)]\right|^{p}\right\rrbracket\nonumber\\&\le
\mathbb{E}\left\llbracket \frac{1}{2^{p}}|A_{o}|^{p}+
\frac{1}{2^{p}}\left|\sum_{k=0}^{\infty}
e^{-k^{2}t}(\widehat{A_{k}}\cos(k\theta)+\widehat{B_{k}}\sin(k\theta)]\right|^{p}\right\rrbracket\nonumber\\&
=\mathbb{E}\left\llbracket
\frac{1}{2^{p}}\left|\sum_{k=0}^{\infty}
e^{-k^{2}t}(\widehat{A_{k}}\cos(k\theta)+\widehat{B_{k}}\sin(k\theta)]\right|^{p}\right\rrbracket
\nonumber\\&
= \mathbb{E}\left\llbracket
\frac{1}{2^{p}}\left|\sum_{k=0}^{\infty}e^{-k^{2}t}\widehat{A_{k}}\cos(k\theta)
+\frac{1}{2^{p}}\sum_{k=0}^{\infty}e^{-k^{2}t}\widehat{B_{k}}\sin(k\theta)]\right|^{p}\right\rrbracket\nonumber\\&
\le\mathbb{E}\left\llbracket
\frac{1}{2^{2p}}\left|\sum_{k=0}^{\infty}e^{-k^{2}t}\widehat{A_{k}}\cos(k\theta)\right|^{p}
+\frac{1}{2^{2p}}\left|\sum_{k=0}^{\infty}e^{-k^{2}t}\widehat{B_{k}}\sin(k\theta)]\right|^{p}\right\rrbracket=\nonumber\\&
\le\frac{1}{2^{p}}|\frac{1}{2^{2p}}\mathbb{E}\left\llbracket\left|\sum_{k=1}^{\infty}e^{-k^{2}t}\widehat{A_{k}}\cos(k\theta)\right|^{p}\right\rrbracket
+\frac{1}{2^{2p}}\mathbb{E}\left\llbracket\left|\sum_{k=0}^{\infty}e^{-k^{2}t}\widehat{B_{k}}\sin(k\theta)]\right|^{p}\right\rrbracket
=\nonumber\\&
=\frac{1}{2^{2p}}\mathbb{E}\left\llbracket\left|\sum_{k=0}^{\infty}e^{-k^{2}t}(A_{k}+\mathscr{A}_{k})
\cos(k\theta)\right|^{p}\right\rrbracket
+\frac{1}{2^{2p}}\mathbb{E}\left\llbracket\left|\sum_{k=0}^{\infty}e^{-k^{2}t}
(B_{k}+\mathscr{B}_{k})\sin(k\theta)]\right|^{p}\right\rrbracket\nonumber\\&
=\frac{1}{2^{2p}}\mathbb{E}\left\llbracket\left|\sum_{k=0}^{\infty}e^{-k^{2}t}A_{k}\cos(k\theta)+\sum_{k=1}^{\infty}e^{-k^{2}t}\mathscr{A}_{k}\cos(k\theta)\right|^{p}\right\rrbracket
\nonumber\\&
+\frac{1}{2^{2p}}\mathbb{E}\left\llbracket\left|\sum_{k=0}^{\infty}e^{-k^{2}t}B_{k}\sin(k\theta)+\sum_{k=1}^{\infty}e^{-k^{2}t}\mathscr{B}_{k}\sin(k\theta)\right|^{p}\right\rrbracket
\nonumber\\&
\le \frac{1}{2^{2p}}\left(\frac{1}{2^{p}}\left|\sum_{k=0}^{\infty}e^{-k^{2}t}A_{k}\cos(k\theta)\right|^{p}
+\frac{1}{2^{p}}\mathbb{E}\left\llbracket\underbrace{\left|\sum_{k=0}^{\infty}
e^{-k^{2}t}\mathscr{A}_{k}\cos(k\theta)\right|^{p}}_{apply~Holder~inequality}\right\rrbracket\right)\nonumber\\&
+\frac{1}{2^{2p}}\left(\frac{1}{2^{p}}\left|\sum_{k=0}^{\infty}e^{-k^{2}t}B_{k}\sin(k\theta)\right|^{p}
+\frac{1}{2^{p}}\mathbb{E}\left\llbracket\underbrace{\left|\sum_{k=0}^{\infty}
e^{-k^{2}t}\mathscr{B}_{k}\sin(k\theta)\right|^{p}}_{apply~Holder~inequality}\right\rrbracket\right)
\nonumber\\&
\le \frac{1}{2^{3p}}\left|\sum_{k=0}^{\infty}e^{-k^{2}t}A_{k}\cos(k\theta)\right|^{p}
+\frac{1}{2^{3p}}\left|\sum_{k=0}^{\infty}e^{-k^{2}t}B_{k}\sin(k\theta)\right|^{p}\nonumber\\&
+\frac{1}{2^{3p}}
\left(\sum_{k=0}^{\inf}\left|e^{-k^{2}t}\cos(k\theta)\right|^{q}\right)^{p/q}
\sum_{k=0}^{\infty}\mathbb{E}\llbracket|\mathscr{A}_{k}|^{p}\rrbracket
+\frac{1}{2^{3p}}
\left(\sum_{k=0}^{\inf}\left|e^{-k^{2}t}\sin(k\theta)\right|^{q}\right)^{p/q}
\sum_{k=0}^{\infty}\mathbb{E}\llbracket|\mathscr{B}_{k}|^{p}\rrbracket\nonumber\\&
\le \frac{1}{2^{3p}}\left|\sum_{k=0}^{\infty}e^{-k^{2}t}A_{k}\cos(k\theta)\right|^{p}
+\frac{1}{2^{3p}}\left|\sum_{k=0}^{\infty}e^{-k^{2}t}B_{k}\sin(k\theta)\right|^{p}
\nonumber\\&+\frac{1}{2^{3p}}
\left(\sum_{k=0}^{\inf}\left|e^{-k^{2}t}\cos(k\theta)\right|^{q}\right)^{p/q}
\sum_{k=0}^{\infty}\mathbb{E}\left\llbracket\underbrace{\left|\pi^{-1}\oint_{\mathfrak{S}^{1}}
\mathscr{J}(\theta)\cos(k\theta)d\theta\right|^{p}}_{apply~Holder~ineq~for~integrals}\right\rrbracket
\nonumber\\&
+\frac{1}{2^{3p}}\left(\sum_{k=0}^{\inf}\left|e^{-k^{2}t}\sin(k\theta)\right|^{q}\right)^{p/q}
\sum_{k=0}^{\infty}\mathbb{E}\left\llbracket\underbrace{\left|\pi^{-1}\oint_{\mathfrak{S}^{1}}
\mathscr{J}(\theta)\sin(k\theta)d\theta\right|^{p}}_{apply~Holder~ineq.~for~integrals}\right\rrbracket
\nonumber\\&
\le \frac{1}{2^{3p}}\left|\sum_{k=0}^{\infty}e^{-k^{2}t}A_{k}\cos(k\theta)\right|^{p}
+\frac{1}{2^{3p}}\left|\sum_{k=0}^{\infty}e^{-k^{2}t}B_{k}\sin(k\theta)\right|^{p}
\nonumber\\&+\frac{1}{2^{3p}}
\left(\sum_{k=0}^{\inf}\left|e^{-k^{2}t}\cos(k\theta)\right|^{q}\right)^{p/q}
\sum_{k=0}^{\infty}\left(\pi^{-1}\oint_{\mathbf{S}^{1}}|\cos(k\theta)|^{q}d\theta\right)^{p/q}
\left(\oint_{\mathfrak{S}^{1}}\mathbb{E}\llbracket|\mathscr{J}(\theta)|^{p}\rrbracket d\theta\right)
\nonumber\\&
+\frac{1}{2^{3p}}\left(\sum_{k=0}^{\inf}\left|e^{-k^{2}t}\sin(k\theta)\right|^{q}\right)^{p/q}
\sum_{k=0}^{\infty}\left(\pi^{-1}\oint_{\mathbf{S}^{1}}|\sin(k\theta)|^{q}d\theta\right)^{p/q}
\left(\oint_{\mathbf{S}^{1}}\mathbb{E}\llbracket|\mathscr{J}(\theta)|^{p}\rrbracket d\theta\right)
\nonumber\\&
= \frac{1}{2^{3p}}\left|\sum_{k=0}^{\infty}e^{-k^{2}t}A_{k}\cos(k\theta)\right|^{p}
+\frac{1}{2^{3p}}\left|\sum_{k=0}^{\infty}e^{-k^{2}t}B_{k}\sin(k\theta)\right|^{p}
\nonumber\\&+\frac{1}{2^{3p}}
\left(\sum_{k=0}^{\inf}\left|e^{-k^{2}t}\cos(k\theta)\right|^{q}\right)^{p/q}
\sum_{k=0}^{\infty}\left(\pi^{-1}\oint_{\mathbf{S}^{1}}|\cos(k\theta)|^{q}d\theta\right)^{p/q}
\left(\frac{1}{2}\oint_{\mathbf{S}^{1}}[\eta^{p/2}+(-1)^{p}\eta^{p/2}]
d\theta\right)
\nonumber\\&
+\frac{1}{2^{3p}}\left(\sum_{k=0}^{\inf}\left|e^{-k^{2}t}\sin(k\theta)\right|^{q}\right)^{p/q}
\sum_{k=0}^{\infty}\left(\pi^{-1}\oint_{\mathbf{S}^{1}}|\sin(k\theta)|^{q}d\theta\right)^{p/q}
\left(\frac{1}{2}\oint_{\mathfrak{S}^{1}}[\eta^{p/2}+(-1)^{p}\eta^{p/2}]
d\theta\right)\nonumber\\&
= \frac{1}{2^{3p}}\left|\sum_{k=0}^{\infty}e^{-k^{2}t}A_{k}\cos(k\theta)\right|^{p}
+\frac{1}{2^{3p}}\left|\sum_{k=0}^{\infty}e^{-k^{2}t}B_{k}\sin(k\theta)\right|^{p}
\nonumber\\&+\frac{1}{2^{3p}}
\left(\sum_{k=0}^{\inf}\left|e^{-k^{2}t}\cos(k\theta)\right|^{q}\right)^{p/q}
\sum_{k=0}^{\infty}\left(\pi^{-1}\oint_{\mathbf{S}^{1}}|\cos(k\theta)|^{q}d\theta\right)^{p/q}
\left(\pi[\eta^{p/2}+(-1)^{p}\eta^{p/2}]\right)\nonumber\\&
+\frac{1}{2^{3p}}\left(\sum_{k=0}^{\inf}\left|e^{-k^{2}t}\sin(k\theta)\right|^{q}\right)^{p/q}
\sum_{k=0}^{\infty}\left(\pi^{-1}\oint_{\mathbf{S}^{1}}|\sin(k\theta)|^{q}d\theta\right)^{p/q}
\left(\pi[\eta^{p/2}+(-1)^{p}\eta^{p/2}]\right)
\end{align}
This is smoothed out and decays as $t\rightarrow\infty$. The sums
\begin{align}
&\sum_{k=0}^{\infty}\left(\pi^{-1}\oint_{\mathbf{S}^{1}}|\cos(k\theta)|^{q}d\theta\right)^{p/q}\\&
\sum_{k=0}^{\infty}\left(\pi^{-1}\oint_{\mathbf{S}^{1}}|\sin(k\theta)|^{q}d\theta\right)^{p/q}
\end{align}
are convergent and finite for all $(p,q)$ such that $\tfrac{1}{p}+\tfrac{1}{q}=1$.
\end{proof}
\subsection{Alternative estimate for the p-moments and volatility}
An alternative estimate can be made as follows
\begin{thm}
Let $\bm{\mathbf{Q}}\subset{\mathbf{R}}^{n}$ and $\widehat{u(x,t)}$ satisfies the SCIVP. If $\int_{{\mathbf{Q}}}|\phi(x)|^{p}d\mu_{n}(x)\le\lambda$ and $\mathbb{E}\llbracket|\mathscr{J}(x)|^{p}\rrbracket
=\tfrac{1}{2}[\zeta^{p/2}+(-1)^{p}\zeta^{p/2}]$ then the a p-moments estimate is
\begin{align}
\mathbb{E}\llbracket|\widehat{u(x,t)}|^{p}\rrbracket \le
2^{p-2}\left[\lambda^{p}+\frac{1}{2}v({\mathbf{Q}}[\zeta^{p/2}+(-1)^{p}\zeta^{p/2}]\right]\left( \frac{1}{(4\pi t)^{n}}\int_{{\mathbf{Q}}}e^{-\frac{|x-y|^{2}}{2t}}d\mu_{n}(y)
\right)^{p}
\end{align}
Then:
\begin{enumerate}
\item For $\bm{\mathbf{Q}}=\bm{\mathfrak{L}}[0,\ell]\subset\bm{\mathrm{R}}^{+}$
\begin{align}
\mathbb{E}\llbracket|\widehat{u(x,t)}|^{p}\rrbracket&\le
2^{p-2}\left[\lambda^{p}+\frac{1}{2}\ell [\zeta^{p/2}+(-1)^{p}\zeta^{p/2}]\right]\left(
\frac{1}{(4\pi t)}\int_{\bm{\mathbf{Q}}}e^{-\frac{|x-y|^{2}}{2t}}dy\right)^{p}\nonumber\\&
=2^{p-2}\left[\lambda^{p}+\frac{1}{2}\ell [\zeta^{p/2}+(-1)^{p}\zeta^{p/2}]\right]
\left(\frac{(erf\left(\frac{x}{2t^{1/2}}\right)-erf\left(\frac{x-\ell}{2t^{1/2}}\right)}{4(\pi t)^{1/2}}\right)^{p}
\end{align}
\item For $\bm{\mathrm{D}}=\bm{\mathrm{B}}_{R}(0)\subset\bm{\mathrm{R}}^{3}$
\begin{align}
\mathbb{E}\llbracket|\widehat{u(x,t)}|^{p}\rrbracket& \le
2^{p-2}\left[\lambda^{p}+\frac{1}{2}\ell [\zeta^{p/2}+(-1)^{p}\zeta^{p/2}]\right]\left(
\frac{1}{(4\pi t)^{3}}\int_{{\mathbf{B} }_{R}(0)}e^{-\frac{|x-y|^{2}}{2t}}d^{3}y\right)^{p}\nonumber\\&
=2^{p-2}\left[\lambda^{p}+\frac{2}{3}\pi R^{3} [\zeta^{p/2}+(-1)^{p}\zeta^{p/2}]\right]
\left(\frac{e^{-\frac{R^{2}}{2t}}(\pi^{1/2}t^{3/2}e^{\frac{R^{2}}{2t}}erf\left(\frac{R}{2^{1/2}t}\right)-2^{1/2}tR}{2^{1/2}t}\right)^{p}
\end{align}
\item The moments decay as $t\rightarrow\infty$ so that
\begin{equation}
\lim_{t\rightarrow\infty}\mathbb{E}\llbracket|\widehat{u(x,t)}|^{p}\rrbracket
\end{equation}
The volatility is given for $p=2$.
\end{enumerate}
\end{thm}
\begin{proof}
The p-moments are
\begin{align}
&\mathbb{E}\llbracket|\widehat{u(x,t)}|^{p}
\rrbracket=\mathbb{E}\left\llbracket\left|
\int_{{\mathbf{Q}}}h(x-y,t)\phi(y)d\mu_{n}(y)+
\int_{{\mathbf{Q}}}h(x-y,t)\otimes \mathscr{J}(y)d\mu_{n}(y)
\right|^{p}\right\rrbracket\nonumber\\&
\le 2^{p-1}\left|
\int_{{\mathbf{Q}}}h(x-y,t)\phi(y)d^{n}y\right|^{p}
+2^{p-1}\mathbb{E}\left\llbracket
\left|\int_{{\mathbf{Q}}}h(x-y,t)\otimes\mathscr{J}(y)d\mu_{n}(y)\right|^{p}\right\rrbracket
\end{align}
Since $|h(x-y,t)\phi(y)|\le\tfrac{1}{2}|h(x-y,t)|^{2}+
\tfrac{1}{2}|\phi(y)|^{2}$ then
\begin{align}
&|h(x-y,t)\phi(y)|^{p}\le (\tfrac{1}{2}|h(x-y,t)|^{2}+\tfrac{1}{2}|\phi(y)|^{2})^{p}\le
2^{p-1}(|\tfrac{1}{2}|h(x-y,t)|^{2})^{p}+2^{p-1}(|\tfrac{1}{2}|\phi(y)|^{2})^{p}\nonumber\\&
=\tfrac{1}{2}|h(x-y,t)|^{2p}+\tfrac{1}{2}|\phi(y)|^{2p}
\end{align}
and
\begin{align}
\int|h(x-y,t)\phi(y)|d\mu_{n}(y)\le\tfrac{1}{2}\int|h(x-y,t)|^{2}d^{n}y+\tfrac{1}{2}\int|\phi(y)|^{2}d^{n}y
\end{align}
so that
\begin{align}
&\left|\int|h(x-y,t)\phi(y)|d\mu_{n}(y)\right|^{p}\le\left|\tfrac{1}{2}
\int|h(x-y,t)|^{2}d^{n}y
+\tfrac{1}{2}\int|\phi(y)|^{2}d\mu_{n}(y)\right|^{p}\nonumber\\&
\le 2^{p-1}\left|\tfrac{1}{2}\int|h(x-y,t)|^{2}dx\right|^{p}
+2^{p-1}\left|\tfrac{1}{2}\int|\phi(y)|^{2}dx\right|^{p}
=\tfrac{1}{2}\left|\int|h(x-y,t)|^{2}d\mu_{n}(y)\right|^{p}\nonumber\\&
+\tfrac{1}{2}\left|\int|\phi(y)|^{2}d\mu_{n}(y)\right|^{p}
\end{align}
The p-moments are then estimated as
\begin{align}
&\mathbb{E}\llbracket|\widehat{u(x,t)}|^{p}
\rrbracket=\mathbb{E}\left\llbracket\left|
\int_{{\mathbf{Q}}}h(x-y,t)\phi(y)d\mu_{n}(y)+
\int_{{\mathbf{Q}}}h(x-y,t)\otimes \mathscr{J}(y)d\mu_{n}(y)
\right|^{p}\right\rrbracket\nonumber\\&
\le 2^{p-1}\left|
\int_{{\mathbf{Q}}}\phi(y)d\mu_{n}(y)\right|^{p}+2^{p-1}\mathbb{E}\left\llbracket
\left|\int_{{\mathbf{Q}}}h(x-y,t)\otimes\mathscr
{J}(y)d\mu_{n}(y)\right|^{p}\right\rrbracket
\nonumber\\&
=2^{p-1}\tfrac{1}{2}\left|\int|h(x-y,t)|^{2}d\mu_{n}(y)\right|^{p}
+2^{p-1}\tfrac{1}{2}\left|\int|\phi(y)|^{2}d\mu_{n}(y)\right|^{p}\nonumber\\&+
2^{p-1}\tfrac{1}{2}\left|\int|h(x-y,t)|^{2}d\mu_{n}(y)\right|^{p}
+2^{p-1}\tfrac{1}{2}\mathbb{E}\left\llbracket\left|\int|\mathscr{G}(y)|^{2}
d\mu_{n}(y)y\right|^{p}\right\rrbracket\nonumber\\&
=2^{p-2}\left|\int|h(x-y,t)|^{2}d\mu_{n}(y)\right|^{p}
+2^{p-2}\left|\int|\phi(y)|^{2}d\mu_{n}(y)\right|^{p}\nonumber\\&+
2^{p-2}\left|\int|h(x-y,t)|^{2}d\mu_{n}(y)\right|^{p}
+2^{p-2}\mathbb{E}\left\llbracket\left|\int|\mathscr{J}(y)|^{2}
d\mu_{n}(y)\right|^{p}\right\rrbracket
\end{align}
\end{proof}
\begin{lem}
Given the general solution $\widehat{u(x,t)}$ of the stochastic initial value problem
then the stochastically averaged Dirichlet energy integral over a domain ${\mathbf{Q}}$
is
\begin{align}
&\mathbb{E}\llbracket\widehat{\bm{\mathrm{I\!E}}_{I}[t]}\rrbracket
=\mathbb{E}\left\llbracket\int_{{\mathbf{Q}}}|\widehat{u(x,t)}|^{2}d\mu_{n}(x)
\right\rrbracket\equiv\int_{{\mathbf{Q}}}\mathbb{E}\left\llbracket
|\widehat{u(x,t)}|^{2}\right\rrbracket d\mu_{n}(x)
\nonumber\\&
=\bm{\mathrm{I\!E}}_{I}[t]+\zeta v({\mathbf{Q}})
\int_{\mathbf{Q}}\left\lbrace\int_{{\mathbf{Q}}}(4\pi t)^{-n}
e^{-\frac{|x-y|^{2}}{2t}}d\mu_{n}(y)\right\rbrace d\mu_{n}(x)
\end{align}
with the derivative and asymptotic property
\begin{align}
&\frac{d}{dt}\mathbb{E}\llbracket\widehat{\bm{\mathrm{I\!E}}_{I}[t]}\rrbracket\le 0\\&
\lim_{t\uparrow\infty}\mathbb{E}\llbracket\widehat{\bm{\mathrm{I\!E}}_{I}[t]}\rrbracket
\lim_{t\uparrow\infty}\bm{\mathrm{I\!E}}_{I}[t]=0
\end{align}
For $\mathbf{Q}=[0,\ell]\subset\bm{\mathrm{R}}^{+}$
\begin{align}
&\mathbb{E}\llbracket\widehat{\bm{\mathrm{I\!E}}_{I}[t]}\rrbracket
=\bm{\mathrm{I\!E}}_{I}[t]+\zeta\ell
\int_{0}^{\ell}\left\lbrace\int_{0}^{\ell}(4\pi t)^{-1}
e^{-\frac{|x-y|^{2}}{2t}}dy\right\rbrace dx\nonumber\\&
=\bm{\mathrm{I\!E}}_{I}[t]+\zeta\ell\left\lbrace\frac{e^{-\frac{\ell^{2}}{2t}}\left(2^{3/2}\pi^{1/2}
(1-e^{\ell^{2}/2t})\right)}{2^{5/2}\pi^{3/2}}
+\frac{2\pi\ell erf\left(\frac{\ell}{(2t)^{1/2}}\right)e^{\frac{\ell^{2}}{2t}}}{t^{1/2}}
\right\rbrace
\end{align}
\end{lem}
\begin{proof}
From Thm 7.2, the volatility follows for $p=2$ giving
\begin{align}
&\mathbb{E}\llbracket|\widehat{u(x,t)}|^{2}\rrbracket \le |u(x,t)|^{2}
+\zeta v(\bm{\mathbf{Q}})\int_{\bm{\mathbf{Q}}}(4\pi t)^{-n}
e^{-\frac{|x-y|^{2}}{2t}}d\mu_{n}(y)\nonumber\\&
\equiv |u(x,t)|^{2} + (\|h(x-\bullet,t)\|_{L_{2}({\mathbf{Q}}}))^{1/2}
\end{align}
(7.70) then follows from the definition of the energy integral. The estimate then follows from performing the incomplete Gaussian integrals.
\begin{align}
&\mathbb{E}\llbracket\widehat{\bm{\mathrm{I\!E}}_{I}[t]}\rrbracket
=\bm{\mathrm{I\!E}}_{I}[t]+\zeta\ell
\int_{0}^{\ell}\left\lbrace\int_{0}^{\ell}(4\pi t)^{-1}
e^{-\frac{|x-y|^{2}}{2t}}dy\right\rbrace dx\nonumber\\&
=\bm{\mathrm{I\!E}}_{I}[t]+\zeta\ell\int_{0}^{\ell}\left\lbrace
\frac{erf(\frac{x}{(2t)^{1/2}})-erf(\frac{x-\ell}{(2t)^{1/2}})}{2^{5/2}(\pi t)^{1/2}}
\right\rbrace dx
\end{align}
Performing the second integration then gives (7.70). The extra term then vanishes as
$t\rightarrow\infty$.
\end{proof}
\begin{defn}
Given a stochastic process $\widehat{X(t)}$ for the stochastic time evolution of a
physical system, subject say to noise or random perturbations, then the p-moment 'Lyupunov
characteristic exponent' is defined as
\begin{align}
\bm{\mathrm{I\!L}}_{p}[\widehat{X(t)}]=\lim_{t\uparrow\infty}\frac{1}{t}\log
\mathbb{E}\llbracket |\widehat{X(t)}|^{p}\rrbracket
\end{align}
For $p=2$,
\begin{align}
\bm{\mathrm{I\!L}}_{2}[\widehat{X(t)}]=\lim_{t\uparrow\infty}\frac{1}{t}\log
\mathbb{E}\llbracket |\widehat{X(t)}|^{2}\rrbracket
\end{align}
Then
\begin{enumerate}
\item The stochastic system is stable or 'weakly intermittent' if $\bm{\mathrm{I\!L}}_{2}[\widehat{X(t)}]\le 0$
as $t\rightarrow\infty$ and the volatility decays with time.
\item The stochastic system is unstable if $\bm{\mathrm{I\!L}}_{2}[\widehat{X(t)}]> 0$ and the volatility grows with time.
as $t\rightarrow\infty$.
\item The stochastic system is 'superstable' if $\bm{\mathrm{I\!L}}_{2}[\widehat{X(t)}]
=-\infty$.
\end{enumerate}
\end{defn}
From the smoothing properties of the heat equation, it is expected that
$\bm{\mathrm{I\!L}}_{2}[\widehat{u(x,t)}]\le 0 $.
\begin{prop}
Let $\widehat{u(x,t)}$ solve the heat equation on a domain ${\mathbf{Q}}$ with random initial Cauchy data $\widehat{u(x,0)}=u(x,0)+\mathscr{J}(x)$, and $u(x,t)$ solves the
heat equation with initial data $\widehat{u(x,0)}=u(x,0)$ with
$\lim_{t\uparrow\infty}|u(x,t)|^{2}=0$. Then the LCE estimate is
\begin{align}
\bm{\mathrm{I\!L}}_{2}[\widehat{u(x,t)}]=\lim_{t\uparrow\infty}\frac{1}{t}
\mathbb{E}\llbracket |\widehat{u(x,t)}|^{2}
\rrbracket\le \lim_{t\uparrow\infty}\frac{1}{t}\log \left(|u(x,t)|^{2}+
\frac{\zeta v({\mathbf{Q}})}{|8\pi t|^{n/2}}\right)=0
\end{align}
Hence, the perturbed solution of the heat equation is stable and decays to zero. The smoothing properties of the heat equation dissipate the effects of the initial data randomness.
\end{prop}
\begin{proof}
From previous results, the volatility/variance is estimated as
\begin{align}
&\mathbb{E}\llbracket|\widehat{u(x,t)}|^{2}\rrbracket =
|u(x,t)|^{2}+\mathbb{E}\left\llbracket\left|\int_{{\mathbf{Q}}}h(x-y,t)
\otimes\mathscr{J}(y)d\mu_{n}(y)\right|^{2}\right\rrbracket\nonumber\\&
|u(x,t)|^{2}+\zeta v({\mathbf{Q}})
\left(\int_{{\mathbf{Q}}}|h(x-y,t)|^{2}d^{n}\right)\nonumber\\&
=|u(x,t)|^{2}+\zeta v({\mathbf{Q}})h(x-x,2t)
=|u(x,t)|^{2}+\frac{\zeta v({\mathbf{Q}})}{|8\pi t|^{n/2}|}
\end{align}
where the semi-group property has been used. The result (6.78) the follows.
\end{proof}
\subsection{Cole-Hopf transform for random initial data}
The Cole-Hopf transform was discussed in Section 4, and can potentially be extended to consider solutions of quasi-linear parabolic PDEs with random Gaussian field initial data.
\begin{prop}
Let $a>0$ and $b\in{\mathbf{R}}^{+}$ and consider the Cauchy problem for the following quasi-linear parabolic PDE
\begin{align}
&\left(\frac{\partial}{\partial t}-a{\Delta}\right)\widehat{\psi(x,t)}+B|{\nabla} \widehat{\psi(x,t)}|^{2}=0,~x\in{\mathbf{R}}^{n},t>0\\&
\widehat{\psi(x,0)}=\phi(x)+\mathscr{J}(x)=\widehat{\phi(x)},~x\in\bm{\mathrm{R}}^{n},t=0
\end{align}
where $\mathscr{J}(x)$ is the usual GRSF. Using the Cole-Hopf transform
\begin{align}
&\widehat{u(x,t)}=\exp\left(-\left(\frac{b}{a}\right)\right)\widehat{\psi(x,t)}\\&
\widehat{u(x,0)}=\exp\left(-\left(\frac{b}{a}\right)\right)\widehat{\psi(x,0)}\equiv \exp\left(-\left(\frac{b}{a}\right)\right)\widehat{\phi(x)}
\end{align}
then $\widehat{u(x,t)}$ immediately satisfies the global stochastic Cauchy initial value problem for the linear heat equation (with 'conductance' $A$) such that
\begin{align}
&\left(\frac{\partial}{\partial t}-a{\Delta}\right)\widehat{u(x,t)}=0,~x\in\bm{\mathrm{R}}^{n},t>0\\&
\widehat{u(0,t)}=\exp\left(-\left(\frac{b}{a}\right)\right)
\widehat{\phi(x)},~x\in\bm{\mathrm{R}}^{n},t=0
\end{align}
This has the stochastic convolution integral solution
\begin{align}
&\widehat{u(x,t)}=\frac{1}{(4\pi a t)^{n/2}}\int_{\mathrm{R}^{n}}\exp\left(-\frac{|x-y|^{2}}{4 a t}\right)\exp\left(-\left(\frac{b}{a}\right)\widehat{\phi(x)}\right)d\mu_{n}(y)\\&
=\frac{1}{(4\pi a t)^{n/2}}\int_{\mathrm{R}^{n}}\exp\left(-\frac{|x-y|^{2}}{4 a t}\right)\exp\left(-\left(\frac{b}{a}\right)\phi(x)\right)\exp\left(-
\left(\frac{b}{a}\right)\otimes \mathscr{J}(x)\right)d\mu_{n}(y)
\end{align}
It follows from (7.82) that the solution to the CIVP for the quasi-linear parabolic system (7.80) is then
\begin{align}
\widehat{\psi(x,t)}&=-\bigg(\frac{a}{b}\bigg)\log\bigg((4\pi a t)^{-n/2}\int_{\bm{\mathrm{R}}^{n}}\exp\bigg(\frac{|x-y|^{2}}{4 a t}\bigg)
\nonumber\\&
\exp\bigg(-\bigg(\frac{b}{a}\bigg)\phi(x)\bigg)
\exp\bigg(-\bigg(\frac{b}{a}\bigg)\otimes\mathscr{J}(x)\bigg)d\mu_{n}(y)\bigg)
\end{align}
The expectation is
\begin{align}
\mathbb{E}\llbracket\widehat{\psi(x,t)}\rrbracket&=-\bigg(\frac{a}{b}\bigg)
\mathbb{E}\bigg\llbracket\log\bigg((4\pi a t)^{-n/2}\int_{\mathrm{R}^{n}}\exp\bigg(\frac{|x-y|^{2}}{4 a t}\bigg)\nonumber\\&\exp\bigg(-\bigg(\frac{b}{a}\bigg)\phi(x)\bigg)
\exp\bigg(-\bigg(\frac{b}{a}\bigg)\otimes\mathscr{J}(x)\bigg)
d\mu_{n}(y)\bigg)\bigg
\rrbracket
\end{align}
and the $p-$moments are
\begin{align}
&\mathbb{E}\llbracket|\widehat{\psi(x,t)}|^{p}\rrbracket
=\bigg(-\bigg(\frac{b}{a}\bigg)\bigg)^{p}
\mathbb{E}\bigg\llbracket\bigg|\log\bigg({(4\pi a t)^{-n/2}}\nonumber\\&\int_{\mathrm{R}^{n}}\exp\bigg(\frac{|x-y|^{2}}{4 a t}\bigg)\exp\bigg(-\bigg(\frac{b}{a}\bigg)\phi(x)\bigg)
\exp\bigg(-\bigg(\frac{b}{a}\bigg)\mathscr{J}(x)
\bigg)d\mu_{n}(y)\bigg)\bigg|^{p}\bigg\rrbracket
\end{align}
\end{prop}
It should also be possible to find solutions and moments for the Burgers equation with random initial data; for example, in modelling turbulence.
\section{Thermal equilibrium with random boundary conditions}
The final scenario considered is the steady state or equilibrium heat problem
on a ball ${\mathbf{B}}_{R}(0)\subset{\mathbf{R}^{n}}$, discussed in Section 4.1.
However, now the boundary conditions are random or randomly perturbed so that $ {\Delta}\widehat{u(x)}=0$ for $x\in {\mathbf{B}}_{R}(0)$, with $\widehat{u(x)}= u(x)+\mathscr{J}(x)$, where $\mathscr{J}(x)$ is the usual GRSF. One can then estimate the volatility or variance $\mathbb{E}\llbracket|\widehat{u(x)}|^{2}\rrbracket$ within the ball or hypersphere, from the solution $\widehat{u(x)}$ within the ball.
\begin{thm}
Let $\subset\bm{\mathrm{R}}^{n}$ with boundary
$\partial{\mathbf{B}}_{R}(0)$. Let $u\in C^{2}({\mathbf{B}}_{R}(0))$ and let
$\mathscr{J}(x)$ be a Gaussian random scalar field (GRSF) with $\mathbb{E}\llbracket
\mathscr{J}(x)\rrbracket=0$ and $\mathbb{E}\llbracket
\mathscr{J}(x)\otimes \mathscr{J}(x)\rrbracket=\zeta$. The stochastic Dirichlet boundary value problem
\begin{align}
&{\Delta}\widehat{u(x)}=0,~x\in {\mathbf{B}}_{R}(0)\\&
\widehat{u(x)}=\phi(x)+\mathscr{J}(x),~x\in\partial {\mathbf{B}}_{R}(0)
\end{align}
then has the solution
\begin{align}
&\widehat{u(x)}=
\frac{R^{2}-\|x\|^{2}}{area(\partial{\mathbf{B}}_{1}(0))R}
\int_{\partial {\mathbf{B}}_{R}(0)}\frac{\phi(y)d^{n-1}y}{\|x-y\|^{n}}\nonumber\\&+
\frac{R^{2}-\|x\|^{2}}{area(\partial{\mathbf{B}}_{1}(0))R}
\int_{\partial\bm{\mathbf{B}}_{R}(0)}\frac{\mathscr{J}(y)d^{n-1}y}{\|x-y\|^{n}},~~x\in
{\mathbf{B}}_{R}(0),y\in\partial {\mathbf{B}}_{R}(0)
\end{align}
In terms of the Poisson kernel ${\mathrm{P}}(x,y)$
\begin{align}
\widehat{u(x)}=\int_{\partial {\mathbf{B}}_{R}(0)}
{\mathrm{P}}(x,y)\phi(y)d^{n-1}y+ \int_{\partial{\mathbf{B}}_{R}(0)}{\mathrm{P}}(x,y)\otimes\mathscr{J}(y)d^{n-1}y
\end{align}
where
\begin{align}
{\mathrm{P}}(x,y)=\frac{R^{2}-\|x\|^{2}}{area(\partial{\mathbf{B}}_{1}(0))R}\frac{1}{\|x-y\|^{n}}
\equiv \frac{1}{area(\partial{\mathbf{B}}_{1}(0))R}\frac{R^{2}-\|x\|^{2}}{\|x-y\|^{n}}\equiv
\frac{1}{area(\partial{\mathbf{B}}_{1}(0))R}{\mathrm{P}}(x,y)
\end{align}
and where the stochastic convolution surface integral $\int_{\partial{\mathbf{B}}_{R}(0)}
{\mathrm{P}}(x,y)\otimes\mathscr{J}(y)d^{n-1}y $ can be shown to exist and can be defined in terms of a Riemann sum (Appendix A.)
\end{thm}
\begin{proof}
If $\widehat{u(x)}$ solves the randomly perturbed Laplace equation (7.1) then
\begin{align}
{\Delta}_{x}\widehat{u(x)}=\int_{\partial\bm{\mathbf{B}}_{R}(0)}
{\Delta}_{x}{\mathrm{P}}(x,y)\psi(y)d^{n-1}y+
\int_{\partial{\mathbf{B}}_{R}(0)}{\Delta}_{x}{\mathrm{P}}(x,y)\otimes
\mathscr{J}(y)d^{n-1}y=0
\end{align}
so it is sufficient to prove that the (smooth) Poisson kernel is harmonic within the ball so that
\begin{align}
{\Delta}_{x}{\mathrm{P}}(x,y)=\frac{1}{area(\partial{\mathbf{B}}_{1}(0))R}
{\Delta}_{x}\left(\frac{R^{2}-\|x\|^{2}}{\|x-y\|^{n}}\right)
\equiv\frac{1}{area(\partial{\mathbf{B}}_{1}(0))R}
{\Delta}_{x}{\mathrm{P}}(x,y)=0
\end{align}
Set
\begin{align}
{\mathrm{P}}(x,y)=V(x)W(x,y)= \frac{R^{2}-\|x\|^{2}}{\|x-y\|^{n}}
\end{align}
where $V=R^{2}-\|x\|^{2}\equiv \|y\|^{2}-\|x\|^{2}$ and $W=\|x-y\|^{-n}$. It is then sufficient to show that ${\Delta}_{x}(VW)=0$. This requires the basic result from vector calculus that
\begin{align}
{\Delta}(VW)=V{\Delta}W+W {\Delta} V + 2{\nabla} V.{\nabla} U
\end{align}
The following basic results are also required
\begin{align}
&{\Delta} V={\nabla}.{\nabla} V\equiv div {\nabla}V\\&
{\nabla}(f(V))=f^{\prime}(V) {\nabla}V\\&
{\nabla}.(V\vec{F})= {\nabla} V.\vec{F}+V {\nabla}.\vec{F}
\end{align}
\begin{align}
{\nabla}V=-2\|x\|,~~{\Delta }V=-2n
\end{align}
Using (2.11)
\begin{align}
{\nabla}W=-n\|x-y\|^{-n-1}{\nabla}\|x-y\|
=n\|x-y\|^{-n-1}\frac{(x-y)}{\|x-y\|}=-n\|x-y\|^{-n-2}(x-y)
\end{align}
The Laplacian is then ${\Delta}_{x}W$ so that
\begin{align}
&{\Delta}_{x}W=-n div(\|x-y\|^{-n-2}(x-y))=-n(-n-2)\|x-y\|^{-n-3} \frac{(x-y)^{2}}{\|x-y\|}-n^{2}\|x-y\|^{-n-2}\nonumber\\&
=2n\|x-y\|^{-n-2}
\end{align}
Combining these using (7.9) then gives
\begin{align}
&{\Delta}(VW)=v{\Delta}W+W{\Delta} V + 2{\nabla} V.{\nabla} U\nonumber\\&
=2n(\|y\|^{2}-\|x\|^{2})\|x-y\|^{-n-2}-2n\|x-y\|^{n}+4n\|x\|\|x-y\|^{-n-2}(x-y)
\end{align}
Multiplying out by $\|x-y\|^{n+2}$ gives
\begin{align}
&\|x-y\|^{n+1}{\Delta}(VW)=
2n(\|y\|^{2}-\|x\|^{2})-2n\|x-y\|^{2}+4n\|x\|(x-y)\nonumber\\&
=2n\|y\|^{2}-2n\|x\|^{2}-2n(\|x\|^{2}-2\|xy\|+\|y\|^{2})+4n\|x\|^{2}-4n\|xy\|\nonumber\\&
=2n\|y\|^{2}-2n\|x\|^{2}-2n\|x\|^{2}+4n\|xy\|-2n\|y\|^{2}+4n\|x\|^{2}-4n\|xy\|=0
\end{align}
and so $\bm{\mathrm{P}}(x,y)$ is harmonic and $\widehat{u(x)}$ is the solution to the deterministic problem.
\end{proof}
\begin{cor}
The inhomogenous stochastic Dirichlet boundary value problem
\begin{align}
&{\Delta}\widehat{u(x)}=f(x),~x\in{\mathbf{B}}_{R}(0)\\&
\widehat{u(x)}=u(x)+\mathscr{J}(x),~x\in\partial {\mathbf{B}}_{R}(0)
\end{align}
then has the solution
\begin{align}
&\widehat{u(x)}=\int_{{\mathbf{B}}_{R}(0)}g(x-y)f(y)d\mu_{n}(y)+\frac{R^{2}-\|x\|^{2}}{area(\partial{\mathbf{B}}_{1}(0))R}
\int_{\partial{\mathbf{B}}_{R}(0)}\frac{\psi(y)d\mu_{n-1}(y)}{\|x-y\|^{n}}\nonumber\\&+
\frac{R^{2}-\|x\|^{2}}{area(\partial {\mathbf{B}}_{1}(0))R}
\int_{\partial{\mathbf{B}}_{R}(0)}\frac{\mathscr{J}(y)d\mu_{n-1}(y)}{\|x-y\|^{n}},~~x\in
{\mathbf{B}}_{R}(0),y\in\partial {\mathbf{B}}_{R}(0)
\end{align}
In terms of the Poisson kernel ${\mathrm{P}}(x,y)$
\begin{align}
\widehat{u(x)}=\int_{{\mathbf{B}}_{R}(0)}g(x-y)f(y)d\mu_{n}(y)
+\int_{\partial{\mathbf{B}}_{R}(0)}
{\mathrm{P}}(x,y)\psi(y)d\mu_{n-1}(y)+ \int_{\partial{\mathbf{B}}_{R}(0)}
{\mathrm{P}}(x,y)\otimes\mathscr{J}(y)d\mu_{n-1}(y)
\end{align}
so that
\begin{align}
{\Delta}_{x}\widehat{u(x)}&=\int_{{\mathbf{B}}_{R}(0)}g(x,y)f(y)d\mu_{n}(y)
+\int_{\partial{\mathbf{B}}_{R}(0)}
{\Delta}_{x}{\mathrm{P}}(x,y)\psi(y)d\mu_{n-1}(y)\nonumber\\&+ \int_{\partial{\mathbf{B}}_{R}(0)}
{\Delta}_{x}{\mathrm{P}}(x,y)\otimes\mathscr{J}(y)d\mu_{n-1}(y)\nonumber\\&
=\int_{{\mathbf{B}}_{R}(0)}{\Delta}_{x}\left(\frac{1}{\|x-y\|}\right)f(y)d\mu_{n}(y)
=-\int_{{\mathbf{B}}_{R}(0)}\delta^{3}(x-y)f(y)d\mu_{n}(y)=f(x)
\end{align}
since ${\Delta}_{x}\bm{\mathrm{P}}(x,y)=0$.
\end{cor}
\begin{thm}
Let $\mathbf{B}_{R}(0)\subset\bm{\mathrm{R}}^{3}$ be an Euclidean ball of radius R and centre $x=(0,0,0$ with surface $\partial\mathbf{B}_{R}(0)$. The steady state of equilibrium heat problem on $\mathbf{B}_{R}(0)$ with randomly perturbed Dirichlet boundary conditions is the perturbed Laplace equation is ${\Delta}\widehat{u(x)}=0, ~x\in \mathbf{B}_{R}(0)$ and $\widehat{\psi(x)}=\psi(x)+\mathscr{J}(x)=\psi+\mathscr{J}(x)$, where $\mathscr{J}(x)$ is the regulated GRSF with $\mathbb{E}\llbracket\mathscr{J}(x)\otimes\mathscr{J}(x)\rrbracket=\zeta$.
If $\widehat{\psi(x)}$ solves this problem then the volatility at $x=(0,0,\alpha)$ for $\alpha\in[0,R]$ is estimated as
\begin{align}
&\mathbb{E}\llbracket\widehat{|u(a)|^{2}}\rrbracket\le \frac{1}{2}(\zeta+\psi^{2})R^{2}
(R^{2}-\alpha^{2})^{2}\left\lbrace \frac{1}{4\alpha R(-2\alpha R+R^{2}+\alpha^{2})}-\frac{1}{4\alpha R(2\alpha R+R^{2}+\alpha^{2})}\right\rbrace\\&
\equiv\frac{(\zeta+\psi^{2})R(R^{2}-\alpha^{2})^{2}}{8\alpha}\left\lbrace \frac{1}{(-2\alpha R+R^{2}+\alpha^{2})} -\frac{1}{(2\alpha R+R^{2}+\alpha^{2})}\right\rbrace\\&
\end{align}
For a unit ball $\bm{\mathbf{B}}_{1}(0)$ with R=1
\begin{align}
&\mathbb{E}\llbracket\widehat{|\psi(\alpha)|^{2}}\rrbracket\le \frac{1}{2}(\zeta+\psi^{2})
(1-\alpha^{2})^{2}\left\lbrace\frac{1}{4\alpha(-2\alpha+1+\alpha^{2})}-\frac{1}{4\alpha(2\alpha+1+\alpha^{2})}\right\rbrace\\&
\equiv\frac{(\zeta+\psi^{2})
(1-\alpha^{2})^{2}}{8\alpha}\left\lbrace \frac{1}{(-2\alpha+1+\alpha^{2})}-\frac{1}{(2\alpha+1+\alpha^{2})}
\right\rbrace\\&
\end{align}
This is a minimum at $\alpha=0$, the center of the ball, and increases as $a\rightarrow R$  or $\alpha\rightarrow 1$. Hence, the volatility decays as we move away from the random boundary conditions as $\alpha\rightarrow 0$.
\end{thm}
\begin{proof}
From the solution (7.20), the volatility is estimated as
\begin{align}
\mathbb{E}\llbracket|\widehat{\psi(x)}|^{2}\rrbracket& =
\underbrace{\left(\int_{\bm{\mathbf{B}}_{R}(0)}\frac{\psi(R^{2}-\|x\|^{2})}{4\pi R}
\frac{d^{2}y}{\|x-y\|^{3}}\right)^{2}}_{apply~Cauchy-Schwarz}+
\underbrace{\mathbb{E}\left\llbracket\left(\int_{\partial\bm{\mathbf{B}}_{R}(0)}\frac{(R^{2}-\|x\|^{2})}{4\pi R}
\frac{\mathscr{J}(y)d^{2}y}{\|x-y\|^{3}}\right)^{2}\right\rrbracket}_{apply~Cauchy-Schwarz}\nonumber\\&
\le \left(\frac{\psi^{2}(R^{2}-\|x\|^{2})^{2}}{|4\pi R|^{2}}\int_{\partial{\mathbf{B}}_{R}(0)}\frac{d^{2}y}{\|x-y\|^{6}}\right)
\left(\int_{\partial{\mathbf{B}}_{R}(0)}\psi^{2}d^{2}y\right)\nonumber\\&
+\left(\frac{\psi^{2}(R^{2}-\|x\|^{2})^{2}}{|4\pi R|^{2}}\int_{{\mathbf{B}}_{R}(0)}\frac{d^{2}y}{\|x-y\|^{6}}\right)
\left(\int_{\partial\bm{\mathbf{B}}_{R}(0)}
\mathbb{E}\llbracket\mathscr{J}(x)\otimes\mathscr{J}(x)\rrbracket d^{2}y\right)\nonumber\\&
=(\psi^{2}+\zeta)Area(\partial{\mathbf{B}}_{R}(0))\left(\frac{(R^{2}-\|x\|^{2})^{2}}{|4\pi R|^{2}}\int_{{\mathbf{B}}_{R}(0)}\frac{d^{2}y}{\|x-y\|^{6}}\right)\nonumber\\&
\equiv (\psi^{2}+\zeta)4\pi R^{2}\left(\frac{(R^{2}-\|x\|^{2})^{2}}{|4\pi R|^{2}}\int_{{\mathbf{B}}_{R}(0)}\frac{d^{2}y}{\|x-y\|^{6}}\right)\nonumber\\&
=(\psi^{2}+\zeta)\left(\frac{(R^{2}-\|x\|^{2})^{2}}{4\pi}
\int_{{\mathbf{B}}_{R}(0)}\frac{d^{2}y}{\|x-y\|^{6}}\right)
\end{align}
Now choosing $x=(0,0,\alpha)$ gives
\begin{align}
\mathbb{E}\llbracket|\widehat{u(\alpha)}|^{2}\rrbracket\le&
\frac{(\zeta+\psi^{2})(R^{2}-|\alpha|^{2})^{2}}{4\pi}
\int_{0}^{2\pi}\int_{-\pi}^{\pi}\frac{R^{2}\sin\theta d\theta
d\varphi}{(R^{2}-2\alpha R\cos\theta+\alpha^{2})^{3}}\nonumber\\&
\equiv\frac {1}{2}(\zeta+\psi^{2})R^{2}(R^{2}-|\alpha|^{2})^{2}\int_{-1}^{1}\frac{d\xi}
{(R^{2}-2\alpha R\xi+\alpha^{2})^{3}}
\end{align}
Performing the integral then gives (8.23).
\end{proof}
\clearpage
\appendix
\renewcommand{\theequation}{\Alph{section}.\arabic{equation}}
\section{Properties of Gaussian Random Scalar Fields on $\bm{R}^{n}$}
This appendix provides basic definitions of existence, properties, correlations, statistics, derivatives and integrals of time-independent Gaussian random scalar fields (GRSFs) on $\bm{\mathrm{R}}^{n}$. Classical random fields correspond naturally to structures, and properties of systems. It will be sufficient to briefly establish the following:
\begin{enumerate}
\item The existence of random fields on $\bm{\mathrm{R}}^{n}$.
\item Statistics, moments, covariances and correlations.
\item Unique properties of Gaussian random fields. (GRVFs.)
\item Sample path continuity, differentiability and the existence of the derivatives of a GRSF.
\item Stochastic integration.
\end{enumerate}
\begin{defn}
Let $(\Omega,\mathcal{F},\bm{\mathrm{I\!P}})$ be a probability space. Within the probability triplet, $(\Omega,\mathcal{F})$ is a\emph{measurable space}, where $\mathcal{F}$ is the $\sigma$-algebra (or Borel field) that should be interpreted as being comprised of all reasonable subsets of the state space $\Omega$. Then $\bm{\mathrm{I\!P}}$ is a function such that $\bm{\mathrm{I\!P}}:\mathcal{F}\rightarrow [0,1]$, so that for all $A\in\mathcal{F}$, there is an associated probability $\bm{\mathrm{I\!P}}(A)$. The measure is a probability measure when $\bm{\mathrm{I\!P}}(\Omega)=1$. The probability space obeys the Kolmogorov axioms such that $\bm{\mathrm{I\!P}}(\Omega)=1$ and $0\le\bm{\mathrm{I\!P}}(A_{i})\le 1$ for all sets $A_{i}\in \mathcal{F}$. And if $A_{i}\bigcap A_{j}=\empty$, then  $\bm{\mathrm{I\!P}}(\bigcup_{i=1}^{\infty}A_{i})=\sum_{i=1}^{\infty}\bm{\mathrm{I\!P}}(A_{i})$. Let $x^{i}\subset{\mathbf{Q}}\subset\bm{\mathrm{R}}^{n}$ be Euclidean coordinates and let $(\Omega,{\mathcal{F}},\bm{\mathrm{I\!P}})$ be a probability space. Let $\widehat{\mathscr{J}}(x;\omega)$ be a random scalar function that depends on the coordinates $x\subset{\mathbf{Q}}\subset\bm{\mathrm{R}}^{n}$ and also $\omega\in\Omega$. Given any pair $(x,\omega)$ there exists maps $M:\bm{\mathrm{R}}^{n}\times\Omega\rightarrow\bm{\mathrm{R}}$ such that $ M:(\omega,x)\longrightarrow\mathscr{J}(x);\omega)$, so that
${\mathscr{J}(x,\omega)}$ is a random variable or field on ${\mathbf{Q}}\subset\bm{\mathrm{R}}^{n}$ with respect to the probability space $(\Omega,\mathcal{F},\bm{\mathrm{I\!P}})$. A random field is then essentially a family of random variables $\lbrace\mathscr{J}(x;\omega)\rbrace$ defined with respect to the space $(\Omega,\mathcal{F},\bm{\mathrm{I\!P}})$ and $\bm{\mathrm{R}}^{n}$.
\end{defn}
The fields can also include a time variable $t\in\bm{\mathrm{R}}^{+}$ so that given any triplet
$(x,t,\omega)$ there is a mapping $f:\bm{\mathrm{R}}\times\bm{\mathrm{R}}^{n}
\times\Omega\rightarrow\bm{\mathrm{R}}$ such that
$f:(x,t,\omega)\longrightarrow \mathscr{J}(x,t;\omega)$. However, it will be sufficient to consider fields that vary randomly in space only. The expected value of the random field with respect to the space $(\Omega,\bm{\mathcal{F}},\bm{\mathrm{I\!P}})$ is defined as follows
\begin{defn}
Given the random scalar field $\bm{\mathscr{J}(x;\omega)}$, then if $ \int_{\Omega}{\mathscr{J}(x;\omega)}d\bm{\mathrm{I\!P}}(\omega)<\infty$, the stochastic expectation of $\widehat{\mathscr{J}}(x;\omega)$ is
\begin{equation}
\mathbb{E}\big\llbracket\mathscr{J}(x;\omega)\big\rrbracket=
\int_{\Omega}\mathscr{J}(x;\omega)d\bm{\mathrm{I\!P}}(\omega)
\end{equation}
\end{defn}
\begin{defn}
Let $(\Omega,\mathbb{F},\bm{\mathrm{I\!P}})$ be a probability space, then an $L_{p}(\Omega,\mathcal{F},\bm{\mathrm{I\!P}})$ space or an $L_{p}$-space for $p\ge 1$ is a linear normed space of random scalar fields that satisfies the conditions
\begin{equation}
\mathbb{E}\bigg\llbracket|\mathscr{J}(x;\omega)|^{p}\bigg\rrbracket
=\int_{\Omega}|{\mathscr{J}(x;\omega)}|^{p}d\bm{\mathrm{I\!P}}(\omega)<\infty
\end{equation}
and the corresponding norm can be defined as $\|\widehat{\mathscr{J}}(x)\|
=(\mathbb{E}\llbracket|\widehat{\mathscr{J}}(x;\omega)|^{p}\rrbracket)^{1/p}$ with the usual Euclidean norm for $p=2$.
\end{defn}
\begin{defn}
Let $(x,y)\in{\mathbf{Q}}\subset{\mathbf{R}}^{n}$ and $(\omega,\eta)\in{\mathbf{Q}}$. The covariance of the field at these points is then formally
\begin{equation}
\mathbb{E}\bigg\llbracket{\mathscr{J}(x,\omega)}\otimes {\mathscr{J}}(y,\xi)
\bigg\rrbracket=\iint_{\Omega}{\mathscr{J}(x,\omega)}\otimes
{\mathscr{J}(y,\xi)}d\bm{\mathrm{I\!P}}(\omega)d\bm{\mathrm{I\!P}}(\xi)
\end{equation}
\end{defn}
The second-order correlations, moments and covariances are now
\begin{defn}
Let $x^{i},y^{i}\in {\mathbf{Q}}\subset\mathbf{R}^{3}$ and let $\omega,\xi\in\Omega$.
The expectations of mean values of the fields ${\mathscr{J}}_{i}(x,\omega)$
and ${\mathscr{J}}(y,\xi)$ are
\begin{align}
&\bm{\mathcal{M}}_{1}(x)=\mathbb{E}\left\llbracket{\mathscr{J}(x)}\right\rrbracket =\int_{\Omega}{\mathscr{J}(x,\omega)}d\bm{\mathrm{I\!P}}(\omega)\\&
\bm{\mathcal{M}}_{1}(y)=\mathbb{E}\left\llbracket{\mathscr{J}(y)}\right\rrbracket=
 \int_{\Omega}{\mathscr{J}(y,\xi)}d\bm{\mathrm{I\!P}}(\xi)
\end{align}
The 2nd-order moment or expectation is
\begin{equation}
\bm{\mathcal{M}}_{1}(x)\llbracket{\mathscr{J}(x)}\otimes {\mathscr{J}(y)}\rrbracket=
\iint_{\Omega}{\mathscr{J}}(x,\omega)\otimes{\mathscr{J}}(y,\xi)
d\bm{\mathrm{I\!P}}(\omega)d\bm{\mathrm{I\!P}}(\xi)
\end{equation}
The covariance is then
\begin{align}
&\mathbb{K}\llbracket{\mathscr{J}(x)}\otimes {\mathscr{J}(y)}\rrbracket=\mathbb{E}\bigg\llbracket({\mathscr{J}(x)}-
{\mathscr{J}(y)})\otimes ({\mathscr{J}(x)}-
{\mathscr{J}(y)})\bigg\rrbracket\nonumber\\&
=\iint_{\Omega}\bigg(\bm{\mathscr{J}}(x;\omega)-
\bm{\mathscr{J}}(y,\xi)\bigg))\otimes\bigg((\bm{\mathscr{J}}(x;\omega)
-\bm{\mathscr{J}}(y,\xi)\bigg)
d\bm{\mathrm{I\!P}}(\omega)d\bm{\mathrm{I\!P}}(\xi)
\end{align}
so that
\begin{align}
&\mathbb{K}\llbracket{\mathscr{J}(x)}\otimes {\mathscr{J}(y)}\rrbracket=\mathbb{E}\bigg\llbracket\widehat{\mathscr{J}(x)}
\otimes\widehat{\mathscr{J}(y)}
\bigg\rrbracket-\mathbb{E}\big\llbracket{\mathscr{J}}(x)
\big\rrbracket\mathbb{E}\big\llbracket{\mathscr{J}}(y)\big\rrbracket\nonumber\\&
=\mathbb{E}\bm{\mathscr{E}}\bigg\llbracket\widehat{\mathscr{J}(x)}
\otimes\widehat{\mathscr{J}(y)}
\bigg\rrbracket-\bm{\mathcal{M}}_{1}(x)\bm{\mathcal{M}}_{1}(y)
\end{align}
\end{defn}
\begin{defn}
Given a set of fields $\widehat{\mathscr{J}}_{i_{1}}(x_{1}),...,\widehat{\mathscr{J}}
_{i_{n}}(x_{n})$ at points $(x_{1}...x_{n})\in{\mathbf{Q}}$ then the $n^{th}$-order moments and cumulants are
\begin{equation}
\mathbb{E}\llbracket\widehat{\mathscr{J}}_{i_{1}}(x_{1})\otimes...
\otimes\widehat{\mathscr{J}}_{i_{n}}(x_{n})\rrbracket
\end{equation}
\begin{equation}
\mathbb{K}\llbracket\widehat{\mathscr{J}}_{j_{1}}(x_{1})\otimes...
\otimes\widehat{\mathscr{J}}_{j_{n}}(x_{n})\rrbracket
\end{equation}
where at second order
\begin{equation}
\mathbb{K}\llbracket\mathscr{J}(x)\otimes\mathscr{J}(y)\rrbracket
=\mathbb{E}\llbracket \mathscr{J}(x)\otimes\mathscr{J}(y)\rrbracket-
\bm{\mathcal{M}}_{1}(x)\bm{\mathcal{M}}_{1}(y)
\end{equation}
\end{defn}
A very important class of random fields are the Gaussian random vector fields (GRVFS) which are characterized only by the first and second moments. The GRVFS can also be isotropic, homogenous and stationary. The details will be made more precise but the advantages of GRVFs are briefly enumerated.
\begin{enumerate}
\item GRVFS have convenient mathematical properties which generally simplify calculations;indeed, many results can only be evaluated using Gaussian fields.
\item A GRVF can be classified purely by its first and second moments and high-order moments and cumulants can be ignored.
\item Gaussian fields accurately describe many natural stochastic processes including Brownian motion.
\item A large superposition of non-Gaussian fields can approach a Gaussian field (Feller 1966.)
\end{enumerate}
For this paper, the following definitions are sufficient for isotropic GRVFS. (More details can be found in $\bm{[24]}$)
\begin{defn}
Any GRVF has normal probability density functions. The following always hold:
\begin{enumerate}
\item The first moment vanishes so that
\begin{equation}
\bm{\mathcal{M}}_{1}(x)=\mathbb{E}\big\llbracket\mathscr{J}(x;\omega)\big\rrbracket
=\int_{\Omega}\mathscr{J}(x;\omega)d\bm{\mathrm{I\!P}}(\omega)=0
\end{equation}
\item The covariance then reduces to the '2-point' function
\begin{align}
\mathbb{K}\llbracket{\mathscr{J}(x)}\otimes {\mathscr{J}(y)}\rrbracket&=\mathbb{E}\big\llbracket\widehat{\mathscr{J}(x)}
\otimes\widehat{\mathscr{J}(y)}
\big\rrbracket-\mathbb{E}\big\llbracket{\mathscr{J}}(x)
\big\rrbracket\mathbb{E}\big\llbracket{\mathscr{J}}(y)\big\rrbracket\nonumber\\&
=\mathbb{E}\big\llbracket\widehat{\mathscr{J}(x)}
\otimes\widehat{\mathscr{J}(y)}
\big\rrbracket=\zeta J(x,y;\ell)
\end{align}
where $\ell$ is a correlation length such that $J(x,y;\ell)\rightarrow 0$ for $\|x-y\|\gg \ell.$
\item The GRSF is regulated at all points $x$ if $J(x,x;\ell)<\infty1$. Here, $J(x,x;\ell)=1$. For a white-in-space random field or noise
\begin{align}
{\mathbb{E}}\big\llbracket\widehat{\mathscr{J}}(x)
\otimes\widehat{\mathscr{J}}(y)
\big\rrbracket=\zeta\delta^{n}(x-y)
\end{align}
which blows up at $x=y$. An example of a regulated 2-point correlation for a GRSF would be a colored-in-space noise of the form
\begin{align}
\mathbb{E}\big\llbracket\widehat{\mathscr{J}}(x)
\otimes\widehat{\mathscr{J}}(y)
\big\rrbracket=\zeta J(x,y;\ell)=\zeta \exp\left(-\frac{\|x-y\|}{\ell}\right)
\end{align}
\end{enumerate}
\end{defn}
\begin{defn}
The GRVF is isotropic if $J(x,y;\ell)=J(y,x;\ell)$ depends only on the separation $\|x-y\|$ and is stationary if $J(x+\delta x,y+\delta y)=J(x,y;\ell)$. Hence, the 2-point function or Greens function is translationally and rotationally invariant $\mathrm{R}^{n}$.
\end{defn}
Having established the basic properties of GRSFS on $\bm{\mathrm{R}}^{n}$, and as a prerequisite to differentiation and integration of random vector fields, the geometric properties are briefly considered in relation to continuity of the sample paths.
\begin{defn}
Let $\lbrace x_{\alpha}\subset{\mathbf{Q}}\subset {\mathbf{R}}^{n}$ be a sequence of points in a domain $\mathbf{Q}$ such that $x_{\alpha}\rightarrow x$ as $\alpha\rightarrow\infty$ or $\lim_{\alpha\rightarrow \infty}\|x_{\alpha}-x\|=0$. Let $\mathscr{J}(x_{\alpha})$ and $\mathscr{J}(x)$ be random vector fields at these points. (Not necessarily Gaussian.) Then:
\begin{enumerate}
\item The random field $\mathscr{J}(x)$ has \emph{continuous sample paths} with unit probability in ${\mathbf{Q}}$ if for every sequence $\lbrace x_{\alpha}\rbrace\in {\mathbf{Q}}$ with $\lim_{\alpha\rightarrow \infty}\|x_{\alpha}-x\|=0$ and $\bm{\mathrm{I\!P}}[\lim_{\alpha\rightarrow\infty}|\mathscr{J}(x_{\alpha})-
\mathscr{J}(x)|=0;x\in {\mathbf{Q}}]=1 $
Continuous sample paths with unit probability, means that with a probability of one, there are no discontinuities within the entire domain ${\mathbf{Q}}$. This condition is also called sample path continuity.
\item The random vector field $\mathscr{J}(x)$ has almost surely continuous sample paths if for every sequence $\lbrace x_{\alpha}\rbrace\in\mathbf{Q}$ with $\lim_{\alpha\rightarrow \infty}\|x_{\alpha}-x\|=0$ so that $
\bm{\mathrm{I\!P}}[\lim_{\alpha\rightarrow\infty}|\mathscr{J}(x_{\alpha})-
\mathscr{J}(x)|=0]=1$. However, there may be $x\in {\mathbf{Q}}$ for which the condition does not hold giving discontinuities within ${\mathbf{Q}}$.
\item The random vector field $\mathscr{J}(x)$ is mean-square continuous in ${\mathbf{Q}}$ if for every sequence $\llbracket x_{\alpha}\rrbracket\in{\mathbf{Q}}$ with $\lim_{\alpha\rightarrow \infty}\|x_{\alpha}-x\|=0$ so that
$\mathbb{E}\llbracket\lim_{\alpha\rightarrow\infty} |\mathscr{J}(x_{\alpha})-\mathscr{J}(x)|^{2}
\rrbracket=0 $
\end{enumerate}
\end{defn}
More details can be found in $\bm{[24]}$.
\begin{defn}
A GRSF $\mathscr{J}(x)$ is almost surely continuous at $x\in {\mathbf{Q}}\in {\mathbf{R}}^{n}$ if $ \mathscr{J}(x+\beta)\longrightarrow\mathscr{J}(x) $ as $\beta\rightarrow 0$
\end{defn}
When this holds for all $x\in{\mathbf{Q}}$ then this is known as 'sample function continuity'. The following result gives the sufficient condition for continuous sample paths within a domain ${\mathbf{Q}}$
\begin{lem}
Let $\mathscr{J}(x)$ be a RVF on ${\mathbf{Q}}\subset{\mathbf{R}}^{n}$. Then if for some $C>0$ and $\lambda>0$ with $\eta>\lambda$
\begin{equation}
\mathbb{E}\llbracket\big|\mathscr{J}(x+\beta)-\mathscr{J}(x)
\big|^{\lambda}\rrbracket \le\frac{C|\zeta|2n}{|\ln|\beta||^{1+\eta}}
\end{equation}
If $\mathscr{J}(x)$ is a Gaussian random field then given some $C>0$ and some $\epsilon>0$
\begin{equation}
\mathbb{E}\llbracket|\mathscr{J}(x+\beta)-\mathscr{J}(x)|^{2}\rrbracket
\le \frac{C}{\ln|\beta|^{1+\epsilon}}
\end{equation}
\end{lem}
\subsection{Differentiability and existence of derivatives}
Let $\mathscr{J}(x)$ be a GRSF, existing for all $x\in {\mathbf{Q}}\subset\bm{\mathrm{R}}^{n}$, with covariance
\begin{equation}
\mathbb{E}({\mathscr{J}(x)}\otimes{\mathscr{J}(y)})=
\mathbb{E}\llbracket{\mathscr{J}(x)}\otimes{\mathscr{J}(y)}\rrbracket
-\bm{\mathcal{M}}_{1}(x)\bm{\mathcal{M}}_{2}(y)
\end{equation}
where $\bm{\mathcal{M}}_{1}(x)=\mathbb{E}\llbracket{\mathscr{J}(x)}\rrbracket=0$.
\begin{defn}
Let $\nabla\mathscr{J}(x)$ denote the gradient of a GRF. Let $\mu_{i}$ be a unit vector along the $i^{th}$ direction such that $\mu_{1}=(1,0,0,0...), \mu_{2}=(0,1,0,0,...)$ etc. with $\|\mu_{i}\|=1$. A GRF is differentiable in the mean square sense(MSS) if
\begin{equation}
\nabla \mathscr{J}(x)=\lim_{|\mathcal{R}|\uparrow 0}|
\frac{\mathscr{J}(x+|\mathcal{R}|\mu_{i})-\mathscr{J}(x)}{|\mathcal{R}|}
\end{equation}
which implies that
\begin{equation}
\lim_{\mathcal{R}\uparrow 0}\mathbb{E}\bigg\llbracket\bigg|
\frac{\mathscr{J}(x+|\mathcal{R}|\mu_{i})-\mathscr{J}(x)}{|\mathcal{R}|}
-\nabla\mathscr{J}(x)|^{2}
\bigg|\bigg\rrbracket=0
\end{equation}
The Laplacian or second-order partial derivative at $x$ is defined as
\begin{equation}
\nabla_{(x)}\nabla_{(x)}{\mathscr{J}(x)}
=\lim_{\mathcal{R},\mathcal{J}\uparrow 0}\frac{1}{|\mathcal{R}\mathcal{S}|}\bigg[{\mathscr{J}(x+|\mathcal{R}|\mu_{i}
+|\mathcal{J}|\mu_{j})}-
{\mathscr{J}(x+|\mathcal{R}|\mu_{i})}-{\mathscr{J}(x+|\mathcal{S}|\mu_{i})}+{\mathscr{J}(x)}
\bigg]
\end{equation}
\end{defn}
An alternative definition, which is probably more useful, is given as follows
\begin{lem}
A GRF is differentiable in the mean-square sense(MSS)with respect to $\mu_{i}$ if
\begin{equation}
\lim_{\mathcal{R},\mathcal{S}\uparrow 0}\bigg\llbracket\bigg|\frac{\mathscr{J}(x+|\mathcal{R}|\mu_{i})-\mathscr{J}(x)}{|\mathcal{R}|}-
\frac{\mathscr{J}(x+|\mathcal{S}|\mu_{i})-\mathscr{J}(x)}{|\mathcal{S}|}\bigg|^{2}
\bigg\rrbracket=0
\end{equation}
Then a GRF is differentiable if
\begin{enumerate}
\item $\bm{M}_{1}(x)$ is differentiable.(For GRFs $\bm{M}_{1}(x)=0$ and this condition can be relaxed.)
\item The following covariance exists and is finite for all points $x=y$.
\begin{align}
&\mathbb{K}\llbracket\nabla_{(x)}{\mathscr{J}(x)}
\otimes\nabla_{(y)}{\mathscr{J}(y)})\rrbracket
=\nabla_{(x)}\nabla_{(y)}\mathbb{K}\llbracket
({\mathscr{J}(x)}\otimes{\mathscr{J}(y)})\rrbracket\nonumber\\&
=\lim_{\mathcal{R},\mathcal{S}\uparrow 0}\frac{1}{|\mathcal{RS}|}\bigg[\mathbb{K}\llbracket
\mathscr{J}(x+\mathcal{R}\mu_{i})\otimes{\mathscr{J}(y+\mathcal{S}\mu_{i})}\rrbracket
-f\mathbb{K}\llbracket({\mathscr{J}(x)}\otimes
{\mathscr{J}(y+\mathcal{S}\mu_{i}))}\rrbracket\nonumber\\&
-\mathbb{E}\llbracket({\mathscr{J}(x+\mathcal{R}\mu_{i})}\otimes{\mathscr{J}(y))}
+\mathbb{K}\llbracket{\mathscr{J}(x)}\otimes\widehat{\mathscr{J})}
\rrbracket\nonumber\\&=\lim_{\mathcal{R},\mathcal{R}\uparrow 0}\frac{1}{|\mathcal{RS}|}\bigg[\mathbb{E}\mathscr{J}(x+\mathcal{R}\mu_{i})
\otimes \mathscr{J}(y+\mathcal{S}\mu_{i})\bigg\rrbracket-\mathbb{E}({\mathscr{J}(x)}\otimes
{\mathscr{J}(y+\mathcal{S}\mu_{i})})\nonumber\\&-
\mathbb{E}\llbracket{\mathscr{J}(x+\mathcal{R}\mu_{i})}\otimes {\mathscr{J}(y))}\rrbracket +\mathbb{E}\llbracket{(\mathscr{J}(x)}\otimes{\mathscr{J}(y))}\rrbracket<\infty
\end{align}
\end{enumerate}
\end{lem}
\begin{rem}
Mean square differentiability at all $x\in{\mathbf{Q}}$ implies mean square continuity at all $x\in{\mathbf{Q}}$ and this is also consistent with Holder and Kolmogorov continuity conditions.
\begin{align}
\lim_{a\uparrow 0}\mathbb{E}\bigg\llbracket{\mathscr{J}(x+a\xi_{i})}-
{\mathscr{J}(x)|}^{2}\bigg|\bigg\rrbracket=\lim_{a\uparrow 0}(|\mathcal{R}|^{2})
\lim_{a\uparrow 0}\frac{1}{|\mathcal{R}|^{2}}\mathbb{E}\bigg\llbracket{\mathscr{J}(x+\mathcal{R}\xi_{i})}
-{\mathscr{J}(x)|}^{2}\bigg|\bigg\rrbracket=0
\end{align}
Note that this (and all definitions of the derivative) also requires a regulated random field or noise. Again, if the GRF is white-in-space noise then $\mathbb{E}\llbracket{\mathscr{J}(x)}\otimes{\mathscr{J}(y)}\rrbracket=\delta^{n}(x-y)$ and the derivative cannot be defined.
\end{rem}
\begin{lem}
The condition for the existence of the derivative $\nabla_{i}{\mathscr{J}(x)}$
is also tantamount to the differentiability of the covariance function so that
\begin{align}
\mathbb{E}\llbracket\big[\nabla_{(x)}\mathscr{J}(x)\otimes\nabla_{y}{\mathscr{J}(y)}\rrbracket
=\nabla_{(x)}\nabla_{(y)}\mathbb{K}\llbracket{\mathscr{J}(x)}\otimes {\mathscr{J}(y)}\rrbracket
\end{align}
exist for all $x=y$. Again, the GRF is differentiable only if it is not a white noise. Note that for $x=y$
\begin{align}
&\bm{\mathcal{K}}\llbracket\big[\nabla_{(x)}\mathscr{J}(x)\otimes\nabla_{x}{\mathscr{J}(x)}
\rrbracket
=\nabla_{(x)}\nabla_{(x)}\bm{\mathcal{K}}\llbracket{\mathscr{J}(x)}\otimes {\mathscr{J}(x)}\rrbracket\nonumber\\&=\nabla_{(x)}\nabla_{(y)}\zeta J(x,y;\ell)=0
\end{align}
\end{lem}
\begin{lem}
The expectation is $\mathbb{E}\llbracket\nabla{\mathscr{J}(x)}\rrbracket
=\nabla\mathbb{E}\llbracket\mathscr{J}(x)\rrbracket=0$ since
\begin{align}
&\mathbb{E}\bigg\llbracket\lim_{\mathcal{R}\uparrow 0}\frac{{\mathscr{J}(x+|\mathcal{R}|\mu_{i})}-
{\mathscr{J}(x)}}{|\mathcal{R}|}\bigg\rrbracket\nonumber\\&=\lim_{\mathcal{R}\uparrow 0}\frac{1}{\mathcal{R}}\mathbb{E}\bigg\llbracket
\mathscr{J}(x+\mathcal{R}\mu_{i})\bigg\rrbracket-\mathbb{E}\bigg\llbracket{\mathscr{J}(x)}
\bigg\rrbracket=\nabla\mathbb{E}\bigg\llbracket{\mathscr{J}(x)}=0
\bigg\rrbracket
\end{align}
\end{lem}
\subsection{Stochastic Integration of GRSFs}
The integral of a GRSF is defined as the limit of a Riemann sum of the field over the partition of a domain.
\begin{prop}
Let ${\mathbf{Q}}\subset{\mathbf{R}}^{n}$ be a (closed) domain with boundary $\partial{\mathbf{Q}}$ and $x=(x_{1},...,x_{n})\subset{\mathbf{Q}}$. Let ${\mathbf{Q}}=\bigcup_{q=1}^{M}{\mathbf{Q}}_{1}$ be a partition of ${\mathbf{Q}}$ with ${\mathbf{Q}}_{q}\bigcap{\mathbf{Q}}_{q}=\varnothing$ if $p\ne q$. Let $x^{(q)}\in{\mathbf{Q}}_{q}$ for all $q=1...M$. Note $x^{(q)}\equiv (x_{1}^{(q)},...x_{n}^{(q)})$. Then $x^{(1)}\in{\mathbf{Q}}_{1},x^{(2)}\in{\mathbf{Q}}_{2},
...,x^{(M)}\in{\mathbf{Q}}$. Let $\zeta({\mathbf{Q}}_{q})$ be the volume of the partition ${\mathbf{Q}}_{q}$ so that $\zeta({\mathbf{Q}})=\sum_{q}^{M}\zeta(\bm{\mathbf{Q}})$. Similarly, if $\partial{\mathbf{Q}}$ is the surface or boundary of $\partial{\mathbf{Q}}$ then let
\begin{equation}
{\mathbf{Q}}=\bigcup_{q=1}^{M}\partial{\mathbf{Q}}_{q}
\end{equation}
be a partition of$\partial{\mathbf{Q}}$ with $\partial{\mathbf{Q}}_{q}\bigcap\partial {\mathbf{Q}}_{q}=\varnothing$ if $p\ne q$. Let $x^{(q)}\in\partial{\mathbf{Q}}_{q}$ for all $q=1...M$. Note $x^{(q)}\equiv (x_{1}^{(q)},...x_{n}^{(q)})$. Then $x^{(1)}\in\partial{\mathbf{Q}}_{1}, x^{(2)}\in\partial{\mathbf{Q}}_{2},...,x^{(M)} \in\partial{\mathbf{Q}}$. Let $\|{\mathbf{Q}}_{q}\|\equiv\mu(\partial{\mathbf{Q}}_{q})$ be the surface area of the partition $\partial{\mathbf{Q}}_{q}$ so that $\mu(\partial{\mathbf{Q}})=\bigcap_{q=1}^{M}A(\partial{\mathbf{Q}}_{q}$. The total volume and area of ${\mathbf{Q}}$ is
\begin{align}
v(\bm{\mathbf{Q}})=\sum_{q=1}^{M}({\mathbf{Q}}_{q})
=\sum_{q=1}^{M}v({\mathbf{Q}}_{q})
\end{align}
\begin{align}
area(\partial{\mathbf{Q}})=\sum_{q=1}^{M}area(\partial {\mathbf{Q}}_{q})=\sum_{q=1}^{M}area(\partial{\mathbf{Q}}_{q})
\end{align}
Given the probability triplet $(\Omega,\mathcal{F},\bm{\mathrm{I\!P}})$ then a Gaussian random field on ${\mathbf{Q}}$ for all $x\in\bm{\mathbf{Q}}$ is ${\mathscr{J}}:\omega\times{\mathbf{Q}}\rightarrow {\mathbf{R}}$ and ${\mathscr{J}(x^{q},\omega)}\in{\mathbf{Q}}_{q}$ exists for all $x^{(q)}\in {\mathbf{Q}}_{q}$ and $\omega\in\Omega$. The stochastic volume
integral and the stochastic surface integral are
\begin{align}
&\int_{{\mathbf{Q}}}\mathscr{J}{(x;\omega)}d\mu_{n}(x)=\lim_{all~v({\mathbf{Q}}_{q})\uparrow 0}\sum_{q=1}^{M}{\mathscr{J}(x^{(q)};\omega)}v({\mathbf{Q}}_{q})\\&
\int_{\partial{\mathbf{Q}}}{\mathscr{J}(x;\omega)}d\mu_{n-1}(x)
=\lim_{all~area(\partial{\mathbf{Q}}_{q})\uparrow 0}\sum_{q=1}^{M}{\mathscr{J}(x^{(q)};\omega)}area(\partial{\mathbf{Q}}_{q})
\end{align}
When a Gaussian random field is integrated, it is the limit of a linear combination of Gaussian random variables/fields so it is again Gaussian.
\end{prop}
Next, the stochastic expectations or averages are defined.
\begin{prop}
Since
\begin{equation}
\mathbb{E}\llbracket\bullet\rrbracket
=\int_{\Omega}(\bullet)d\bm{\mathrm{I\!P}}(\omega)
\end{equation}
the expectation of the volume integral is as follows.
\begin{align}
\mathbb{E}\bigg\llbracket\int_{{\mathbf{Q}}}
\mathscr{J}(x;\omega)d\mu_{n}(x)
\bigg\rrbracket
=\lim_{all~v(\bm{\mathbf{Q}}_{q})\uparrow 0}\mathbb{E}\bigg\llbracket\sum_{q=1}^{M}
\mathscr{J}(x^{(q)};\omega)v({\mathbf{Q}}_{q})\bigg\rrbracket
\end{align}
or
\begin{align}
&\mathbb{E}\bigg\llbracket\int_{{\mathbf{Q}}}{\mathscr{J}(x;\omega)}d\mu_{n}(x)\bigg\rrbracket
\equiv\int\!\!\int_{{\mathbf{Q}}}{\mathscr{J}(x;\omega)}d\mu_{n}(x) d\bm{\mathrm{I\!P}}(\omega)\\&
=\lim_{M\uparrow\infty}\lim_{all~v({\mathbf{Q}}_{q})\uparrow 0}\int_{\Omega}\sum_{q=1}^{M}
{\mathscr{J}(x^{(q)};\omega)}v({\mathbf{Q}}_{q})d\bm{\mathrm{I\!P}}(\omega)=0
\end{align}
which vanishes for GRSFs since $\mathbb{E}\bigg\llbracket
\mathscr{J}(x^{(q)})\bigg\rrbracket=0$. Similarly, for the stochastic surface integrals
\begin{align}
\mathbb{E}\left\llbracket\int_{{\mathbf{Q}}}{\mathscr{J}(x;\omega)}
d^{n-1}x\right\rrbracket
=\lim_{all~area({\mathbf{Q}}_{q})\uparrow 0}\mathbb{E}\left\llbracket\sum_{q=1}^{M}
{\mathscr{J}(x^{(q)};\omega)}A(\partial{\mathbf{Q}}_{q})\right\rrbracket
\end{align}
or
\begin{align}
&\mathbb{E}\left\llbracket\int_{\partial{\mathbf{Q}}}{\mathscr{J}(x;\omega)}d\mu_{n-1}(x)
\right\rrbracket\equiv\int_{\Omega}\int_{\partial{\mathbf{Q}}}{\mathscr{J}(x;\omega)}
d^{n-1}x d\bm{\mathrm{I\!P}}(\omega)\\&=\lim_{all~\mu(\partial{\mathbf{Q}}_{q})\uparrow 0}\int_{\Omega}\sum_{q=1}^{M}{\mathscr{J}(x^{(q)};\omega)}
area(\partial{\mathbf{Q}}_{q})d\bm{\mathrm{I\!P}}(\omega)=0
\end{align}
\end{prop}
Given an integral (or summation) over a random field or stochastic process, the Fubini theorem states that the expectation of the integral or sum over a random field is equivalent to the integral or sum of the expectation of the field.
\begin{thm}
Let $\mathscr{J}(x)$ be a random field not necessarily Gaussian, existing for all $x\in{\mathbf{Q}}$ with expectation $\mathbb{E}\llbracket \mathscr{J}(x)\rrbracket $, not necessarily zero. Then
\begin{equation}
\mathbb{E}\bigg\llbracket\int_{{\mathbf{Q}}}\mathscr{J}(x)d\mu(x)
\bigg\rrbracket\equiv \int_{{\mathbf{Q}}}\mathbb{E}\big\llbracket
\mathscr{J}(x)\big\rrbracket d\mu_{n}(x)
\end{equation}
For a set of N random fields $\mathscr{J}_{q}(x)$
\begin{equation}
\mathbb{E}\bigg\llbracket\sum_{q=1}^{N}\mathscr{J}_{q}(x)\bigg\rrbracket=
\sum_{q=1}^{N}\mathbb{E}\big\llbracket \mathscr{J}_{q}(x)\big\rrbracket
\end{equation}
\end{thm}
It is also possible to define a 'mollifier' or convolution integral.
\begin{prop}
Let $(x,y)\in {\mathbf{Q}}$ and let $k(x-y)$ be a smooth function of $(x,y)$ that will depend on the separation $\|x-y\|$. Given the GRSF ${\mathscr{J}(y)}$ define the volume and surface integral convolutions
\begin{align}
&{\mathscr{J}(x)}=k(x-y)* {\mathscr{J}(y)}\equiv\int_{{\mathbf{Q}}}k(x-y)\otimes {\mathscr{J}(y)}d\mu_{n}(y),~(x,y)
\in\bm{\mathbf{Q}}\\& {\mathscr{J}(x)}=k(x-y)* {\mathscr{J}(y)}\equiv\int_{\bm{\mathbf{Q}}}k(x-y)\otimes
{\mathscr{J}(y)}d\mu_{n-1}(y),(x,y)\in\partial{\mathbf{Q}}
\end{align}
\end{prop}
For example, if $k(x-y)$ is a Gaussian function then the random field ${\mathscr{J}(x)}$ can be 'Gaussian smoothed' on a scale $\ell$.
\begin{align}
&\mathscr{J}(x)=k(x-y)* {\mathscr{J}(y)}\equiv C\int_{{\mathbf{Q}}}\exp\left(-\frac{\|x-y\|^{2}}{\ell^{2}}\right)\otimes{\mathscr{J}(y)}
d\mu_{n}(y),
~(x,y)\in{\mathbf{Q}}\\& \mathscr{J}(x)= k(x-y) * {\mathscr{J}(y)}\equiv C\int_{{\mathbf{Q}}}\exp\left(-\frac{\|x-y\|^{2}}{\ell^{2}}\right)\otimes
{\mathscr{J}(y)}d\mu_{n-1}(y),(x,y)\in\partial{\mathbf{Q}}
\end{align}
For the heat kernel $h(x-y,t)$ the convolution give a time-dependent GRSF that solves the randomly perturbed heat equation $\square\mathscr{J}(x,t)=0$.
\begin{align}
&\mathscr{J}(x,t)=h(x-y,t)* {\mathscr{J}(y)}\equiv \int_{{\mathbf{Q}}}\frac{1}{(4\pi t)^{n/2}}\exp\left(-\frac{\|x-y\|^{2}}{4t}\right)\otimes{\mathscr{J}(y)}d\mu_{n}(y),
~(x,y)\in{\mathbf{Q}}
\end{align}
The result is extended to include GRSFs ${\mathscr{J}(x)}$.
\begin{prop}
Let ${\mathscr{J}(x)}$ be a regulated GRSF defined for all $x\in\bm{\mathcal{D}}$ such that $\mathbb{E}\llbracket{\mathscr{J}(x)}\rrbracket=0$ and
\begin{equation}
\mathbb{E}\bigg\llbracket\big|{\mathscr{J}(x)}\big|^{\ell}\bigg\rrbracket=
\frac{1}{2}[\zeta^{\ell/2}+(-1)^{\ell}\zeta^{\ell/2}]
\end{equation}
so that $\mathbb{E}\big\llbracket\big|{\mathscr{J}(x)}\big|^{\ell}\big\rrbracket=0$ for all odd p. Then
\begin{align}
\mathbb{E}\bigg\llbracket
\bigg|\int_{{\mathbf{Q}}}{\mathscr{J}(x)}d\mu(x)\bigg|^{\ell}\bigg\rrbracket&~\le
\frac{1}{2}C[\zeta^{\ell/2}+(-1)^{\ell}\zeta^{\ell/2}]\bigg|\int_{\bm{\mathbf{Q}}}d\mu(x)\bigg|^{\ell}
\nonumber\\&=\frac{1}{2}[\alpha^{\ell/2}+(-1)^{\ell}\alpha^{\ell/2}]\|\bm{\mathbf{Q}}\|^{\ell}
\end{align}
where $\|{\mathbf{Q}}\|$ is the volume of ${\mathbf{Q}}$.
\end{prop}
\begin{proof}
Follows immediately from Fubini Thm.
\end{proof}

}
\end{document}